\documentclass[a4paper,11pt]{article}

\usepackage[T1]{fontenc}
\usepackage{pgfplots}
\usepgfplotslibrary{fillbetween}

\usepackage
[
    a4paper,
    left=2.3cm,
    right=2.3cm,
    top=3cm,
    bottom=3cm,
]
{geometry}
\usepackage[hidelinks]{hyperref}

\usepackage{amsthm}
\usepackage{amsmath}
\usepackage{amsfonts}
\usepackage{amssymb}
\usepackage{mathdots}
\usepackage{stmaryrd} 
\usepackage[shortlabels]{enumitem}

\usepackage{marginnote}
\newcommand\note[1]%
                {$^\dagger$\marginnote{\footnotesize{$^\dagger${#1}}}}

\makeatletter
\DeclareFontFamily{OMX}{MnSymbolE}{}
\DeclareSymbolFont{MnLargeSymbols}{OMX}{MnSymbolE}{m}{n}
\SetSymbolFont{MnLargeSymbols}{bold}{OMX}{MnSymbolE}{b}{n}
\DeclareFontShape{OMX}{MnSymbolE}{m}{n}{
    <-6>  MnSymbolE5
   <6-7>  MnSymbolE6
   <7-8>  MnSymbolE7
   <8-9>  MnSymbolE8
   <9-10> MnSymbolE9
  <10-12> MnSymbolE10
  <12->   MnSymbolE12
}{}
\DeclareFontShape{OMX}{MnSymbolE}{b}{n}{
    <-6>  MnSymbolE-Bold5
   <6-7>  MnSymbolE-Bold6
   <7-8>  MnSymbolE-Bold7
   <8-9>  MnSymbolE-Bold8
   <9-10> MnSymbolE-Bold9
  <10-12> MnSymbolE-Bold10
  <12->   MnSymbolE-Bold12
}{}

\let\llangle\@undefined
\let\rrangle\@undefined
\DeclareMathDelimiter{\llangle}{\mathopen}%
                     {MnLargeSymbols}{'164}{MnLargeSymbols}{'164}
\DeclareMathDelimiter{\rrangle}{\mathclose}%
                     {MnLargeSymbols}{'171}{MnLargeSymbols}{'171}
\makeatother

\usepackage{times}
\usepackage{bm}
\usepackage{dutchcal}

\usepackage{tikz-cd}
\usepackage{fancybox}
\usepackage{caption}
\usepackage{subcaption}
\usepackage{graphicx}
\pgfplotsset{compat=1.14}

\usepackage[toc,page]{appendix}
\usepackage{todonotes}

\DeclareMathOperator{\Hom}{Hom}
\DeclareMathOperator{\diag}{diag}
\DeclareMathOperator{\im}{Im}
\DeclareMathOperator{\id}{Id}

\DeclareMathOperator{\pr}{pr}
\DeclareMathOperator{\SL}{SL}
\DeclareMathOperator{\SO}{SO}

\DeclareMathOperator{\SU}{SU}

\DeclareMathOperator{\wt}{wt}
\DeclareMathOperator{\hw}{hw}

\DeclareMathOperator{\Lie}{Lie}
\DeclareMathOperator{\Ad}{Ad}

\DeclareMathOperator{\UU}{U}

\DeclareMathOperator{\GWidth}{GWidth}
\DeclareMathOperator{\Vol}{Vol}

\newcommand{\n}{^{-1}}

\newcommand\lie{\mathfrak}

\renewcommand{\a}{\lie{a}}

\newcommand{\g}{\lie{g}}
\newcommand{\h}{\lie{h}}
\newcommand{\hhh}{\lie{h}}

\newcommand{\nnn}{\lie{n}}

\renewcommand{\t}{\lie{t}}
\newcommand{\ttt}{\lie{t}}
\newcommand{\kk}{\lie{k}}
\newcommand{\z}{\lie{z}}

\usepackage{theoremref}
\newtheorem{theorem}{Theorem}[section]

\newtheorem{proposition}[theorem]{Proposition}
\newtheorem{lemma}[theorem]{Lemma}
\newtheorem{corollary}[theorem]{Corollary}

\newtheorem{notation}[theorem]{Notation}

\theoremstyle{definition}
\newtheorem{definition}[theorem]{Definition}

\newtheorem{example}[theorem]{Example}

\newtheorem{remark}[theorem]{Remark}

\newcommand\bb{\mathbb}

\newcommand\N{\bb{N}}
\newcommand\Z{\bb{Z}} 
\newcommand\Q{\bb{Q}}
\newcommand\R{\bb{R}} 
\newcommand\C{\mathbb{C}}

\newcommand\T{\bb{T}}

\newcommand\ii{\mathbf{i}}

\newcommand\CC{\mathcal{C}}

\renewcommand{\geq}{\geqslant}
\renewcommand{\le}{\leqslant}
\renewcommand{\leq}{\leqslant}
\newcommand\nnfootnote[1]{%
  \begin{NoHyper}
  \renewcommand\thefootnote{}\footnote{#1}%
  \addtocounter{footnote}{-1}%
  \end{NoHyper}
}

\usepackage{titlesec}
\titleformat{\paragraph}
{\normalfont\normalsize\bfseries}{\theparagraph}{1em}{}
\titlespacing*{\paragraph}
{0pt}{3.25ex plus 1ex minus .2ex}{1.5ex plus .2ex}

\begin{document}
\title{Action-angle coordinates on coadjoint orbits and multiplicity free spaces 
from partial tropicalization}
\author{Anton Alekseev\and Benjamin Hoffman\and Jeremy Lane \and Yanpeng Li}

\newcommand{\Addresses}{{
  \bigskip
  \footnotesize

  \noindent\textsc{Section of Mathematics, University of Geneva, 2-4 rue du Li\`evre, c.p. 64, 1211 Gen\`eve 4, Switzerland}\par\nopagebreak
  \noindent\textit{E-mail address}: \texttt{Anton.Alekseev@unige.ch}

  \medskip

  \noindent\textsc{Department of Mathematics, Cornell University, 310 Malott Hall, Ithaca, NY 14853, USA}\par\nopagebreak
  \noindent\textit{E-mail address}: \texttt{bsh68@cornell.edu}
  
  \medskip
  
  \noindent\textsc{Department of Mathematics \& Statistics, McMaster University, Hamilton Hall, 1280 Main Street W, Hamilton, ON, L8S 4K1, Canada}\par\nopagebreak
  \noindent\textit{E-mail address}: \texttt{lanej5@math.mcmaster.ca}
  
  \medskip

  \noindent\textsc{Section of Mathematics, University of Geneva, 2-4 rue du Li\`evre, c.p. 64, 1211 Gen\`eve 4, Switzerland}\par\nopagebreak
  \noindent\textit{E-mail address}: \texttt{yanpeng.li@unige.ch}

}}
\date{}
\maketitle

\nnfootnote{\emph{Keywords:} Poisson-Lie groups, coadjoint orbits, symplectic geometry}

\begin{abstract}

Coadjoint orbits and multiplicity free spaces of compact Lie groups are important examples of symplectic manifolds with Hamiltonian groups actions. Constructing action-angle variables on these spaces is a challenging task. A fundamental result in the field is the Guillemin-Sternberg construction of  Gelfand-Zeitlin integrable systems for the groups $K=\UU(n), \SO(n)$. Extending these results to groups of other types is one of the goals of this paper.

Partial tropicalizations are Poisson spaces with constant Poisson bracket built using techniques of Poisson-Lie theory and the geometric crystals of Berenstein-Kazhdan. They provide a bridge between dual spaces of Lie algebras ${\rm Lie}(K)^*$ with linear Poisson brackets and polyhedral cones which parametrize the canonical bases of irreducible modules of $G=K^\mathbb{C}$. 

We generalize the construction of partial tropicalizations to allow for arbitrary cluster charts, and apply it to questions in symplectic geometry. For each regular coadjoint orbit of a compact group $K$, we construct an exhaustion by symplectic embeddings of toric domains. As a by product we arrive at a conjectured formula for Gromov width of regular coadjoint orbits. We prove similar results for multiplicity free $K$-spaces.

\end{abstract}

\tableofcontents

\section{Introduction}\label{intro}

There is a dichotomy in symplectic geometry between local and global coordinates. Whereas Darboux's theorem tells us that symplectic manifolds have no local invariants, the problem of finding large coordinate charts often relates to subtle properties of symplectic manifolds. Most famously, Gromov's non-squeezing theorem demonstrates that the volume of certain coordinate charts on a symplectic manifold may have an upper bound strictly less than the total volume of the symplectic manifold \cite{G}.

Action-angle coordinates are a type of coordinate chart on symplectic manifolds that originate from the study of commutative completely integrable systems in classical mechanics. The domains of action-angle coordinates are products of the form $U \times \T^n$, where $U$ is an open subset of $\R^n$ and $\T^n$ is a direct product of $n$ circles, $S^1 \times \dots \times S^1$. Such domains carry a canonical symplectic form,
\begin{equation}\label{equation; omega std}
	\omega_{\mathrm{std}} = \sum_{i=1}^n d \varphi_i \wedge d\lambda_i,
\end{equation}
where $\lambda_i$ are coordinates on $\mathbb{R}^n$ and $\varphi_i$ are coordinates on $\T^n$. The Liouville-Arnold theorem guarantees existence of local action-angle coordinates in a neighbourhood of  compact regular fibers of  commutative completely integrable systems \cite{arnold}. A compact symplectic toric manifold of dimension $2n$ with Delzant polytope $\triangle$ has a dense subset symplectomorphic to $(\mathring{\triangle} \times \T^n,\omega_{\mathrm{std}})$, where $\mathring{\triangle}$ denotes the interior of $\triangle$. However, there are also many interesting examples of  action-angle coordinates on dense subsets that do not arise from a toric structure, such as Gelfand-Zeitlin systems \cite{GS1}, Goldman systems on moduli spaces of flat 
connections \cite{goldman,weitsman}, bending flow systems on moduli spaces of polygons 
\cite{kapmill}, and integrable systems constructed by toric degeneration on smooth projective varieties \cite{HK,K1,K2}. 

Multiplicity free spaces are the natural non-abelian generalization of toric manifolds.
 A \emph{multiplicity free space} $(M,\omega,\mu)$ is a symplectic manifold $(M,\omega)$ equipped with a Hamiltonian action of a compact Lie group $K$ generated by an equivariant moment map $\mu\colon M \to \kk^*=\Lie(K)^*$ with the property that the non-empty symplectic reduced spaces are all zero dimensional \cite{GS4}. Compact multiplicity free spaces are classified by the quotient $M/K$, which is identified with a convex polytope, together with the principal isotropy subgroup $L$, which is a subgroup of the centralizer of a torus subgroup of $K$ \cite{knop2}. In this paper, we consider the case where $L$ is a subgroup of the maximal torus, $T$.  In this case, the range of possible dimensions of a multiplicity free space is
\begin{equation}\label{equation; dims range}
	\dim K - \dim T \leq \dim M \leq \dim K + \dim T.
\end{equation}
Unlike toric manifolds, multiplicity free spaces are not known to have action-angle coordinates on dense subsets, except for several examples. 
The first and oldest example is that of $\mathrm{SO}(3)$ acting by rotation on $S^2$, equipped with a rotation invariant area form.  The action of a maximal torus $T\leq SO(3)$ gives $S^2$ the structure of a symplectic toric manifold and the action-angle coordinates  are simply cylindrical coordinates.  
The second major example is Gelfand-Zeitlin systems, which define action-angle coordinates on dense subsets of coadjoint orbits and multiplicity free spaces of compact Lie groups of type A, B, and D \cite{GS1}.  
The final example is that of spherical varieties. If a multiplicity free space embeds equivariantly  as a projective variety and  the symplectic structure is induced by the embedding, then it is a spherical variety and dense action-angle coordinates can be constructed by toric degeneration \cite{HK}. However, most multiplicity free spaces are not spherical varieties (in fact, a compact multiplicity free space need not admit an invariant compatible complex structure \cite{W2}).

The main result of this paper is a construction of action-angle coordinates on large subsets of compact multiplicity free spaces whose principal isotropy group is contained in a maximal torus.  For any subset $C$ of a Euclidean space, let $C(\delta)$ denote the set of points in $C$ that have distance more than $\delta$ from the boundary of $C$.
\begin{theorem} \label{thm:intro_1}
Let $K$ a compact connected Lie group and $(M, \omega, \mu)$ a compact multiplicity free space for $K$ of dimension $2n$ such that the principal isotropy subgroup is contained in the maximal torus of $K$. Then, there is a convex polytope $\triangle_M\subset \R^n$ of dimension $n$ such that for all $\delta >0$, there exists a symplectic embedding 
\begin{equation}\label{equation;thm intro_1 embeddings}
	(\triangle_M(\delta)\times \T^n,\omega_{\mathrm{std}}) \hookrightarrow (M,\omega).
\end{equation}
Moreover, for all $\varepsilon>0$, there exists $\delta>0$ such that the symplectic volumes satisfy
\begin{equation*}
	{\rm Vol}(\triangle_M(\delta) \times \mathbb{T}^n, \omega_{\mathrm{std}}) >
{\rm Vol}(M, \omega) - \varepsilon.
\end{equation*}
\end{theorem}

Additionally, the embeddings~\eqref{equation;thm intro_1 embeddings} intertwine naturally defined Hamiltonian $T\times T$ actions; see Theorem~\ref{main corollary 1} for a precise statement.

In fact, we construct many polytopes $\triangle_M$ which satisfy the conclusions of Theorem \ref{thm:intro_1}. These polytopes are parameterized by several choices, including a choice of seed for the cluster algebra structure on the big double Bruhat cell in a Borel subgroup of the complexification of $K$.


One particularly important instance of Theorem~\ref{thm:intro_1} is the extremal case where $\dim M = \dim K - \dim T$. In this case $(M,\omega, \mu)$ is a coadjoint orbit, diffeomorphic to $K/T$, and $\omega$ is the canonical Kostant-Kirillov-Souriau symplectic form \cite{kos,kir,sou}.  Coadjoint orbits diffeomorphic to $K/T$, also known as regular coadjoint orbits, are parameterized by elements $\lambda$ in the interior of the positive Weyl chamber. The coadjoint orbit parameterized by $\lambda$ along with its Kostant-Kirillov-Souriau form is denoted $(\mathcal{O}_\lambda,\omega_\lambda)$.

\begin{theorem} \label{thm:intro_2}
Let $\mathcal{O}_\lambda$ be a regular coadjoint orbit of a compact connected Lie group $K$. Then, there is a convex polytope 
$\triangle_\lambda$ of dimension $n = \frac{1}{2}(\dim K - \dim T)$ such that for all $\delta >0$ there exists a symplectic embedding 
\begin{equation*}
	(\triangle_\lambda(\delta)\times \T^n,\omega_{\mathrm{std}}) \hookrightarrow (\mathcal{O}_\lambda,\omega_\lambda).
\end{equation*}
Moreover, for all $\varepsilon >0$, there exists $\delta >0$ such that 
\begin{equation*}
	{\rm Vol}(\triangle_\lambda(\delta) \times \mathbb{T}^n, \omega_{\mathrm{std}}) >
{\rm Vol}(\mathcal{O}_\lambda, \omega_\lambda) - \varepsilon.
\end{equation*}
\end{theorem}

Theorems~\ref{thm:intro_1} and~\ref{thm:intro_2}  have two limitations. First, we must assume that the principal isotropy subgroup is contained in a maximal torus.  Second, they do not yield action-angle coordinates on a dense subset. In particular, these action-angle charts do not currently have an interpretation as action-angle coordinates for a globally defined commutative integrable system. On the other hand, these theorems illustrate that there are no non-trivial obstructions to the volume of action-angle coordinates on such multiplicity free spaces. 

Returning to the problem of finding large Darboux charts on symplectic manifolds, let $B^{2n}(r)\subset \R^{2n}$ denote the open ball of radius $r>0$ equipped with the standard symplectic structure on $\R^{2n}$,
\begin{equation}\label{equation; omega std on r2n}
	\omega_{\mathrm{std}} = dx_1\wedge dy_1 + \dots + dx_n\wedge dy_n.
\end{equation}
The \emph{Gromov width} of a connected symplectic manifold $(M,\omega)$ of dimension $2n$, denoted $\GWidth(\mathcal{O}_\lambda, \omega_\lambda)$, is supremum of all cross-sectional areas $\pi r^2$ such that $(B^{2n}(r),\omega_{\rm{std}})$ embeds symplectically into $(M,\omega)$. 
Combining Theorem~\ref{thm:intro_2} with some results from \cite{FLP} regarding the geometry of the polytopes $\triangle_\lambda$ yields the following.

\begin{theorem}\label{gromov width theorem}
    Let $(\mathcal{O}_\lambda, \omega_\lambda)$ be a regular coadjoint orbit of a compact connected simple Lie group $K$. Then,
    \begin{equation}\label{gromov width formula}
        \GWidth(\mathcal{O}_\lambda, \omega_\lambda) \geq \min\{ 2\pi\langle \sqrt{-1} \lambda, \alpha^\vee \rangle \mid \alpha\in R_+ \},
    \end{equation}
    where $R_+$ is the set of positive roots of $K$ and $\alpha^\vee$ is the coroot of $\alpha \in R_+$.
\end{theorem}

It follows from the upper bounds of \cite{CC} that~\eqref{gromov width formula} is an equality. This was already known to be an equality in several cases.
The case where $K$ is type A, B, or D and $\lambda$ is arbitrary was proved using Gelfand-Zeitlin systems in \cite{P}.  The case where $K$ is arbitrary type and $\lambda$ is a positive scalar multiple of a dominant integral weight was proved using toric degenerations by \cite{FLP} (see \cite[Section 2]{FLP} for a detailed survey of earlier results).  In fact, \cite{CC} provides upper bounds for Gromov width of non-regular orbits as well. The methods developed in this paper do not  apply to those orbits. 

\begin{remark} The authors BH and JL will present an alternate approach to the problem of constructing action-angle coordinates on multiplicity free spaces in a forthcoming paper~\cite{HL}. Their approach uses a gradient-Hamiltonian flows inside the total space of toric degenerations of the base affine space $G\sslash N$, $G=K^\C$, which were constructed in \cite{C}. These produce integrable systems with dense action-angle coordinates on the symplectic implosion of the cotangent bundle of $K$ (cf. \cite{GJS}). These integrable systems may then be combined with symplectic contraction (cf. \cite{HMM}) to construct completely integrable systems on arbitrary compact multiplicity Hamiltonian $K$-manifolds that have dense action-angle coordinates (no assumption about the principal isotropy subgroup is needed). The polytopes obtained by this construction are similar to the polytopes obtained by \cite{AB} in the case of spherical varieties. We expect these results will be sufficient to prove tight lower bounds for the Gromov width of the remaining non-regular coadjoint orbits (modulo  combinatorial results about string polytopes). It would be interesting to better understand the relationship between the toric degeneration and Ginzburg-Weinstein approaches to this problem.
\end{remark}

\subsection{Methods}
Our results are obtained by a new method of independent value that combines new findings with previous results from~\cite{ABHL1, ABHL2,AHLL}. The main idea is that to each coadjoint orbit $\mathcal{O}_\lambda$ one can associate a family of symplectic spaces $D_{\exp(s\lambda)}$ called dressing orbits. The dressing orbits are symplectomorphic to $\mathcal{O}_\lambda$ for all values of the parameter $s \in \R^\times$. For $s$ small, $D_{\exp(s\lambda)}$ resembles of $\mathcal{O}_\lambda$, and there is a natural way to include $\mathcal{O}_\lambda$ in the family at $s=0$. For $s\ll 0$ large, there are coordinates on $D_{\exp(s\lambda)}$ coming from cluster algebra theory which make its symplectic structure (exponentially) close to the constant one. Using this, one may construct action-angle coordinates on $D_{\exp(s\lambda)}$, and hence on $\mathcal{O}_\lambda$, which exhaust the symplectic volume as $s\to -\infty$. We call these charts \emph{big}, in the sense that by picking $s\ll 0$, their volume may be made arbitrarily close to the volume of $\mathcal{O}_\lambda$. 

In what follows we briefly outline the main points used in our construction. While we focus here on coadjoint orbits, in the main body we extend this treatment to allow for arbitrary compact regular multiplicity-free spaces. We make use of some standard Lie-theoretic notation, which is introduced in detail in Section~\ref{lie theory section}.

\subsubsection{Dual Poisson-Lie groups}
The space $\mathfrak{k}^*=\Lie(K)^*$ carries a canonical linear Poisson structure $\pi_{\mathfrak{k}^*}$ defined by formula
$$
\{ f_\xi, f_\eta\} = f_{[\xi, \eta]},
$$
where $\xi, \eta \in \mathfrak{k}$ and $f_\xi(x)=\langle x, \xi\rangle$ for $x \in \mathfrak{k}^*$.
Coadjoint orbits are symplectic leaves of $\pi_{\mathfrak{k}^*}$.

Consider the complex Lie group $G=K^{\mathbb{C}}$. It admits the Iwasawa decomposition $G=AN_-K$, where $N_-$ is the maximal nilpotent subgroup and $A=\exp(\sqrt{-1} \mathfrak{t})$. Denote $K^*=AN_-$, and observe that $\Lie(K^*)$ may be identified with $\kk^*=\Lie(K)^*$ under the pairing
$$
\langle x, \xi \rangle := 2 {\rm Im} \, (x, \xi)_\mathfrak{g},\quad x\in \Lie(K^*),~\xi\in \kk.
$$
Here $(\cdot, \cdot)_\mathfrak{g}$ is a fixed nondegenerate, $\Ad$-invariant bilinear form on $\mathfrak{g}=\mathfrak{k}^\mathbb{C}$. 

The group $K^*$ is naturally isomorphic to the symmetric space $G/K$. Hence, it carries a $K$-action known as the dressing action. Furthermore, it has a unique Poisson structure $\pi_{K^*}$ given by the Lu formula
$$
\pi_{K^*}^\sharp(\langle \theta^R, \xi \rangle) = -\underline{\xi},\qquad \forall\xi\in\kk,
$$
where $\theta^R$ is the right-invariant Maurer-Cartan form on $K^*$ and $\underline{\xi}$ is the fundamental vector field of $\xi$. The group $(K^*,\pi_{K^*})$ is the \emph{dual Poisson-Lie group to $K$}. Symplectic leaves of the Poisson structure $\pi_{K^*}$ are the orbits of the $K$-action on $K^*$, which are called \emph{dressing orbits}.

See Section~\ref{PLgroupSection} for further exposition of these notions.

\subsubsection{The Ginzburg-Weinstein isomorphism and scaling}
There is an intimate link between the Poisson spaces $\mathfrak{k}^*$ and $K^*$ given by the Ginzburg-Weinstein Theorem:
\begin{theorem} \cite{GW} \label{thm:intro_gw} 
There is a Poisson isomorphism $\gamma : \mathfrak{k}^* \to K^*$ which restricts to the exponential map on 
$\mathfrak{t}^* \cong \sqrt{-1}\mathfrak{t}$.
\end{theorem}
The choice of the Ginzburg-Weinstein isomorphism $\gamma$ is non unique. Indeed, it can be pre-composed with any Poisson isomorphism of $\mathfrak{k}^*$ preserving $\mathfrak{t}^*$ or composed with any Poisson isomorphism of $K^*$ preserving $A$.

The Poisson structure $\pi_{\mathfrak{k}^*}$ on $\mathfrak{k}^*$ is linear, and so under scaling transformations $A_s: x \to sx$ it also scales linearly: $(A_s)_* \pi_{\mathfrak{k}^*} = s \pi_{\mathfrak{k}^*}$. Therefore, the map
$$
\gamma_s= \gamma \circ A_s: x \mapsto \gamma(sx)
$$
is a Poisson isomorphism between $(\mathfrak{k}^*, \pi_{\mathfrak{k}^*})$ and $(K^*, s\pi_{K^*})$. In particular, for each $s\ne 0$ the map $\gamma_s$ restricts to a symplectomorphism between the coadjoint orbit $\mathcal{O}_\lambda\subset \kk^*$ through an element $\lambda$ of the positive Weyl chamber $\ttt^*_+$, and the dressing orbit $D_{\exp(s\lambda)}\subset K^*$ through $\gamma_s(\lambda)\in \gamma_s(\ttt^*_+)$. Therefore, we may study $D_{\exp(s\lambda)}$ for arbitrary $s\in \mathbb{R}^\times$ instead of $\mathcal{O}_\lambda$.

A construction of $\gamma_s$ is described in Section~\ref{GWsection} and Remark~\ref{GWremark}. 

\subsubsection{Coordinates on \texorpdfstring{$K^*$}{K*} and cluster algebras}
The dual Poisson-Lie group $K^*$ is a real subgroup of the Borel subgroup of $G$: 
$$
K^* = AN_- \subset B_-.
$$
Recall that the double Bruhat cell 
$$
G^{w_0,e}=Bw_0B \cap B_- \subset B_-
$$ 
is dense in $B_-$. Double Bruhat cells carry cluster structures which provide an infinite set of distinguished coordinate systems called \emph{toric charts}, which are parametrized by cluster seeds. These coordinate systems restrict to $K^* \cap G^{w_0, e}$ and define dense charts on $K^*$. When restricted to $K^*$, these coordinate systems combine complex and real coordinates. Toric charts on $G^{w_0,e}$ and $K^*$ are described in Sections~\ref{section;chartsDBC} and~\ref{ChartsonKstarSection}, respectively. 

For the sake of exposition, we focus here on a finite number of toric charts known as factorization coordinates. They are parametrized by reduced expressions in the Weyl group of $G$ representing the longest element $w_0$. 
The results below apply to all (twisted) cluster charts, see Sections~\ref{PositivitySection} and~\ref{PTSection} for a more detailed discussion.\footnote{Strictly speaking, the following example is related to a certain twisted cluster chart on $G^{w_0,e}$ by a monomial change of coordinates given by~\cite[Theorem~1.9]{FZ}.}

Assume $K$ is semisimple of rank $r$. For a simple root $\alpha_i$ of $\mathfrak{g}$, denote by $\alpha^\vee_i\in \mathfrak{h}$ the corresponding coroot and by $e_i\in \mathfrak{n}$ and $f_i\in \mathfrak{n}_-$ the corresponding root vectors. Let
$$
x_i(t)=\exp(t f_i)
$$
be a 1-parameter subgroup of $N_-$. For a fixed reduced word 
$w_0=s_{i_1} \cdots s_{i_m}$, introduce the following map
\begin{align*}
&\qquad\qquad\qquad\qquad\qquad L_s \colon \R^{m+r} \times (S^1)^m  \to K^* \\
&(\lambda_{-r},\dots,\lambda_{-1},\lambda_1,\dots,\lambda_m,e^{\sqrt{-1}\varphi_1},\dots,e^{\sqrt{-1}\varphi_m})  \\ &\qquad\qquad\qquad \qquad \mapsto \exp\left(\frac{s}{2}\sum_{i=1}^r \lambda_{-i} \alpha^\vee_i\right)x_{i_1}(e^{\frac{s\lambda_1}{2}+\sqrt{-1}\varphi_1})\cdots x_{i_m}(e^{\frac{s\lambda_m}{2}+\sqrt{-1}\varphi_m}).
\end{align*}
The map $L_s$ is injective for $s\in \R^\times$, and its image is open and dense in $K^*$. We regard the $\lambda_i$, $\varphi_j$ as coordinates on $K^*$ which depend on the parameter $s$. The map $L_s$ is given in general in Definition~\ref{definition;detropicalization}.

In the main body of the paper, we consider a wider class of toric charts, so-called \emph{cluster charts} and \emph{twisted cluster charts}; see Definition~\ref{definition;unreduced cluster chart}.
Coordinates in twisted cluster charts are similar to factorization parameters described above while coordinates in cluster charts are generalizations of minors of a matrix. Furthermore, along with the double Bruhat cell $G^{w_0, e}$ we need to consider its Langlands dual $(G^\vee)^{w_0, e}$ and its (twisted) cluster charts.


We denote 
\[
\pi_s =  (L_s)\n_*(s \pi_{K^*}) = (L_s\n \circ \gamma_s)_* \pi_{\mathfrak{k}^*}
\]
wherever $L_s\n$ is defined.
While $L_s\n$ is exponentially contracting, the map $\gamma_s$ is exponentially expanding. 
For $s\ll 0$, their composition possesses extraordinary properties.

\subsubsection{Partial tropicalization}
Our strategy is to consider the limit $s\to -\infty$. A priori, such a limit does not make sense since the Poisson structure $s\pi_{K^*}$ linearly diverges. However, after the coordinate change $L_s$, the Poisson structure $\pi_s$ does admit a limit which is described by the following theorem. It makes use of a map
\[
\hw^{PT}\colon \R^{m+r}\times (S^1)^m \to \ttt^*_+
\]
to the positive Weyl chamber called the \emph{highest weight map}. This map, and its origin, are described in Sections~\ref{PTstructureSection} and~\ref{section;geometriccrystals}, respectively. The following summarizes the main results of Section~\ref{PTSection}.
\begin{theorem}  \label{thm:intro_cone}
There is a unique open polyhedral cone $\mathcal{C} \subset \mathbb{R}^N$ such that:
\begin{enumerate}
\item The limit
\[
\lim_{s\to -\infty}\left( \pi_s |_{\mathcal{C} \times (S^1)^m}\right)
\]
exists and is a constant Poisson structure $\pi_{-\infty}$ on $\mathcal{C} \times (S^1)^m$.
\item The symplectic leaves of $(\mathcal{C}\times (S^1)^m,\pi_{-\infty})$ are the fibers of $\hw^{PT} \colon \mathcal{C}\times (S^1)^m \to \ttt^*_+$.
\item For each regular weight $\lambda\in \mathring{\ttt}^*_+$, the symplectic volume of $(\hw^{PT})\n(\lambda)$ is equal to the symplectic volume of $\mathcal{O}_\lambda$.
\item For a given $\lambda\in \ttt^*_+$, the fiber $(\hw^{PT})\n(\lambda)\cap \mathcal{C}\times(S^1)^m$ is of the form 
\[
\mathring{\Delta}_\lambda \times (S^1)^m,
\]
 where $\Delta_\lambda$ is a convex polytope.
\item After a linear change of variables, $\pi_{-\infty}$ acquires the form
\[
\pi_{-\infty}= \sum_{i=1}^m \frac{\partial}{\partial \lambda_i} \wedge \frac{\partial}{\partial \varphi_i}.
\]
\end{enumerate}
\end{theorem}

The Poisson manifold $(\mathcal{C} \times (S^1)^m, \pi_{-\infty})$ is called the \emph{partial tropicalization of the dual Poisson-Lie group $K^*$}. A priori, its definition depends on the coordinate map $L_s$ corresponding to a fixed toric chart. However, one can show that, for all (twisted) cluster charts, the construction gives rise to isomorphic Poisson spaces\footnote{Strictly speaking, this isomorphism is defined only off a union of hyperplanes in $\R^m$; see Definition~\ref{definition;PTcoordtrans} and Theorem~\ref{theorem;PTpoissonisomorphism}.}.

Cluster coordinates on $K^*$ give a good control of the Poisson bracket $\pi_s$, and they allow to establish parts 1, 2, 4 and 5 of Theorem~\ref{thm:intro_cone}. Twisted cluster coordinates on $G^{w_0,e}$ and on its Langlands dual $(G^\vee)^{w_0,e}$ give a tool for volume estimates and are needed in proving part 3 of Theorem~\ref{thm:intro_cone}.

%

\begin{remark}
Theorem \ref{thm:intro_cone} hints that if the map $L_s\n \circ \gamma_s$ had a limit for $s\to -\infty$, the limit map $\gamma_\infty\colon \mathcal{O}_\lambda \to (\hw^{PT})\n(\lambda)$ would give rise to densely defined action-angle coordinates on the coadjoint orbit. This scenario is realized in the case of $K=\UU(n)$ \cite{ALL}. However, for arbitrary compact $K$ it is not known how to choose the Ginzburg-Weinstein isomorphism $\gamma$ giving rise to a convergent map $\lim_{s\to-\infty} (L_s\n \circ \gamma_s)$.
\end{remark}

\begin{remark}
The real Poisson-Lie group $K^*$ has a complex form, denoted $G^*$. The group $K^*$ is a connected component of the fixed point set of an anti-holomorphic involution of $G^*$. There have been several recent approaches to constructing systems of coordinates $z_1,\dots,z_n$ on $G^*$ whose Poisson brackets are log canonical~\cite{GSV19, SS, Shen}, meaning that $\{z_i,z_j\}_{G^*} = c_{i,j} z_iz_j$ for some $c_{i,j}\in \C$. 

One might expect that, by changing to polar coordinates $z_i=e^{\lambda_i+\sqrt{-1}\varphi_i}$, the restriction of the functions $\lambda_i,\varphi_i$ to $K^*$ would give a Darboux chart on $K^*$. However, this is not the case. 
Assume we are given a $\C$-valued log canonical coordinate system on a real manifold. This doesn't necessarily give rise to an $\R$-valued log canonical coordinate system or to a Darboux chart. Indeed, let $z, z': M \to \C$ such that $\{z, z'\} = c zz'.$ If $z=e^{\lambda + i \varphi}, w=e^{\lambda' + i \varphi'}$, then
\begin{align*}
\{\lambda,\lambda'\} - \{ \varphi, \varphi'\} & = \Re(c) \\
\{\lambda, \varphi'\} + \{\varphi, \lambda'\} & = \Im(c)
\end{align*}
but we cannot conclude whether or not the Poisson bracket is constant in coordinates $\lambda,\lambda',\varphi,\varphi'$.

In our particular example, if $K=\SU(2)$, then 
\[
G^* = \left\{\left(\begin{pmatrix} a & b \\ 0 & a\n \end{pmatrix},\begin{pmatrix} a\n & 0 \\ c & a \end{pmatrix}\right) \colon a\in \C^\times, b,c\in \C\right\} \subset \SL(2)\times \SL(2).
\]
Then $K^*$ is the set of points where $a\in \R_{>0}$ and $b=-\overline{c}$. A system of log canonical coordinates constructed in \cite{GSV19} is $z_1=a, z_2=ab, z_3 =-bc+a^2+a^{-2}$. To extract a Darboux coordinate system, following the strategy of~\cite[Section~6]{ABHL1}, for instance, one would need the Poisson brackets $\{z_1,z_2\}_{K^*}$,~$\{z_1,\overline{z}_2\}_{K^*}$, and $\{z_2,\overline{z}_2\}_{K^*}$ to be log canonical on $K^*$. However, 
\[
\overline{z}_2 = \frac{z_1^2 z_3-z_1^4-1}{z_2}
\]
on $K^*$, and the Poisson bracket $\{z_2,\overline{z}_2\}$ is not log canonical. We will explore the connection between the partial tropicalization of $K^*$ and log canonical coordinate systems on $G^*$ in future work.
\end{remark}

\subsubsection{The cone \texorpdfstring{$\mathcal{C}$}{C} and the Berenstein-Kazhdan potential}
In Theorem \ref{thm:intro_cone}, the cone $\mathcal{C}$ appears as an unexpected and mysterious element of the construction. In fact, it is a version of a well known object in the representation theory of reductive complex Lie groups. For a dominant integral weight $\lambda\in P_+$, let $V_\lambda$ be the irreducible $G$-module with high weight $\lambda$. Then, the canonical basis of $\C[G]^N\cong \oplus_{\lambda\in P_+} V_\lambda$ can be parametrized as the integral points of the closure $\overline{\mathcal{C}}$.

%

The defining inequalities of $\overline{\mathcal{C}}$ are determined by the tropicalization of a distinguished function 
\[
\Phi_{BK}\colon G^{w_0,e} \to \mathbb{C},
\]
called the \emph{Berenstein-Kazhdan potential}. On $K^*$, the potential $\Phi_{BK}$ can be written as a Laurent polynomial in the factorization coordinates $e^{\frac{s\lambda_i}{2}+\sqrt{-1}\varphi_i}$ introduced in the previous section. The coefficients of this Laurent polynomial are positive integers. The function $\Phi_{BK}$ is a main technical tool in the theory that follows; see for instance Corollary~\ref{corollary;domination}.

It is a surprising fact that the $s\to -\infty$ behavior of all Poisson brackets on $K^*$ is determined (dominated) by the behavior of a single function $\Phi_{BK}$. At the moment, this phenomenon has no conceptual explanation. The function $\Phi_{BK}$ was introduced in the context of geometric crystals, which are briefly described in Section~\ref{section;geometriccrystals}. 

\subsubsection{Proof sketch for coadjoint orbits}

%
 
We sketch a proof of Theorem~\ref{thm:intro_2}.
For $s< 0$, we have a Ginzburg-Weinstein isomorphism 
$$
\gamma_s: \mathcal{O}_\lambda \to D_{\exp(s\lambda)}.
$$
The image $L_s\n(D_{\exp(s\lambda)}) \subset \mathbb{R}^{m+r} \times (S^1)^m$ has the following property: the symplectic volume of
\[
(L_s\n(D_{\exp(s\lambda)}))\cap(\mathcal{C} \times (S^1)^m)
\]
can be made arbitrarily close to the volume of $L_s\n(D_{\exp(s\lambda)})$ by fixing $s\ll 0$.
 This manifold is very near to the symplectic leaf $(\hw^{PT})\n(\lambda)$ of $\pi_{-\infty}$, and their symplectic forms are also very close to each other (in the appropriate sense). This allows one to use the Moser argument and to construct an embedding
$$
(\Delta_\lambda(\delta)  \times (S^1)^m,\pi_{-\infty}) \hookrightarrow (L_s\n(D_{\exp(s\lambda)}),\pi_s) \cong (D_{\exp(s\lambda)},s\pi_{K^*}) \cong (\mathcal{O}_\lambda,\pi_{\mathfrak{k}^*})
$$
which is an action-angle coordinate chart on  $\mathcal{O}_\lambda$. 
Everything in this argument is equivariant with respect to the maximal torus of $K$ and thus the resulting embeddings are also $T$-equivariant.

The constant $\delta>0$ depends on $s$, and $\delta\to 0$ as $s\to-\infty$. And, the difference in symplectic volumes
\[
{\rm Vol}(\mathcal{O}_\lambda) - {\rm Vol}(\Delta_\lambda(\delta)\times(S^1)^{m}) >0
\]
approaches $0$ as $\delta\to 0$. So, by picking $s\ll 0$, one arrives at the desired embedding.

 The proof of Theorem \ref{thm:intro_1} requires some additional work. Na\"ively, one would like to  take action coordinates constructed as above, on subsets of $\kk^*$, and pull them back via the moment map to a multiplicity free space. However, pullbacks of functions generating periodic flows on $\kk^*$ do not necessarily generate periodic flows: there may be so-called  ``nutation effects'' that cause the resulting  flow  to be aperiodic \cite{GS5}. Thus, one of the main steps in the proof of Theorem \ref{thm:intro_1} is to show that  our coordinates generate periodic flows on multiplicity free spaces. This is achieved in Section~\ref{mainresultssection} by first constructing coordinates that generate approximately periodic flows for large $s$, then deforming (via a Moser trick) to coordinates that generate periodic flows.

\subsection{Structure of the paper}

The construction presented in the paper requires background material from different fields. Sections \ref{background section}, \ref{PositivitySection} and \ref{PTSection} contain an extensive review of this material. We also recall our previous results, and in some cases upgrade them to the level needed for the present paper. Sections \ref{mainresultssection} and \ref{gromov width section} contain the main results of the paper and their proofs.

We can imagine several ways to read the present text. One possibility is to first check the introductory sections and then read the main body of the paper which are Sections \ref{mainresultssection} and \ref{gromov width section}. Another interesting option is to start directly with those two sections and then consult the introductory sections for missing details.

In what follows we describe the content of each section with an accent on elements which are either new or require non-standard presentation.

In Section~\ref{background section}, we collect background material on Lie theory, the theory of Hamiltonian and Poisson group actions and on cluster structures. 
Section~\ref{lie theory section} is devoted to Lie theory including the theory of double Bruhat cells and generalized minors. We also touch upon {\em comparison maps} between the group $G$ and its Langlands dual $G^\vee$.
Section~\ref{hamiltonian group actions section} summarizes the theory of  Hamiltonian $K$-manifolds with focus on Thimm torus actions and multiplicity free Hamiltonian spaces. 
Section~\ref{PLgroupSection} is a reminder of the theory of Poisson $K$-actions and of the dual Poisson-Lie group $K^*$. An interesting new element is the notion of {\em Legendre transform} on $K^*$.
Section~\ref{Cluster Varieties} provides background on cluster algebras and on cluster structures on double Bruhat cells. A somewhat non-standard part of the discussion concerns {\em homogeneous cluster algebras}.

Section~\ref{PositivitySection} describes positivity and tropicalization of double Bruhat cells. 
Section~\ref{section; positivity and trop} is devoted to generalities on positive varieties with potential and their tropicalizations. We also discuss comparison maps between a cluster variety and its Langlands dual.
Section~\ref{section;chartsDBC} contains applications of the material of the previous section to double Bruhat cells $G^{w_0, e}$. In particular, we recall the definition of the Berenstein-Kazhdan (BK) potential $\Phi_{BK}$.
In Section~\ref{section; domination}, we discuss functions dominated by potential, and we give criteria for domination by $\Phi_{BK}$.  
Section~\ref{section;geometriccrystals} contains elements of the geometric crystal theory. In fact, in this paper we do not need crystal operations. So, we only review the {\em highest weight} and {\em weight maps} on geometric crystals.
In Section~\ref{section;canonical bases}, we discuss tropicalization of double Bruhat cells and reduced double Bruhat cells. The corresponding polyhedral cones will serve as targets for action variables.
In Section~\ref{section;comparisonDBC}, we recall the comparison map between double Bruhat cells $G^{w_0, e}$ and $G^{\vee; w_0, e}$ in the group $G$ and its Langlands dual $G^\vee$ which is the main tool in volume calculations.

Section~\ref{PTSection} defines partial tropicalizations. 
In Section~\ref{ChartsonKstarSection} , we introduce coordinate systems on the dual Poisson-Lie group $K^*$ induced by the cluster structure on the double Bruhat cell $G^{w_0, e}$. 
Then, in Section~\ref{section; definition of partialtrop} we define the partial tropicalization $PT(K^*)$ which carries a constant Poisson structure $\pi_{-\infty}$.
In Section~\ref{section; partialtrop is limit}, we explain how to obtain the Poisson structure of the partial tropicalization as an $s \to -\infty$  limit of a family of Poisson structures on $K^*$.
Section~\ref{PTstructureSection} is the description of symplectic leaves and of torus actions on the partial tropicalization.
Section~\ref{section; properties of partialtrop} is devoted to explicit computations of $\pi_{-\infty}$ in cluster charts of Section 4.1. 
We prove a version of Theorem \ref{thm:intro_cone} which establishes a Darboux normal form for 
$\pi_{-\infty}$. 

Section~\ref{mainresultssection} 
is devoted to the proof of Theorem \ref{thm:intro_1} and Theorem \ref{thm:intro_2}.
Section~\ref{section; statement of main theorem} states the main technical result of the paper which is an extension of partial tropicalization to the universal multiplicity free space $K \times \mathring{\ttt}_+^*$.
Section~\ref{section; proof of main theorem} is devoted to the proof of this result. This section is the core of the paper.
In Section~\ref{section; charts on multiplicity free}, we construct big action-angle coordinate charts on multiplicity free spaces which is one of our main results.

Section~\ref{gromov width section} contains the proof of Theorem~\ref{gromov width theorem} on the Gromov width of coadjoint orbits.

\subsection{Acknowledgements}
We are indebted to A. Berenstein for introducing us to the notion of potential and for his invaluable advice. We would like to thank A. Pelayo for his interesting suggestions.  

Research of AA and YL was supported in part by the National Center for Competence in Research (NCCR) SwissMAP and by the grants number 178794 and 178828 of the Swiss National Science Foundation.  JL thanks the
Fields Institute as well as the organizers of the Toric Topology and Polyhedral Products thematic program for
the support of a Fields Postdoctoral Fellowship during the writing of this paper. BH was supported in part by the National Science Graduate Research Fellowship, grant number DGE-1650441.


\section{Background} \label{background section}

Some of the key statements of this Section are Example~\ref{T^*K cross section example} which introduces the universal multiplicity free space $K \times \mathring{\ttt}_+^*$, the Delinearization Theorem~\ref{theorem;delin} which explains how to replace $\mathfrak{k}^*$ valued moment maps with $K^*$ valued ones, 
Example~\ref{example; delinearization of T^*K cross section} which is an application of this technique to $K \times \mathring{\ttt}_+^*$,
Proposition~\ref{proposition; homogeneous cluster creterion} which gives a criterion for cluster algebras to be homogeneous and Corollary~\ref{corollary;homog} which states that double Bruhat cells are homogeneous cluster varieties.

\subsection{Lie theory}\label{lie theory section}

In this section we fix notation which will be used throughout the article. Much of the material was developed in \cite{BKII, FZ}.

\subsubsection{The groups \texorpdfstring{$K$}{K} and \texorpdfstring{$G$}{G}}

In all that follows, $K$ will be a compact connected Lie group, and $G$ will be its complexification. Then $G$ is a connected reductive complex algebraic group. Pick a pair of opposite Borel subgroups $B, B_- \subset G$, and let $H=B\cap B_-$ be the Cartan subgroup. Let $T=H\cap K$ be the maximal torus of $K$, and let $N\subset B$ and $N_-\subset B_-$ be the unipotent radicals. We denote the Lie algebras of these groups with fraktur letters, so for instance $\operatorname{Lie}(K)=\mathfrak{k}$ and $\operatorname{Lie}(B_-)=\mathfrak{b}_-$. 

Let $\ttt^*$ be the $\R$-linear dual of $\ttt$. The projection $\kk^*\to \ttt^*$ induces a linear isomorphism $(\kk^*)^T \cong \ttt^*$, where $(\kk^*)^T$ is the subspace of $T$-invariant elements of $\kk^*$. We will identify these two spaces so that $\ttt^*\subset\kk^*$. Let $\h^*$ be the $\C$-linear dual of $\h$. Then $\ttt^*\otimes \mathbb{C} \cong \h^*$.

The operation of taking the center of a Lie group (resp. algebra) is denoted by $Z(\cdot)$ (resp. $\mathfrak{z}(\cdot)$). The Lie algebras $\g$ and $\kk$ split as the direct sums 
\[
\mathfrak{g}=\z(\g) \oplus [\g,\g],\qquad \kk = \z(\kk) \oplus [\kk,\kk],
\]
where $[\g,\g]$ and $[\kk,\kk]$ are semisimple. 
Let $r$ denote the rank of $\g$, defined as the rank of the semisimple part of $\g$. Let $\tilde r = \dim_\C (\h)$.

\subsubsection{Lattices and Weyl chambers}
%
%

Recall the functors $X_* = \Hom(\C^\times,-)$ and $X^* = \Hom(-,\C^\times)$. In particular, for a complex torus $S$ of dimension $\tilde{r}$,
 \begin{align*}
 X^*(S)  = \Hom (S, \C^\times)\cong \Z^{\tilde{r}},\qquad 
 X_*(S)  = \Hom(\C^\times, S) \cong \Z^{\tilde{r}}
 \end{align*}
 are the lattices of \emph{characters} and \emph{cocharacters} of $S$, respectively. They come with the natural evaluation pairing $\langle \cdot,\cdot\rangle$. When $H$ is as in the previous section, we make the standard identifications
 \[
 X^*(H)\otimes_\Z \R =  \sqrt{-1}\mathfrak{t}^*,\qquad X^*(H)\otimes_\Z \C = \h^*,\qquad X_*(H)\otimes_\Z \C= \mathfrak{h}.
 \]
 For a character $\gamma\in X^*(H)$ and $h\in H$, write $h\mapsto h^\gamma$ for the evaluation of $\gamma$ at $h$. 

Let $R\subset X^*(H)$ be the set of roots of $G$, and $R^\vee\subset X_*(H)$ the set of coroots. The choice of positive Borel subgroup $B$ determines the positive roots $R_+\subset R$.  Fix an enumeration of the simple roots $\alpha_1,\dots,\alpha_r \in R_+ $ and simple coroots $\alpha^\vee_1,\dots,\alpha^\vee_r\in R^\vee$. We write $[1,r]=\{1,\dots,r\}$ for their index set. The Lie algebra $\g$ can be decomposed into root spaces $\g = \h \oplus \bigoplus_{\alpha \in R} \g_\alpha$, where $\g_\alpha$ is the $\alpha$-weight space for the adjoint representation of $\g$.
Let $\omega_1,\dots,\omega_r$ be the \emph{fundamental weights}. By definition,
\[
\langle \omega_i,\alpha_j^\vee\rangle = \delta_{i,j}, \qquad \langle \omega_i, \mathfrak{z}(\ttt)\rangle = 0,\qquad \forall i,j\in [1,r].
\]

The positive Weyl chamber $\ttt_+^*\subset\kk^*$ is the intersection of half spaces defined by simple coroots  $\alpha_1^\vee, \dots , \alpha_r^\vee$,
\[
    \ttt_+^* = \{ \xi \in \ttt^* \mid \langle \sqrt{-1} \xi,\alpha_i^\vee\rangle \geqslant 0, ~ \forall i = 1,\dots ,r\}.
\]
For each $J \subset [1,r]$ define
\[
    \sigma_J = \{ \xi \in\ttt^*_+ \mid \langle \sqrt{-1} \xi,\alpha_i^\vee\rangle > 0 \text{ if and only if } i \in J\}.
\]
The subsets $\sigma_J$, $J\subset [1,r]$, define a stratification of $\ttt_+^*$. 
 The maximal stratum of $\ttt^*_+$ (with respect to the inclusion partial order) is the relative interior of $\ttt_+^*$,  denoted $\mathring{\ttt}_+^*$.

The fundamental weights and simple roots of $\g$ generate the \emph{weight lattice} $P$ and \emph{root lattice} $Q$, respectively.  Let $P_+=\sqrt{-1}\ttt^*_+\cap P$ and $X^*_+(H)=\sqrt{-1}\ttt^*_+ \cap X^*(H)$ denote the sets of dominant weights and dominant characters, respectively.

\subsubsection{\texorpdfstring{$\SL_2$}{SL2} triples and the Weyl group}


The Weyl group $W= \operatorname{Norm}_G(H)/H$ acts on the character lattice $X^*(H)$ by
\begin{equation}
\label{equation;weylaction}
h^{w\gamma} = (\tilde{w}\n h \tilde{w})^\gamma, 
\end{equation}
where $h\in H,~w\in W,~\gamma\in X^*(H)$, and $\tilde{w}\in \operatorname{Norm}_G(H)$ is any lift of $w$.  The action \eqref{equation;weylaction} does not depend on the choice of lift $\tilde{w}$. The Weyl group is generated by the simple reflections $s_i,~i\in [1,r]$, whose action on $\gamma\in X^*(H)$ is given by
\[
s_i(\gamma) = \gamma-\gamma(\alpha_i^\vee) \alpha_i.
\]
Let $w_0$ be the longest element in $W$, whose length $\ell(w_0)$ in terms of the simple reflections $s_i$ is $m:=\ell(w_0)$.

Consider the Chevalley generators $e_i,~f_i,~\alpha^\vee_i$ of $[\g,\g]$. Define
\[
x_i(a):=\exp(ae_i)
 \in N,\quad y_i(a):= \exp(af_i)
 \in N_- .\]
 We lift the Weyl group to $\operatorname{Norm}_G(H)$ by setting
\[
\overline{s}_i = x_i(-1) y_i(1) x_i(-1).
\]
 If $w\in W$ and $w=s_{i_1}\cdots s_{i_l}$ is any reduced expression for $w$, define $\overline{w} = \overline{s}_{i_1} \cdots \overline{s}_{i_l}\in G$.  The $\overline{s}_i$'s satisfy the Coxeter relations of $W$, so the definition of $\overline{w}$ does not depend on the choice of reduced expression.
 
 \subsubsection{Anti-holomorphic involutions on \texorpdfstring{$G$}{G}}
 
 Consider the $\C$-antilinear involution $(\cdot)^\dag\colon \g \to \g$ which fixes $\mathfrak{a}$ and has $e_i^\dag = f_i$ for all Chevalley generators $e_i$. This lifts to an anti-holomorphic Lie group anti-involution 
\begin{align*}
(\cdot)^\dag\colon &  G\mapsto G \ :\   g\mapsto g^\dag.
\end{align*}
The fixed points of $g\mapsto (g\n)^\dag$ form the compact subgroup $K$ of $G$.
 
 Additionally, consider the $\C$-antilinear involution $\overline{(\cdot)} \colon \g \to \g$ which fixes $\mathfrak{a}$ and has $\overline{e}_i = e_i$ and $\overline{f}_i = f_i$ for all Chevalley generators $e_i, f_i$. This lifts to an anti-holomorphic Lie group involution $\overline{(\cdot)}\colon G\to G$ which restricts to an involution of $K$.

\subsubsection{Generalized minors and double Bruhat cells}

An element $x\in G$ is \emph{Gaussian decomposable} if $x\in G_0 = N_- H N$. The set $G_0$ is an open subvariety of $G$. For $x\in G_0$, write
\[
x= [x]_- [x]_0 [x]_+,\qquad \text{where }[x]_-\in N_-,~[x]_0\in H, ~[x]_+\in N.
\]
Similarly, write $[x]_{\geqslant 0} = [x]_0 [x]_+$ and $[x]_{\leqslant 0} = [x]_- [x]_0$. 

Let $\gamma \in X^*_+(H)$ be a dominant character. The \emph{principal minor} $\Delta_\gamma\in \C[G]$ is the regular function uniquely determined by
\[
\Delta_\gamma(x) = [x]_0^\gamma, \quad \text{ for } ~ \forall x\in G_0.
\]
Now, let $u,v\in W$. The \emph{generalized minor} $\Delta_{u\gamma,v\gamma}\in \C[G]$ is defined to be
\[
\Delta_{u\gamma,v\gamma} (x) := \Delta_\gamma(\overline{u}\n x \overline{v} ).
\]
 For $h,h'\in H$, the generalized minors satisfy
 \begin{equation}
     \label{homogeneitygenminor}
     \Delta_{u\gamma,v\gamma} ( hxh') = h^{u\gamma} \Delta_{u\gamma,v\gamma} ( x) {h'}^{v\gamma}
 \end{equation}
 for all $x\in G$.
 When $G=\SL_n(\C)$, the generalized minors are minors.

For each pair of Weyl group elements $(u,v)$, a \emph{double Bruhat cell} and a \emph{reduced double Bruhat cell} are defined respectively by
\[
  G^{u,v}:=BuB\cap B_-vB_-;\quad L^{u,v}:=N\overline{u}N\cap B_-vB_-.
\]
A point $x\in G^{u,v}$ is contained in $L^{u,v}$ if and only if
\begin{equation} \label{equation;RDBCdefining}
[\overline{u}\n x]_0 = 1.
\end{equation}
Multiplication in $G$ induces a biregular isomorphism $H\times L^{u,v}\cong G^{u,v}$. 

If $p \colon \widehat{G}\to G$ is a covering group of $G$, then \eqref{equation;RDBCdefining} implies that the covering morphism induces biregular isomorphisms
\begin{equation} \label{equation;coverL}
p\colon \widehat{L}^{u,v} \to L^{u,v}
\end{equation}
from the reduced double Bruhat cells of $\widehat{G}$ to those of $G$. For any weight $\gamma \in P$ of $G$, if $\gamma$ is not contained in $X^*(H)$, we can still make sense of the generalized minor
\[
\Delta_{u\gamma,v\gamma}\in \C[ L^{u,v} ]
\]
by identifying $L^{u,v}$ with the reduced double Bruhat cell of any covering group $\widehat{G}$ for which $\gamma$ is a character.

In what follows we will focus on the double Bruhat cell $G^{w_0,e}$, which is an open subvariety of $B_-$. We recall a special case of the \emph{twist map} of \cite{FZ}. Our twist map is the biregular involution
\begin{equation}
\label{twist map equation}
\zeta\colon G^{w_0,e} \to G^{w_0,e}, \quad x \mapsto ([\overline{w_0}\n x]_{\geqslant 0 })^\theta,
\end{equation}
where $g\mapsto g^\theta$ is a Lie group involution of $G$ determined by
\[
x_i(t)^\theta = y_i(t),  \text{~for~} i\in [1,r]; \quad  h^\theta = h\n, \text{~for~}  h\in H.
\] 
For a generalized minor $\Delta_{\gamma,\delta}$, let $\Delta_{\gamma,\delta}^\zeta := \Delta_{\gamma,\delta}\circ \zeta$ be the \emph{twisted minor}. For $\gamma\in X^*_+(H)$, 
\begin{equation}\label{equation;twistfrozen}
\Delta^\zeta_{w_0\gamma,\gamma} = \Delta_{\gamma,\gamma}\n,\qquad \Delta^\zeta_{\gamma,\gamma} = \Delta_{w_0\gamma,\gamma}\n.
\end{equation}

\subsubsection{Langlands dual groups}

Let $G^\vee$ be the \emph{Langlands dual group} of $G$.
For a fixed Cartan subgroup $H^\vee\subset G^\vee$, one has
\[
X^*(H) = X_*(H^\vee);\qquad X^*(H^\vee) = X_*(H).
\]
The group $G^\vee$ is important in describing the representation theory of $G$; see Section~\ref{section;geometriccrystals} below. 
As with $H^\vee$, we superscript with $\vee$ those varieties and maps associated with $G^\vee$ (as opposed to $G$). So for instance the reduced double Bruhat cells of $G^\vee$ are $L^{\vee; u,v}$ rather than $L^{u,v}$.

\subsubsection{Comparison map and bilinear form on \texorpdfstring{$\g$}{g}}\label{section;comparison}

The Cartan matrix $A$ of $G$ is the $r\times r$ matrix with entries $A_{i,j} = \langle \alpha_j,\alpha_i^\vee \rangle$. Let $D=\diag(d_1,\dots,d_r)$ be a $r\times r$ diagonal matrix, with entries $d_i\in \Z$, so that $AD=(AD)^T$. The matrix $D$ is then a \emph{symmetrizer} of $A$.

Fix a non-degenerate, $\g$-invariant bilinear form $(\cdot,\cdot)_\g\colon \g \otimes_\C \g \to \C$, as follows. Pick $\gamma^\vee_{1}, \dots \gamma^\vee_{\tilde{r}-r} \in X_*(H)\cap \z(\g)$ so that $\alpha^\vee_1,\dots,\alpha^\vee_r,\gamma^\vee_{1},\dots, \gamma^\vee_{\tilde{r}-r}$ forms a $\C$-basis of $\h$. Then $(\cdot,\cdot)_\g$ is uniquely determined by:
\[
(\alpha^\vee_i,\alpha_j^\vee)_\g = A_{i,j} d_j, \quad (\gamma^\vee_i,\gamma^\vee_j)_\g=\delta_{i,j},\quad \z(\g)=[\g,\g]^\perp.
\]
Denote the induced bilinear form on $\g^*$ by $(\cdot,\cdot)_{\g^*}$. If $V$ is a (real or complex) linear subspace of $\g$ or $\g^*$, denote the restriction of the bilinear form to $V$ by $(\cdot,\cdot)_V$. When it is clear from context we omit the subscripts.

The \emph{comparison map} 
\begin{equation}
\psi_\g\colon \g \to \g^*
\end{equation} is the $\C$-linear map determined by $(\cdot,\cdot)_\g$:
\[
\langle \psi_\g(X), Y \rangle = (X,Y)_\g, \qquad \forall X,Y \in \g.
\]
For $\kk\subset \g$ and $\h\subset \g$, we have similarly defined maps
\begin{align}
\psi_\kk \colon \kk \to \kk^*,\qquad \psi_\h \colon \h \to \h^* \label{equation;comparisonmap}
\end{align}
Then $\psi_\h(\alpha^\vee_i)= d_i \alpha_i$. We will assume we have chosen $D$ so that $\psi_\h$ restricts to a $\Z$-module homomorphism $X_*(H)\to X^*(H)=X_*(H^\vee)$, which is possible by \cite[Proposition~2.2]{ABHL2}. This induces a group homomorphism
\[
\varPsi_H\colon H \to H^\vee.
\]
\subsection{Hamiltonian \texorpdfstring{$K$}{K}-manifolds}\label{hamiltonian group actions section}

This section recalls basic facts and conventions regarding Hamiltonian group actions. More details may be found in \cite{symplectic-techniques}. 

The Hamiltonian vector field of a smooth function $f\in C^{\infty}(M)$ on a symplectic manifold $(M,\omega)$ is the vector field $X_f$ defined by the equation $\iota_{X_f} \omega = -df$.
Given a Lie group $K$ with Lie algebra $\kk$ and a smooth left action of $K$ on a manifold $M$, the fundamental vector field of $X\in\kk$ is the vector field $\underline{X}\in \mathfrak{X}(M)$ defined point-wise as
\[
\underline{X}_m = \frac{d}{dt}\bigg\vert_{t=0}\exp(-tX)\cdot m,\qquad m\in M.
\]

\begin{definition}[Hamiltonian $K$ action]\label{hamiltonian K-manifold definition}
  Suppose that $(M,\omega)$ is a symplectic manifold and $K$ is a connected Lie group. A smooth left action of $K$ on $M$ is \emph{Hamiltonian} if there is a $K$-equivariant map $\mu\colon M \to \kk^*$ such that $\iota_{\underline{X}}\omega = -d\langle \mu,X\rangle$ for all $X\in \kk$. The map $\mu$ is called a \emph{moment map} for the action of $K$. The data of $(M,\omega)$ together with a Hamiltonian action of $K$ and a moment map $\mu$ is a \emph{Hamiltonian $K$-manifold}, denoted $(M,\omega,K,\mu)$ or, when it is unambiguous, simply $(M,\omega,\mu)$.
\end{definition}

A map of Hamiltonian $K$-manifolds $(M,\omega,\mu)$ and $(M',\omega',\mu')$ is a $K$-equivariant symplectic map $F\colon M\to M'$ such that $\mu'\circ F = \mu$. If such an $F$ is also a symplectomorphism, then we say that $F$ is an \emph{isomorphism} of Hamiltonian $K$-manifolds.


Let $K$ be a compact connected Lie group with positive Weyl chamber $\ttt_+^*$  as in Section \ref{lie theory section}. For every connected Hamiltonian $K$-manifold $(M,\omega,\mu)$ there is a unique stratum $\sigma \subset \ttt_+^*$, called the \emph{principal stratum of} $(M,\omega,\mu)$, with the property that $\mu(M) \cap \sigma$ is dense in $\mu(M) \cap \ttt_+^*$ \cite{symplectic-techniques}. The principal stratum of $(M,\omega,\mu)$ is the maximal stratum of $\ttt_+^*$ with the property that $\mu(M) \cap \sigma$ is non-empty. 
If the principal stratum of $(M,\omega,\mu)$ is $\mathring{\ttt}_+^*$, then the symplectic cross-section theorem can be stated as follows.

\begin{theorem}[Symplectic cross-section theorem]\label{cross section theorem}\cite[Theorem 26.7, Proposition 41.2]{symplectic-techniques}
    Let $K$ be a compact connected Lie group and let $(M,\omega,\mu)$ be a connected Hamiltonian $K$-manifold with principal stratum $\mathring{\ttt}_+^*$. Then,
    \begin{enumerate}[(i)]
        \item The preimage, $\mu\n(\mathring{\ttt}_+^*)$, is a non-empty, symplectic, $T$-invariant submanifold of $M$.
        \item The action of $T$ on $\mu\n(\mathring{\ttt}_+^*)$ is Hamiltonian with moment map $\mu$.
    \end{enumerate}
\end{theorem} 

\begin{example}[Symplectic cross-section of $T^*K$]\label{T^*K cross section example}
Let $T^*K$ denote the cotangent bundle of $K$. Fix the standard trivializations, $T^*K \cong K \times \kk^*$ and $T(T^*K) \cong K \times \kk^* \times \kk\times \kk^*$ by left invariance. The canonical symplectic form on $T^*K$ is $\omega_\mathrm{can} = -d\theta_{\mathrm{taut}}$, where $\theta_{\mathrm{taut}}$ is the tautological 1-form.

There is a (left) action of $K\times K$ on $T^*K$ by the cotangent lift of left and right multiplication. In terms of the trivialization, the action is
$(k_1,k_2) \cdot (k,\xi) = (k_1kk_2\n,\Ad_{k_2}^*\xi)$. 
This action is Hamiltonian with moment map
\begin{equation}\label{equation; cotangent bundle moment map}
    (\mu_L , \mu_R) \colon K\times \kk^* \rightarrow \kk^* \times \kk^*, \quad (k,\xi) \mapsto (\Ad^*_k\xi,-\xi).
\end{equation}
The symplectic cross-section of $T^*K$ with respect to the action of the right copy of $K$ and the opposite Weyl chamber $-\ttt_+^*$ is $\mu_R\n(-\mathring{\ttt}_+^*) = K \times \mathring{\ttt}_+^*$. 
It is a Hamiltonian $K\times T$ manifold. 
\end{example} 

\subsubsection{The Thimm torus action}\label{thimm torus background section}

This section recalls the definition of an additional Hamiltonian $T$-action, called the \emph{Thimm torus action}, that is defined on certain dense subsets of Hamiltonian $K$-manifolds and commutes with the action of $K$. See \cite{GS1}.

Let $(M,\omega,\mu)$ be a connected Hamiltonian $K$-manifold  with principal stratum\footnote{We include this assumption only for the sake of simplified exposition; it is not necessary in order to define a Thimm torus action on a dense subset of $M$.}  $\mathring{\ttt}_+^*$. Consider the symplectic cross-section of $(M,\omega,\mu)$ with respect to $\ttt_+^*$ and the symplectic cross-section of $(T^*K,\omega_{\mathrm{can}},\mu_R)$ with respect to the opposite Weyl chamber $-\ttt_+^*$ (see Example \ref{T^*K cross section example}). Their symplectic product,
\[
    (K \times \mathring{\ttt}_+^*\times \mu\n(\mathring{\ttt}_+^*),\omega_{\mathrm{can}} \oplus \omega),
\]
is a Hamiltonian $K\times T \times T$-manifold, 
with action $(k',t_1,t_2) \cdot (k,\xi,x) = (k' k t_1\n,\xi,t_2\cdot x)$ and moment map 
\[
    (\mu_L,\mu_R,\mu)\colon K \times \mathring{\ttt}_+^*\times \mu\n(\mathring{\ttt}_+^*) \to \kk^* \times \ttt^* \times \ttt^*, \quad (k,\xi,x) \mapsto (\Ad_k^*\xi,-\xi,\mu(x)).
\]
The diagonal action of $T$ by $t\cdot (k,\xi,x) = (kt\n,\xi,t\cdot x)$ is Hamiltonian with moment map $(k,\xi,x) \mapsto \mu(x)-\xi$. The symplectic reduction by the diagonal $T$-action is a smooth symplectic manifold, 
\begin{equation} \label{equation;reducedproduct}
    K\times_T \mu\n(\mathring{\ttt}_+^*) := (K\times \mu\n(\mathring{\ttt}_+^*))/T  \cong (K \times \mathring{\ttt}_+^*\times \mu\n(\mathring{\ttt}_+^*),\omega_{\mathrm{can}} \oplus \omega)\sslash_0 T.
\end{equation}
The isomorphism arises from the $T$-equivariant map $(k,x) \mapsto (k,\mu(x),x)$.

Let $[k,x]\in K\times_T \mu\n(\mathring{\ttt}_+^*)$ denote the $T$-equivalence class of $(k,x)$. The reduced space \eqref{equation;reducedproduct} is a Hamiltonian $K\times T$-manifold, with action $(k',t)\cdot [k,x] = [k'kt,x] = [k'k,t\cdot x]$ and moment map
\[
    (\overline{\mu}_L,\overline{\mu}_R) \colon K \times_T \mu\n(\mathring{\ttt}_+^*) \to \kk^*\times \ttt^*, \quad \overline{\mu}_L([k,x])=\Ad_k^*\mu(x),\quad \overline{\mu}_R([k,x]) = -\mu_R(x) = \mu(x).
\]
The map
\begin{equation}\label{dense embedding}
    \varphi\colon K\times_T \mu\n(\mathring{\ttt}_+^*) \to M,\quad [k,x] \mapsto k\cdot x
\end{equation}
is a $K$-equivariant symplectomorphism onto a dense, $K$-invariant subset of $M$ and satisfies $\mu \circ \varphi = \overline{\mu}_L$. Denote the image of $\varphi$ by $U$. As a consequence of this construction, one has a Hamiltonian action of $K\times T$ on $U$.

The construction of an extra Hamiltonian $T$ action on $U$ can be rephrased in the following way. Define a map $\mathcal{S}\colon \kk^* \to \ttt_+^*$ by letting $\mathcal{S}(\xi)$ be the unique element of the intersection $\mathcal{O}_\xi \cap \ttt_+^*$. By $K$-equivariance, $U = (\mathcal{S}\circ \mu)\n(\mathring{\ttt}_+^*)$. Define an action of $T$ on $U$ by
\begin{equation}\label{new action}
    t\ast x = (k\n t k)\cdot x,
\end{equation}
where $k\in K$ such that $\Ad_k^* \mu(x)\in \mathring{\ttt}_+^*$. This action is well-defined since the stabilizer subgroup of elements in $\mu\n(\mathring{\ttt}_+^*)$ is contained in $T$. The restriction of $\mathcal{S}\circ \mu$ to $U$ is a moment map for this action (in particular, its restriction to $U$ is smooth) \cite{GS1}. The preceding discussion can be summarized as follows.

\begin{proposition}\label{proposition about dense model space}
    Let $K$ be a compact connected Lie group, let $(M,\omega,\mu)$ be a connected Hamiltonian $K$-manifold with principal stratum $\mathring{\ttt}_+^*$, and let $U=(\mathcal{S}\circ \mu)\n(\mathring{\ttt}_+^*)$. Then, the map
    \[
        K\times_T \mu\n(\mathring{\ttt}_+^*) \to U,\quad [k,x] \mapsto k\cdot x
    \]
    is an isomorphism of Hamiltonian $K\times T$-manifolds, i.e.\ 
    it is a $K\times T$-equivariant symplectomorphism and the following diagrams commute.
    \begin{equation}
    \begin{tikzcd}
        K\times_T \mu\n(\mathring{\ttt}_+^*) \ar[d,"\overline{\mu}_L"] \ar[r] & U \ar[d,"\mu"] \\
        \kk^* \ar[r,"="] & \kk^*
    \end{tikzcd}\quad \quad   
    \begin{tikzcd}
        K\times_T \mu\n(\mathring{\ttt}_+^*) \ar[d,"\overline{\mu}_R"] \ar[r] & U \ar[d,"\mu"] \\
        \ttt^*_+  & \kk^*\ar[l,"\mathcal{S}"]
    \end{tikzcd}
\end{equation}
\end{proposition}

The extra Hamiltonian $T$ action on $U$ is often referred to as the \emph{Thimm torus action}\footnote{Although we like to call it the Thimm torus action, we note that to our knowledge this torus action originated in the work of Guillemin and Sternberg \cite{GS1}.}. It does not in general extend to a well-defined  action of $T$ on $M$. 

\subsubsection{Toric manifolds and multiplicity free spaces}\label{section;multiplicity free spaces}

A Hamiltonian $T$-manifold $(M,\omega,\mu)$ is \emph{completely integrable} if $T$ acts locally transitively on fibers of $\mu$. If the action of $T$ is also effective, then $(M,\omega,\mu)$ is a \emph{toric manifold} (or \emph{toric $T$-manifold} if we want to emphasize the r\^ole of the torus).

\begin{example}\label{toric domain example} As in Example \ref{T^*K cross section example}, the cotangent bundle $T^*T$ of a compact torus $T$ has a canonical symplectic structure, $\omega_{\mathrm{std}} = -d\theta_{\mathrm{taut}}$. With respect to the trivialization $T^*T = \ttt^* \times T$, the action of $T$ on $T^*T$ by cotangent lift of left multiplication is $t \cdot (\lambda,t') = (\lambda,tt')$ and the moment map is $\pr(\lambda,t) = \lambda$ (cf.~the formula for $\mu_L$ in  Equation \eqref{equation; cotangent bundle moment map}). For any open subset $U\subset \ttt^*$, $(U\times T, \omega_{\mathrm{std}}, \pr)$ is a toric manifold.  For the torus $\T^n = S^1 \times \dots \times S^1$ we fix the identification $\Lie(S^1)^* = \R$ such that $\omega_{\mathrm{std}}$ is given by Equation \eqref{equation; omega std}.
\end{example}




By equivariance of the moment map, a Hamiltonian $K$-manifold $(M,\omega,\mu)$ is a multiplicity free space if and only if the action of the maximal torus $T\subset K$ on the symplectic cross-section is  completely integrable.

Let $(M,\omega,\mu)$ be a compact, connected multiplicity free space with principal stratum $\mathring{\ttt}_+^*$, Kirwan polytope $\triangle = \mu(M) \cap \ttt_+^*$, and principal isotropy subgroup $L \leq T$. Let $\mathring{\triangle}$ denote the relative interior of $\triangle$. Then $W = \mu\n(\mathring{\triangle})$ is an open, dense, $T$-invariant subset of the symplectic cross-section $\mu\n(\mathring{\ttt}_+^*)$. The kernel of the action of $T$ on $W$ is the subgroup $L$ and the action of $T/L$ on $W$ is free. The convex set $\mathring{\triangle}$ spans an affine subspace of $\ttt^*$ that is a translation of the subspace of $\ttt^*$ given by the image of the map $\Lie(T/L)^* \to \ttt^*$ dual to the quotient map $\ttt \to \Lie(T/L)$. In particular, the dimension of  $\mathring{\triangle}$ and $T/L$ are equal and $\mathring{\triangle}$ is open as a subset of this affine subspace \cite{knop2}. Fix a linear identification of $\Lie(T/L)^*$ with the affine subspace spanned by $\mathring{\triangle}$. 

With respect to the identification of $\Lie(T/L)^*$ with the affine subspace spanned by $\mathring{\triangle}$, the map  $\mu\colon W \to \mathring{\triangle}$ is a moment map for the free action of $T/L$ on $W$.  The tuple $(W,\omega,\mu)$ is thus a toric $T/L$-manifold.

Define the toric $T/L$-manifold $(\mathring{\triangle}\times (T/L),\omega_{\mathrm{std}},\pr)$ as in Example \ref{toric domain example}, taking care to note that $\Lie(T/L)^*$ is identified with the affine subspace of $\ttt^*$ spanned by $\mathring{\triangle}$.
It follows from compactness of $M$ that $\mu$ is proper as a map to $\mathring{\triangle}$. By the classification of non-compact toric manifolds with proper moment maps,  $(W,\omega,\mu)$ is isomorphic as a toric $T/L$-manifold to  $(\mathring{\triangle}\times (T/L),\omega_{\mathrm{std}},\pr)$  \cite{KL}. The preceding discussion is summarized as follows. Note that the commutative diagram \eqref{diagram: mult free isomorphism} is stated in terms of maps to $\ttt^*$. This is due to our identification of $\Lie(T/L)^*$ with an affine subspace of $\ttt^*$.

\begin{lemma}\label{model for sympelctic slice}
    Let $(M,\omega,\mu)$ be a compact, connected, multiplicity free Hamiltonian $K$-manifold with principal stratum $\mathring{\ttt}_+^*$, Kirwan polytope $\triangle = \mu(M) \cap \ttt_+^*$, and principal isotropy subgroup $L \leq T$. Let $(W,\omega,\mu)$ and $(\mathring{\triangle}\times (T/L),\omega_{\mathrm{std}},\pr)$ be the toric $T/L$-manifolds defined as above. Then $(W,\omega,\mu)$ and $(\mathring{\triangle}\times (T/L),\omega_{\mathrm{std}},\pr)$ are isomorphic as Hamiltonian $T/L$-manifolds, i.e.\ there is a $T$-equivariant (therefore also $T/L$-equivariant) symplectomorphism $(W,\omega) \cong (\mathring{\triangle} \times (T/L),\omega_{\mathrm{std}})$ such that the following diagram commutes.
    \begin{equation}\label{diagram: mult free isomorphism}
    \begin{tikzcd}
        W \ar[d,"\mu"] \ar[r,"\cong"] & \mathring{\triangle} \times (T/L) \ar[d,"\pr"] \\
        \ttt^*  & \ttt^*\ar[l,"="]
    \end{tikzcd}
\end{equation}
\end{lemma}

\subsection{Hamiltonian Poisson-Lie group actions} \label{PLgroupSection}

This section recalls the relevant details of compact connected Poisson-Lie groups and their Hamiltonian actions from \cite{A97,AMW,LW}.   For further background on Poisson-Lie groups the reader may consult \cite{Chari-Pressley,Etingof-Schiffmann,Laurent-Gengoux}. 

In general, a \emph{Poisson-Lie group}, $(K,\pi)$, is a Lie group $K$ equipped with a Poisson bivector field $\pi$ such that the group multiplication map $K\times K \to K$ is a Poisson map (with respect to the product Poisson structure on $K\times K$).  The bivector field $\pi$ necessarily vanishes at the group identity and therefore has a linearization $\delta_\pi\colon \kk \to \kk\wedge \kk$. The linearization is a Lie algebra 1-cocycle and the dual map, $\delta_\pi^*\colon \kk^* \wedge \kk^* \to \kk^*$, is a Lie bracket. In other words, the tuple $(\kk,[\cdot,\cdot]_\kk,\delta_\pi)$ is a Lie bialgebra. Two Lie bialgebras,  $(\kk,[\cdot,\cdot]_{\kk},\delta_\kk)$ and $(\mathfrak{l},[\cdot,\cdot]_{\mathfrak{l}},\delta_\mathfrak{l})$, are \emph{dual} to one another if there is a dual pairing, i.e.\ a non-degenerate bilinear map $\langle\cdot,\cdot\rangle\colon \kk \times \mathfrak{l} \to \R$, such that  
\begin{equation}\label{dual lie bialgebras}
    \langle[X,Y]_\kk,\xi\rangle = \langle X\wedge Y,\delta_{\mathfrak{l}}(\xi)\rangle \quad \text{and} \quad \langle X,[\eta,\xi]_{\mathfrak{l}}\rangle = \langle \delta_\kk(X),\eta\wedge \xi\rangle \quad \forall X,Y \in \kk, \, \eta,\xi \in \mathfrak{l}.
\end{equation}
Two Poisson-Lie groups are \emph{dual}\footnote{It is common in the literature to define \emph{the} dual Poisson-Lie group of a Poisson Lie group $(K,\pi)$ to be the unique simply connected dual Poisson-Lie group.} to one another if their linearizations are dual to each other as Lie bialgebras. 

A \emph{Manin triple} is a tuple of Lie algebras  $(\g,\kk,\mathfrak{l})$ along with a non-degenerate ${\rm ad}$-invariant bilinear form on $\g$ such that: i) $\kk$ and $\mathfrak{l}$ are Lie subalgebras of $\g$, ii) $\g = \kk \oplus \mathfrak{l}$ as vector spaces, and iii) $\kk$ and $\mathfrak{l}$ are isotropic with respect to the bilinear form.  Given a Manin triple $(\g,\kk,\mathfrak{l})$ the bilinear form on $\g$ defines a dual pairing between $\kk$ and $\mathfrak{l}$. If $\delta_\kk$ and $\delta_\mathfrak{l}$ denote the maps defined by Equation \eqref{dual lie bialgebras} with respect to this pairing, then  $(\kk,[\cdot,\cdot]_{\kk},\delta_\kk)$ and $(\mathfrak{l},[\cdot,\cdot]_{\mathfrak{l}},\delta_\mathfrak{l})$ are dual Lie bialgebras.

\subsubsection{Compact Poisson-Lie groups}\label{compact poisson lie groups}

Let $K$ be a compact connected Lie group with complexification $G$. 
For all $s\in \R$, $s\neq0$, define the real-valued bilinear form
\[
    \llangle\cdot,\cdot\rrangle_s := \frac{2}{s}\im (\cdot,\cdot)_\g.
\]
Denote the real Lie subalgebra $\a\oplus \nnn_- \subset \g$ by $\a\nnn_-$ and the corresponding connected subgroup of $G$ by $AN_-$.  The tuple $(\g,\kk,\a\nnn_-)$ along with  $\llangle\cdot,\cdot\rrangle_s$ is a Manin triple.\footnote{Most references, such as \cite{LW}, prefer to use $\a\nnn_+$ in their definition of the standard Manin triple on $\g$. Our convention is equivalent: the map $\g \to \g$, $x\mapsto -x^\dagger$, defines an isomorphism of Manin triples $(\g,\kk,\a\nnn_-,\llangle\cdot,\cdot\rrangle_s)\cong (\g,\kk,\a\nnn_+,\llangle\cdot,\cdot\rrangle_{-s})$.} The dual Lie bialgebras defined by this Manin triple are $(\kk,[\cdot,\cdot]_\kk,\delta_{\kk,s})$ and $(\a\nnn_-,[\cdot,\cdot]_{\a\nnn_-},\delta_{\a\nnn_-,s})$, where $\delta_{\kk,s}$ and $\delta_{\a\nnn_-,s}$ are defined by Equation \eqref{dual lie bialgebras}. Denote $\delta_{\kk} := \delta_{\kk,1}$ and $\delta_{\a\nnn_-} := \delta_{\a\nnn_-,1}$, and observe that $\delta_{\kk,s} = s\delta_\kk$ and  $\delta_{\a\nnn_-,s} = s\delta_{\a\nnn_-}$. The Lie bialgebras $(\kk,[\cdot,\cdot]_\kk,s\delta_{\kk})$ and $(\a\nnn_-,[\cdot,\cdot]_{\a\nnn_-},s\delta_{\a\nnn_-})$ integrate to dual Poisson-Lie groups $(K,s\pi_K)$ and $(AN_-,s\pi_{K^*})$.  For this reason, we often denote $K^* := AN_-$.

\begin{remark}
If $\widehat{G}\to G$ is a covering group of $G$, then the covering map restricts to a Lie group isomorphism from $\widehat{AN}_-\subset \widehat{G}$ to $AN_-\subset G$. For this reason, we often assume that $G$ is of the form
\begin{equation}\label{equation;niceformG}
    G \cong Z \times G_{sc}, 
\end{equation}
where $G_{sc}$ is semisimple and simply connected, and $Z$ is an algebraic torus with Lie algebra $\z(\g)$.
\end{remark}

\begin{remark}\label{remark; quasitriangular r matrix} 
For $K$, which is not assumed to be simply connected, the integration of $\delta_{\kk}$ is achieved by defining an anti-symmetric $r$-matrix $\Lambda_0 \in \kk\wedge\kk$ such that the Lie algebra 1-cocycle $\delta_\kk$ is a coboundary, i.e.\ $\partial\Lambda_0 = \delta_\kk$. The formula for this $r$-matrix is the same as in the semisimple or simple case (see e.g.\ \cite{Etingof-Schiffmann}). The Poisson-Lie group structure on $K$ is then defined as $\pi_K := \Lambda_0^L - \Lambda_0^R$, where $\Lambda_0^L$ and $\Lambda_0^R$ denote the right and left invariant bivector fields on $K$ with $\Lambda_0^L(e) = \Lambda_0^R(e) = \Lambda_0$.
\end{remark}

\begin{remark}
Although there is some choice in the definition of the bilinear form $(\cdot,\cdot)_\g$, the Poisson-Lie groups $K$ and $AN_-$ only depend on the restriction of $(\cdot,\cdot)_\g$ to $[\g,\g]$, since
\[
    \delta_\kk\vert_{\z(\kk)} = 0 \text{ and } \delta_{\a\nnn_-}\vert_{i\z(\kk)} = 0.
\]
\end{remark}

\begin{definition} \label{definition;dressing} For all $k \in K$ and $b\in AN_-$, the equation
\begin{equation}
    kb = b^kk^b
\end{equation}
uniquely determines elements $b^k \in AN_-$ and $k^b\in K$ by the Iwasawa decomposition, $G = AN_- K$. This defines a left action $K \times AN_- \to AN_-$, $(k,b) \mapsto b^k$, and a right action, $AN_- \times K \to K$, $(k,b) \mapsto k^b$. These are the \emph{dressing actions} of $K$ and $AN_-$ on each other. 
\end{definition}

Note that since $T$ normalizes $AN_-$, the dressing action of $T$ on $AN_-$ is simply conjugation.

\subsubsection{Explicit formula for the Poisson bracket on \texorpdfstring{$K^*$}{K*}}
\label{subsubsection;explicitformula}
The real Poisson structure $\pi_{K^*}$ can be described in terms of a coboundary Poisson-Lie group structure on $G$ as follows. 

Fix a $\C$-basis of $\g$, consisting of $E_\alpha \in \g_\alpha$, $\alpha \in R$, and $X_j \in \a$, $j= 1,\ldots ,\tilde r$, subject to 
\begin{equation} \label{equation;basisproperties}
(E_\alpha,E_{-\alpha})_\g = 1, \forall \alpha \in R, \qquad  (X_i,X_j)_\g = \delta_{i,j}, i,j \in [1,\tilde{r}].
\end{equation}
Taking tensors over $\R$, define
\begin{equation}
\label{equation;rmatrix}
    \begin{split}
        \mathrm{r} = &\frac{1}{2}\sum_{j=1}^{\tilde r} (\sqrt{-1}X_j)\otimes X_j \\
        &+ \frac{1}{2}\sum_{\alpha\in R_+}\left(E_\alpha\otimes (\sqrt{-1}E_{-\alpha}) + (\sqrt{-1}E_\alpha)\otimes E_{-\alpha}-E_{-\alpha}\otimes \sqrt{-1}E_{-\alpha}+(\sqrt{-1}E_{-\alpha})\otimes E_{-\alpha}   \right).
    \end{split}
\end{equation}
Let $\pi_G = \mathrm{r}^L-\mathrm{r}^R$, where $\mathrm{r}^L$ and $\mathrm{r}^R$ denote the left and right-invariant tensor fields on $G$ satisfying $\mathrm{r}^L(e) = \mathrm{r}^R(e) = \mathrm{r}$.
%
%
%
If $G$ is viewed as a complex manifold and $\pi_G$ a section of the complexification of $\wedge^2 TG$, then the formula simplifies, taking tensors over $\C$: 
\begin{equation}
\label{equation;rmatrixholomorphic}
        \mathrm{r} = \frac{\sqrt{-1}}{2}\sum_{j=1}^{\tilde r} X_j\otimes X_j 
        + \sqrt{-1}\sum_{\alpha\in R_+}E_\alpha\otimes E_{-\alpha}.
\end{equation}

The dual Poisson Lie group $(AN_-,\pi_{K^*})$ is an anti-Poisson submanifold of $(G, \pi_G)$.
The following is immediate from the expression~\eqref{equation;rmatrix} for $\mathrm{r}$. Here and in what follows, we follow the convention that, for $X\in \mathfrak{b}_-$, and $f\in C^\infty(B_-,\C)$, the action of $X$ is
\[
(X\cdot f)(b) = \left.\frac{d}{dt}\right|_{t=0} f(\exp(-tX) b),\qquad (f\cdot X) (b) =\left. \frac{d}{dt}\right|_{t=0} f(b\exp(tX)).
\]
\begin{lemma}
\label{lemma;formulaformixedbrackets}
  Let $f$ and $g$ be complex valued functions defined on an open subset of $B_-$. 
   If $f$ is holomorphic and $g$ is anti-holomorphic, then
  \begin{equation*}
      \begin{split}
        \{f\vert_{K^*},g\vert_{K^*}\}_{K^*} = & \frac{\sqrt{-1}}{2}\sum_{j=1}^{\tilde r} \left((X_j\cdot f  )(X_j\cdot g)-(f\cdot X_j  )(g\cdot X_j) \right)\vert_{K^*}\\ &+\sqrt{-1}\sum_{\alpha\in R_+}\left((E_{-\alpha}\cdot f  )(E_{-\alpha}\cdot g)-(f\cdot E_{-\alpha}  )(g\cdot E_{-\alpha}) \right)  \vert_{K^*}.
      \end{split}
  \end{equation*}
\end{lemma}

\begin{proof}
Computing directly from the definition~\eqref{equation;rmatrix}, one finds for $b\in K^*$,
\begin{align*}
\{f|_{K^*},g|_{K^*}\}_{K^*}(b) & = -\{f,g\}_G(b) = \iota_{dg|_b}\iota_{df|_b}(\mathrm{r}b-b\mathrm{r}) \\
& = \left. \frac{d}{dt}\right|_{t=0} \frac{\sqrt{-1}}{2}\sum_{j=1}^{\tilde r} \left(f(\exp(tX_j) b)g(\exp(tX_j) b)- f(b\exp(tX_j)) g(b\exp(tX_j)) \right)
\\ &\quad +\sqrt{-1}\sum_{\alpha\in R_+} \left(f(\exp(tE_{-\alpha})b) g(\exp(t E_{-\alpha}) b)-f(b\exp(tE_{-\alpha})) g(b\exp(t E_{-\alpha}) ) \right).
\end{align*}
where in the last line we use that $df$ is $\C$-linear and $dg$ is $\C$-antilinear. The result then immediately follows.
\end{proof}

\subsubsection{Hamiltonian \texorpdfstring{$K$}{K}-manifolds with \texorpdfstring{$K^*$}{K*}-valued moment maps}

Let $(K,\pi)$ be a Poisson-Lie group and let $(M,\Pi)$ be a Poisson manifold. A smooth left action of $K$ on $M$ is a  \emph{Poisson action} if the action map $K\times M \to M$ is a Poisson map (with respect to the product Poisson structure on $K\times M$).  Let $\theta^R \in \Omega^1(AN_-,\a\nnn_-)$ denote the right invariant Maurer-Cartan form.

\begin{definition}[Hamiltonian $K$-manifold with $K^*$-valued moment map]\label{hamiltonian action with K* valued moment map definition} Let $s\neq 0$. A Poisson action of $(K,s\pi_K)$ on a symplectic  manifold $(M,\omega)$ is \emph{Hamiltonian with $K^*$-valued moment map} if there is a $K$-equivariant map $\Psi\colon M \to AN_-$ such that
\[
    \iota_{\underline{X}}\omega = -\Psi^*\llangle \theta^R,X\rrangle_s, \quad \forall X \in \kk.
\]
The map $\Psi$ is a \emph{$K^*$-valued moment map}. The tuple $(M,\Omega,\Psi)$ is a \emph{Hamiltonian $K$-manifold with $K^*$-valued moment map}.
\end{definition}
A $K^*$-valued moment map is a Poisson map to $(AN_-,s\pi_{K^*})$ \cite[Theorem 4.8]{Lu}.
Hamiltonian $K$-manifolds with $K^*$-valued moment maps can be defined more generally for Poisson actions of $K$ under the assumption that dual Poisson-Lie groups $K$ and $K^*$ admit a double \cite{Lu}. 


\subsubsection{Delinearization of Hamiltonian \texorpdfstring{$K$}{K}-manifolds} \label{GWsection}


Let $(K,s\pi_K)$ and $(AN_-,s\pi_{K^*})$ be as in Section \ref{compact poisson lie groups}. For all $s \in \R$, define a map \[
E_s \colon \kk^* \to AN_-\]
 as follows. Recall from~\eqref{equation;comparisonmap} the comparison map $\psi_\kk \colon \kk \to \kk^*$ given by the choice of form $(\cdot,\cdot)_\g$. For all $\lambda \in \ttt^*$, define $E_s(\lambda) := \exp(\frac{s\sqrt{-1}}{2}\psi\n_{\kk}(\lambda))$.  Extend the definition of $E_s$ to all of $\kk^*$ by $K$-equivariance with respect to the coadjoint action and the dressing action, i.e.\ 
\[
    E_s(\Ad_k^*\lambda) := E_s(\lambda)^k \quad \forall k \in K.
\]
For $s\neq 0$, $E_s$ is a $K$-equivariant diffeomorphism.

Let $\theta^L \in \Omega^1(AN_-,\a\nnn_-)$ denote the left invariant Maurer-Cartan form. Following \cite{A97}, for all $s\in\R$, define 
\begin{equation}\label{delinearization 2-form}
    \beta^s := \frac{1}{2\sqrt{-1}}H(E_s^*(\theta^L\wedge(\theta^L)^\dagger)_\g) \in \Omega^1(\kk^*)
\end{equation}
where $H\colon \Omega^\bullet(\kk^*) \to \Omega^{\bullet-1}(\kk^*)$ is the Poincar\'e homotopy operator defined by the homotopy $[0,1]\times \kk^* \to \kk^*$, $(t,\lambda) \mapsto t\lambda$, and $(\theta^L)^\dagger\in\Omega^1(AN_-,\a\nnn_+)$ denotes the anti-involution  $(\cdot)^\dagger$ applied to the values of $\theta^L$.  

\begin{lemma}\label{T-invariance of form} The 2-form $\beta_s$ is invariant with respect to the coadjoint action of $T$. \end{lemma}

\begin{proof}
    Since the homotopy used to define $H$ is $\Ad^*$-equivariant and the map $E_s$ is equivariant, it suffices to show that $(\theta^L\wedge(\theta^L)^\dagger)_\g$ is invariant with respect to the dressing action of $T$. 
    
The dressing action of $T$ on $AN_-$ coincides with the action by conjugation, i.e.\ $b^{t} = tbt\n$ for all $t\in T$ and $b \in AN_-$ (see Definition \ref{definition;dressing}). Let $C_t$ denote conjugation by $t$, and let $\Ad_t\theta^L$ denote the adjoint action of $t$ on the values of $\theta^L$. Then $C_t^*\theta^L = \Ad_t\theta^L$, and
    \[
        C_t^*(\theta^L \wedge (\theta^L)^\dagger)_\g = ((C_t^*\theta^L) \wedge (C_t^*\theta^L)^\dagger)_\g = ((\Ad_t\theta^L) \wedge (\Ad_t\theta^L)^\dagger)_\g.
    \]
    A short calculation 
    shows that $((\Ad_t\theta^L) \wedge (\Ad_t\theta^L)^\dagger)_\g = ((\theta^L) \wedge (\theta^L)^\dagger)_\g$. 
    \end{proof}

Let $M$ be a smooth manifold equipped with a smooth left action of $K$. Given a symplectic form $\omega \in \Omega^2(M)$ and a map $\mu\colon M \to \kk^*$, define
\begin{equation}\label{linearization equation}
    \Psi^s := E_s \circ\mu, \quad \Omega^s := \omega +\mu^*d\beta^s.
\end{equation}
For all $s\neq 0$ the forms $\Omega^s$ are symplectic \cite{A97}. 

\begin{theorem}[Linearization/Delinearization Theorem \cite{A97}] \label{theorem;delin}
    Let $M$ be a smooth manifold equipped with a smooth left action of $K$. Let $\omega,\Omega^s\in \Omega^2(M)$ and $\Psi^s,\mu$ be related by \eqref{linearization equation}. Then, for all $s\neq 0$, $(M,\Omega^s,\Psi^s)$ is a Hamiltonian $(K,s\pi_K)$-manifold with $(AN_-,s\pi_{K^*})$-valued moment map if and only if $(M,\omega,\mu)$ is a Hamiltonian $K$-manifold.
\end{theorem}

Let $\pi_{\kk^*}$ denote the canonical Poisson structure on $\kk^*$ and let $\pi_{\kk^*}^{\sigma^s}$ denote the gauge transformation\footnote{See e.g.~\cite{Bursztyn} for more details on gauge transformations of Poisson structures.}  of  $\pi_{\kk^*}$ by $\sigma^s = d\beta^s$. 
Define a vector field $Y^s$ on $\kk^*$ by the equation
\begin{equation}\label{Poisson delinearization moser vf}
	Y^s = (\pi_{\kk^*}^{\sigma^s})^\sharp\left(-\frac{\partial}{\partial s} \beta^s\right)
\end{equation}
and let $\psi_s$ denote the flow of $Y^s$.  By Moser's trick for Poisson manifolds, $(\psi_s)_*\pi_{\kk^*} = \pi_{\kk^*}^{\sigma^s}$ for all $s$ such that $\psi_s$ is defined. 
On any Hamiltonian $K$-manifold $(M,\omega,\mu)$, define a Moser vector field $X^s$ by the equation
\begin{equation}\label{delinearization moser vf}
    \iota_{X_s}\Omega^s = -\frac{\partial}{\partial s}\mu^*\beta^s, \quad s \in \R.
\end{equation}
and let $\phi_s$ denote the flow of $X_s$. By Moser's trick, $\phi_s^*\Omega^s = \omega$ for all $s$ such that $\phi_s$ is defined.

Let $\wt\colon AN_- \to A$ be given by $\wt(x) = [x]_0$. If $(M,\Omega,\Psi)$ is a Hamiltonian $(K,s\pi_K)$-manifold with $K^*=AN_-$-valued moment map, then $(M,\Omega,E_s\n\circ \wt\circ \Psi)$ is a Hamiltonian $T$-manifold, where $T$ acts as the maximal torus of $K$. 

\begin{proposition}\label{moser flow equivariance}
    The flows $\psi_s$ and $\phi_s$ have the following properties (valid for all $s$ such that  $\psi_s$ and $\phi_s$ are defined).
    \begin{enumerate}[(i)]
    	\item $\pi_{\kk^*}^{\sigma^s} = (E_s\n)_*s\pi_{K^*}$.
    	\item The flow $\psi_s$ preserves coadjoint orbits. In particular, $\psi_s$ is defined for all $s$. 
    	\item $\psi_s\circ \mu = \mu \circ \phi_s$.
        \item Both $\psi_s$ and $\phi_s$ are equivariant with respect to the action of $T$ as the maximal torus of $K$.
        \item The flow $\phi_s$ is equivariant with respect to the Thimm torus action defined on the dense subset $U\subset M$ (see Section \ref{thimm torus background section}).
        
        \item The following diagrams commute.
        \begin{equation}
    \begin{tikzcd}
        M \ar[d,"\mathcal{S}\circ \mu"] \ar[r,"\phi_s"] &M \ar[d,"\mathcal{S}\circ\mu"] \\
         \ttt^*_+   & \ttt^*_+ \ar[l,"= "]\\
    \end{tikzcd}\quad \quad 
    \begin{tikzcd}
        M \ar[d,"\mu"] \ar[r,"\phi_s"] & M \ar[d,"\Psi^s"] \\
        \kk^* \ar[d,"\pr_{\ttt^*}"] & K^*\ar[d,"\wt"]  \\
        \ttt^* \ar[r,"E_s"] & A  \\
    \end{tikzcd} 
\end{equation}
    \end{enumerate}
\end{proposition}

\begin{proof} 
    \noindent (i) This is \cite[Theorem 1]{A97} (see also the remark at the end of \cite{A97}).
    
    \noindent (ii) The first claim follows from (i) and the definition of $Y^s$. The second claim follows by the Escape Lemma since coadjoint orbits are compact.
    
    \noindent (iii) It suffices to prove that $\mu_*X^s = Y^s$. Letting $(\Omega^s)\n$ denote the Poisson structure on $M$ defined by the symplectic structure $\Omega^s$, observe that \eqref{delinearization moser vf} can be re-written equivalently as 
    \[
    	X^s = ((\Omega^s)\n)^\sharp\left(-\frac{\partial}{\partial s}\mu^*\beta^s\right).
    \]
    Since $\Psi^s = E_s \circ \mu$ is a Poisson map with respect to $s\pi_{K^*}$, it follows by (i) that $\mu = E_s\n \circ \Psi^s$ is a Poisson map with respect to $\pi_{\kk^*}^{\sigma^s}$. Thus
    \[
    	\mu_*X^s = (\pi_{\kk^*}^{\sigma^s})^\sharp\left(-\frac{\partial}{\partial s} \beta^s\right) = Y^s.
    \]
    
    \noindent (iv) Follows by Lemma~\ref{T-invariance of form} along with Equations \eqref{linearization equation}, \eqref{Poisson delinearization moser vf}, \eqref{delinearization moser vf}, and equivariance of $\mu$.
    
    \noindent (v) The map $\mu$ is equivariant as a map from $U$ to $\kk^*$ with respect to the Thimm torus action of $T$ on $U$ and the trivial action of $T$ on $\kk^*$. Since $\beta^s$ is invariant with respect to the trivial action of $T$ on $\kk^*$, it follows by the same argument as (iv) that $\phi_s$ is equivariant with respect to the Thimm torus action of $T$ on $U$.
    
    \noindent (vi) The first diagram is a direct consequence of (iii) since $\mathcal{S}$ is constant on coadjoint orbits. 
    
    The action of $T$ on $(M,\omega)$ as the maximal torus of $K$ is Hamiltonian with moment map $\pr_{\ttt^*}\circ \mu$. The action of $T$ on $(M,\Omega^s)$ as the maximal torus of $K$ is Hamiltonian with moment map $E_s\n\circ \wt\circ \Psi^s$. Since $M$ is connected and  $\phi_s\colon (M,\omega) \to (M,\Omega^s)$ is a $T$-equivariant symplectomorphism, it follows by uniqueness of moment maps that there is an element $\xi \in \ttt^*$ such that 
    \[
        E_s\n \circ \wt \circ \Psi^s \circ \phi_s = \pr_{\ttt^*} \circ \mu + \xi.
    \]
    
    The set $\mu\n(\ttt^*)$ is non-empty by $K$-equivariance of $\mu$. Let $m\in \mu\n(\ttt^*)$. Then it follows by \eqref{linearization equation} and (iii) that
    \begin{equation*}
        \begin{split}
            E_s\n \circ \wt \circ \Psi^s \circ \phi_s(m) & = E_s\n \circ \wt \circ E_s\circ \mu(\phi_s(m))\\
            &=E_s\n \circ \wt \circ E_s\circ \psi_s(\mu(m)).
            \end{split}
            \end{equation*}
             By (ii) and (iv), $\psi_s$ fixes elements of $\ttt^*$. Thus
             \begin{equation*}
             \begin{split}
             E_s\n \circ \wt \circ \Psi^s \circ \phi_s(m) & = E_s\n \circ \wt \circ E_s\circ \mu(m) \\
             & =\mu(m) = \pr_{\ttt^*}\circ \mu(m),
             \end{split}
             \end{equation*}
    and so $\xi =0$. Thus the second diagram commutes.
\end{proof}

The flow $\phi_s$ is defined for all $s$ if $M$ is compact. However, one does not need such a strong assumption.

\begin{corollary}\label{completeness of moser flow}
    Let $(M,\omega,\mu)$ be a Hamiltonian $K$-manifold such that the connected components of the fibers of $\mu$ are compact. Then, the symplectic forms $\omega$ and $\Omega^s$ are isotopic for all $s\neq 0$. In particular, if $(M,\omega,\mu)$ is multiplicity free, then  $\omega$ and $\Omega^s$ are isotopic for all $s$.
\end{corollary}

\begin{proof}
    Since $\mu$ is $K$-equivariant and $K$ is compact,  connected components of fibers of $\mu$ are compact if and only if the connected components of $\mu\n(\mathcal{O})$ are compact for every coadjoint orbit $\mathcal{O}$. 
    Let $m\in M$. By Proposition \ref{moser flow equivariance} (ii) and (iii), the flow $\phi_s(m)$ is contained in a connected component of  $\mu\n(K\cdot \mu(m))$. The result then follows by the Escape Lemma for flows of smooth vector fields. The final claim follows since $K$ acts transitively on connected components of $\mu\n(\mathcal{O})$ when $(M,\omega,\mu)$ is multiplicity free.
\end{proof}

\begin{remark} \label{GWremark}
Using a similar argument, one can show that $E_1\circ \psi_1\colon(\mathfrak{k}^*,\pi_{\kk^*}) \to (K^*,\pi_{K^*})$ is well defined. It is therefore a Ginzburg-Weinstein isomorphism in the sense of Theorem~\ref{thm:intro_gw}.
\end{remark}

\begin{example}\label{example; delinearization of T^*K cross section}
Recall the Hamiltonian $K\times T$-manifold $(K \times \mathring{\ttt}_+^*,\omega_{\rm{can}},(\mu_L,\mu_R) )$ from Example~\ref{T^*K cross section example}. Since $T$ is a torus, the Poisson-Lie group structure $s\pi_{K\times T}$ defined in Section~\ref{compact poisson lie groups} equals $s\pi_K$  and the dual Poisson-Lie group is $(AN_- \times A, s\pi_{K^*})$. The map $E_s$ for $K\times T$ is 
\begin{equation}
    \left(E_s,\exp\left(\frac{s\sqrt{-1}}{2}\psi_\kk\n\right)\right)\colon \kk^* \times \ttt^* \to AN_-\times A
\end{equation}
Since $T$ is a torus, the 2-form $\beta^s \in \Omega^2(\kk^* \times \ttt^*)$ coincides with the pullback of the 2-form $\beta^s \in \Omega^2(\kk^*)$ under projection $\kk^* \times \ttt^* \to \kk^*$.  Thus, 
\begin{equation}
    \Psi^s = \left(E_s\circ \mu_L,\exp\left(\frac{s\sqrt{-1}}{2}\psi_\kk\n\circ \mu_R\right)\right), \quad \Omega^s = \omega_{\rm{can}}+ \mu_L^*d\beta^s.
\end{equation}
The action of $T$ on $K \times \mathring{\ttt}_+^*$ coincides with the action of the Thimm torus of $K$. It follows by Proposition~\ref{moser flow equivariance} (ii), (iii), and (v) that 
 the action of $T$ on $(K \times \mathring{\ttt}_+^*,\Omega^s)$ is also Hamiltonian with moment map $\mu_R$.

An explicit formula for the Poisson structure defined by $\Omega^s$ can be given as follows. Fix the same basis of $\g$ consisting of elements $X_j \in \h$ and root vectors $E_{\alpha} \in \g_{\alpha}$ as in Section~\ref{subsubsection;explicitformula}. Let $P_j \in\h^*$ denote the basis dual to $X_j$. Then, as a section of $\wedge^2 T^\C(K\times \mathring{\ttt}_+^*)$,
\begin{equation}\label{equation; delinearization of T^*K cross section}
	    (\Omega^s)\n = \sum_{\alpha\in R_+}\frac{2s\sqrt{-1}(E_\alpha\wedge E_{-\alpha})^L}{\exp(2s\sqrt{-1}\langle\xi,[E_\alpha, E_{-\alpha}]\rangle)-1} + \sum_j X_j^L\wedge P_j + s\pi_K  
\end{equation}
where for any $X \in \wedge^k\g$, $X^L$ denotes the left-invariant multivector field on $K$ whose value at $e$ equals $X$. This formula follows by combining the explicit formula for $s\pi_{K^*}$ given in \cite[Proposition 5.12]{Lu00}, the fact that $(\Psi^s)_*(\Omega^s)\n = s\pi_{K^*}$, and the definition of $\Omega^s$.
\end{example}

\subsubsection{Legendre transforms on \texorpdfstring{$K^*$}{K*}}\label{section;legendre transform}

The \emph{Legendre transform} of $f \in C^\infty(K^*)$ is the map  $\mathfrak{L}_f \colon K^* \to \kk$ uniquely defined by the equation 
\begin{equation}\label{definition;legendre transform}
    \llangle \theta^R ,\mathfrak{L}_f\rrangle_s = df.
\end{equation}
Note that if $\beta \in \Omega^1(K^*)$, then $\mathfrak{L}_\beta$ is defined as in \eqref{definition;legendre transform} by replacing $\mathfrak{L}_f$ with $\mathfrak{L}_\beta$ and $df$ with $\beta$. This is the analogue for $K^*$-valued moment maps of the Legendre transform on $\kk^*$ used in \cite{GS5,GS1} to study Hamiltonian flows of collective  functions  (i.e.~functions of the form $f\circ \mu$ where $\mu$ is a moment map). The main property of Legendre transforms on $\kk^*$ generalizes directly to the setting of Hamiltonian spaces with $K^*$-valued moment maps as follows.
\begin{proposition}\label{proposition;legendre transform}
If $(M,\Omega,\Psi)$ is Hamiltonian $K$-manifold with $K^*$-valued moment map, then
\begin{equation}
    (X_{f\circ \Psi})_m = \underline{\mathfrak{L}_f(\Psi(m))}_m, \quad \forall m \in M.
\end{equation}
\end{proposition}
\begin{proof}
By \eqref{definition;legendre transform}, the moment map equation, and the definition of $X_{f\circ \Psi}$: for all $m \in M$,
\begin{equation*}
    -\iota_{\underline{\mathfrak{L}_f(\Psi(m))}_m}\Omega = (\Psi^*\llangle \theta^R ,\mathfrak{L}_f\rrangle_s)_m = (\Psi^*df)_m = -\iota_{(X_{f\circ \Psi})_m}\Omega. 
\end{equation*}
The proposition then follows by non-degeneracy of $\Omega$.
\end{proof}

\subsection{Cluster algebras}\label{Cluster Varieties}

In this section, we first recall some basic definitions from cluster theory. Secondly, we recall the construction of the cluster algebra structure on the double Bruhat cell $G^{u,v}$. At the end, we recall the homogeneity of the cluster algebra structure on $G^{u,v}$. In the rest of the paper, we mainly focus on the cell $G^{w_0,e }$. More detail can be found in \cite{BFZ,GSV}.

\subsubsection{Seeds and cluster mutations} 

\begin{definition}\label{seed}
 	A \emph{seed} $\sigma=(I,J,M)$ consists of a finite set $I$, a subset $J\subset I$ and an integer matrix $M=\left[ M_{ij} \right]_{i,j\in I}$ which is skew-symmetrizable, i.e.~there exists an $I\times I$ diagonal matrix $\bm{D}$ with positive integer entries, called a skew-symmetrizer, such that $M\bm{D}=-(M\bm{D})^T$. The \emph{principal part} of $M$ is given by $M_0=\left[ M_{ij} \right]_{i,j\in J}$.

Consider the ambient field $\mathcal{F}$ of rational functions over $\mathbb{C}$ in $|I|$ independent variables. A \emph{labeled seed} in $\mathcal{F}$ is a pair $({\bm z}_\sigma, \sigma)$ where $\sigma$ is a seed and ${\bm z}_\sigma:=\{z_i \in \mathcal{F} \mid i\in I\}$ is a set of $|I|$ elements forming a free generating set of $\mathcal{F}$. We refer to ${\bm z}_\sigma$ as a \emph{cluster} and the $z_i$ as \emph{cluster variables}. Denote ${\bm z}_\sigma^0:=\{z_i\mid i\in J\}$, which is set of \emph{unfrozen variables}. 
\end{definition} 

If $(I,J,M)$ is a seed, the mutation $\mu_k(M)$ of $M$ in direction $k\in J$  is the $I\times I$ matrix with entries:
\[
  \mu_k(M)_{ij}=\left\{
  \begin{aligned}
    \ &-M_{ij}, & &\text{~if~}\ k\in\{i,j\};\\
    \ &M_{ij}+\frac{1}{2}\Big(|M_{ik}|M_{kj}+M_{ik}|M_{kj}|\Big), & &\text{~otherwise~}.
  \end{aligned}\right.
\]
If $\bm{D}$ is a skew-symmetrizer of $M$, then $\bm{D}$ is a skew-symmetrizer of $\mu_k(M)$.  A \emph{seed mutation} in \emph{direction} $k\in J$ transforms a labeled seed $({\bm z}_\sigma, \sigma)$ into a new labeled seed $\mu_k({\bm z}_\sigma, \sigma)=({\bm z}_{\sigma'}, \sigma')$, where $\sigma'=(I,J,\mu_k(M))$ and the new cluster ${\bm z}_{\sigma'}\subset \mathcal{F}$ contains the cluster variables:
\begin{equation} \label{equation;mutationvariable}
  \mu_k(z_i)=\left\{
  \begin{aligned}
    \ & z_i, & &\text{~if~}\ i\neq k;\\
    \ & z_k^{-1}\left(\prod_{M_{jk}>0}z_j^{M_{jk}}+\prod_{M_{jk}<0}z_j^{-M_{jk}}\right), & &\text{~if~}\ i=k.
  \end{aligned}\right.
\end{equation}
Two seeds will be called \emph{mutation equivalent} if they are related by a sequence of mutations. The equivalence class of a labeled seed $({\bm z}_\sigma,\sigma)$ is denoted by $|\sigma|$. 

\begin{definition}
Let $(\bm{z}_\sigma,\sigma)$ be a labeled seed. Let $\mathcal{P}$ be the $\C$-algebra generated by $\{z_i, z_i^{-1} \mid i\in I\backslash J\}$. The \emph{cluster algebra} $\mathcal{A}_{|\sigma|}$ over $\mathcal{P}$ associated with the labeled seed $({\bm z}_\sigma,\sigma)$ is the $\mathcal{P}$-subalgebra of  $\mathcal{F}$ generated by the set \[
\{z \in \mathcal{F} \mid z\in {\bm z}_{\sigma'} ^0,~\sigma'\in |\sigma|\}\]
 of all unfrozen cluster variables. In this case $({\bm z}_\sigma,\sigma)$ is the \emph{initial seed} of $\mathcal{A}_{|\sigma|}$.
\end{definition}

\begin{definition}[\cite{FG}]
  The {\em (Langlands) dual seed} of a seed $\sigma=(I,J,M)$ is $\sigma^\vee:=(I,J,-M^T)$. If $(\bm{z}_\sigma,\sigma)$ is a labeled seed, let $\mathcal{F}^\vee$ be the ambient field of rational functions over $\C$, in the variables $z_i^\vee$, $i\in I$. The \emph{dual labeled seed} is $(\bm{z}^\vee_\sigma,\sigma^\vee)$, where $\bm{z}^\vee_\sigma = \{z_i^\vee \in \mathcal{F}^\vee \mid i \in I\}$. For a labeled seed $(\bm{z}_\sigma,\sigma)$, denote by $\mathcal{A}^\vee_{|\sigma|}:=\mathcal{A}_{|\sigma^\vee|}$ the cluster algebra associated with $(\bm{z}^\vee_{\sigma},\sigma^\vee)$, which we call the {\em (Langlands) dual cluster algebra} of $\mathcal{A}_{\sigma}$. 
\end{definition}

If $\mu$ is a sequence of mutations at indices $j_1,\dots,j_n\in J$ which takes a labeled seed $(\bm{z},\sigma)$ to $(\bm{z}',\sigma')$, let $\mu^\vee$ denote the sequence of mutations at indices $j_1,\dots,j_n$ applied to the dual seed $(\bm{z}^\vee,\sigma^\vee)$.
It is straightforward to verify that $\mu^\vee$ takes the dual seed $(\bm{z}^\vee,\sigma^\vee)$ to $({{\bm z}'}^\vee,{\sigma'}^\vee)$.

Let $(\bm{z}_\sigma,\sigma=(I,J,M))$ be a labeled seed, and assume a skew-symmetrizer $\bm{D}$ of $M$ has been fixed, with diagonal entries $\bm{D}_{ii} = d_i$. 
For each labeled seed $(\bm{z}_{\sigma'},\sigma') \in |\sigma|$, define a $\C$-algebra homomorphism 
\begin{equation}\label{comparisongeneral}
\varPsi^*_{\sigma'}\colon \mathcal{A}^\vee_{|\sigma|} \to \mathcal{A}_{|\sigma|}, \quad z_i^\vee  \mapsto z_i^{d_{i}},\quad z_i^\vee \in \bm{z}_{\sigma'}.
\end{equation}
Note that $\varPsi_{\sigma'}\ne \varPsi_{\sigma}$ in general. 

%

\subsubsection{Cluster algebra structures on double Bruhat cells}\label{doublecellcluster}

In this section, we assume that $G$ is a connected reductive algebraic group of rank $r$ as before. We will recall how to construct a 
 cluster algebra structure on the coordinate algebra of a reduced double Bruhat cell $L^{u,v}$ and double Bruhat cell $G^{u,v}$, for any pair $(u,v) \in W\times W$.

A \emph{double reduced word} $\mathbf{i}=(i_1,\dots,i_n)$ for $(u,v)$ is a shuffle  of a reduced word for $u$, written in the alphabet $\{-r,\dots,-1\}$, and a reduced word for $v$, written in the alphabet $\{1,\dots,r\}$, where $n=\ell(u)+\ell(v)$. For $k\in -[1,r]\cup[1,n]$, we denote by
\begin{equation}\label{k+}
    k^+=\min\{j\mid j>k,|i_j|=|i_k|\}.
\end{equation}
If $|i_j|\neq |i_k|$ for all $j>k$, we set $k^+=n+1$. An index $k$ is \emph{$\mathbf{i}$-exchangeable} if both $k,k^+\in [1,n]$. Let $\bm{e}(\mathbf{i})$ denote the set of all $\mathbf{i}$-exchangeable indices.

Fix a double reduced word $\mathbf{i}$ of $(u,v)$. Let $I=[-r,-1]\cup [1,n]$, $J=\bm{e}(\mathbf{i})$, and $L:=[-r,-1]\cup\bm{e}(\mathbf{i})\subset I$. Construct a $I \times I$ matrix $M(\mathbf{i})$ as in \cite[Remark 2.4]{BFZ}: For $k,l\in I$, set $p=\max\{k,l\}$ and $q=\min\{k^+,l^+\}$, and let $\epsilon(k)$ be the sign of $k$. Then following \cite[Remark 2.4]{BFZ}, let
\begin{equation}\label{matixinseed}
  \begin{split}
    M(\mathbf{i})_{kl}=\left\{
    \begin{aligned}
      &-\epsilon(k-l)\cdot\epsilon(i_p),& &\text{~if~} p=q;\\
      &-\epsilon(k-l)\cdot\epsilon(i_p)\cdot A_{|i_k|,|i_l|},& &\text{~if~} p<q \text{~and~} \epsilon(i_p)\epsilon(i_q)(k-l)(k^+-l^+)>0;\\
      &0, & &\text{~otherwise~}. 
    \end{aligned}\right.
  \end{split}
\end{equation}
Here and throughout we use the convention that $i_{-k} = k$, for $k\in [1,r]$. Recall that $A$ is the Cartan matrix of $\mathfrak{g}$. Denote by $\overline{M}(\mathbf{i}):=[M(\mathbf{i})_{kl}]_{k,l\in L}$ the $L\times L$ submatrix of $M(\mathbf{i})$. 

Recall the matrix $D$, which is the symmetrizer of the Cartan matrix $A$ fixed in Section~\ref{section;comparison}. Consider the diagonal $I \times I$ matrix ${\bm D}$ with entries
\begin{equation} \label{equation;skewsymmetrizer}
{\bm D}_{j,j}= D_{|i_j|,|i_j|} = d_{|i_j|}\in \Z.
\end{equation}
Let $\overline{\bm D}$ be the submatrix of ${\bm D}$ with rows and columns indexed by $L$. Then $M(\mathbf{i})$ is skew-symmetrizable, with skew-symmetrizer ${\bm D}$, and $\overline{M}(\mathbf{i})$ is skew-symmetrizable, with skew-symmetrizer $\overline{\bm D}$.

Define the following seeds:
\[
  \sigma(\mathbf{i}):=(I,J,M(\mathbf{i})),\quad \overline{\sigma}(\mathbf{i}):=(L,J,\overline{M}(\mathbf{i})).
\]
If $\sigma$ is obtained by mutating $\sigma(\mathbf{i})$ at indices $j_1,\dots,j_n$, let $\overline{\sigma}$ denote the seed obtained by mutating $\overline{\sigma}(\mathbf{i})$ at the same indices.
For a double reduced word $\mathbf{i}$ of $(u,v)$ and $k\in [1,n]$, denote
\[
  u_k:=\prod_{\substack{l=1,\ldots,k\\ i_l<0}} s_{|i_l|}, \quad v_k:=\prod_{\substack{l=n,\ldots,k+1\\ i_l>0}} s_{i_l},
\]
where the index is increasing (resp. decreasing) in the product on the left (resp. right). Denote
\begin{equation}\label{fullcluster}
  \Delta_k:=\Delta_{u_k\omega_{|i_k|},v_k\omega_{|i_k|}}\text{~for~}k\in [1,n]; \quad \Delta_k:=\Delta_{\omega_{|k|},v^{-1}\omega_{|k|}}\text{~for~}k\in[-r,-1].  
\end{equation}

\begin{theorem}\cite[Theorem 2.10]{BFZ}\label{Thm2.10} 
  For every double reduced word $\mathbf{i}$ for $(u,v)$, the map
  \[
    \varphi^{u,v}\colon \mathcal{A}_{|\overline{\sigma}(\mathbf{i})|}\to \mathbb{C}[L^{u,v}], \quad z_k\mapsto \Delta_k\big|_{L^{u,v}}, \ \text{~for~}\   k \in L
  \]
  is an isomorphism of algebras. If $G$ is simply connected, the map
  \[
    \phi^{u,v}\colon \mathcal{A}_{|\sigma(\mathbf{i})|}\to \mathbb{C}[G^{u,v}], \quad z_k\mapsto \Delta_k, \ \text{~for~}\   k \in I
  \]
  is an isomorphism of algebras. Here $z_k\in \bm{z}_{\overline{\sigma}(\mathbf{i})}$ (resp. $z_k\in\bm{z}_{\sigma(\mathbf{i})}$) is in the initial seed of $ \mathcal{A}_{|\overline{\sigma}(\mathbf{i})|}$ (resp. $ \mathcal{A}_{|\overline{\sigma}(\mathbf{i})|}$).
\end{theorem}

\begin{remark}
Theorem 2.10 in~\cite{BFZ} is stated for non-reduced double Bruhat cells; the statement for reduced double Bruhat cells follows by passing to a cover as in~\eqref{equation;coverL} and specializing the frozen variables $\Delta_{w_0\omega_i,\omega_i}=1$. Strictly speaking, in~\cite{BFZ} they show there is an isomorphism from $\C[G^{w_0,e}]$ to the \emph{upper} cluster algebra of $\sigma(\mathbf{i})$. We ignore this detail here for two reasons: First, it is known that in the case of double Bruhat cells, the cluster algebra and upper cluster algebra coincide (see for instance the introduction of~\cite{GY}). Second, we only make use of the Laurent phenomenon of  $\mathcal{A}_{|\sigma(\mathbf{i})|}$ in what follows, and so all our results hold if ``cluster'' is replaced by ``upper cluster''.
\end{remark}

Let us apply Theorem~\ref{Thm2.10} to the Langlands dual group $G^\vee$ of $G$. Let $\mathbf{i}$ be a double reduced word for $(u,v)$ as before. Let $M^\vee(\mathbf{i})$ be the matrix defined in \eqref{matixinseed}, with respect to the Cartan matrix $A^\vee=A^T$ of $G^\vee$. Denote by $\overline{M}^\vee(\mathbf{i})$ the  submatrix of $M^\vee(\mathbf{i})$ formed by taking the $[-r,-1]\cup {\bm e}(\mathbf{i})$ rows and columns. Direct computation shows that 
\[
\overline{M}^\vee(\mathbf{i})=-\overline{M}^T(\mathbf{i}).
\] Thus $\C[L^{\vee;u,v}]$ is the Langlands dual cluster algebra to $\C[L^{u,v}]$.

\subsubsection{Homogeneous cluster algebras}  \label{section;homogeneous}

In this section, we recall the notion of homogeneous cluster algebras. 
For a more detailed discussion, see \cite{GSV,Gra}.
\begin{definition}
  A cluster algebra $\mathcal{A}_{|\sigma|}$ with initial seed $({\bm z}_\sigma,\sigma)$ is \emph{graded} by an abelian group $\mathcal{G}$ if the algebra $\mathbb{C}[{\bm z}_\sigma^{\pm}]$ is graded by $\mathcal{G}$ and the initial cluster variables $z_i$ are homogeneous for $i\in I$. We say a graded cluster algebra is \emph{homogeneous} if all cluster variables are homogeneous with respect to the grading. Denote by $|z|$ the degree of a homogeneous element $z\in \mathbb{C}[{\bm z}_\sigma^{\pm}]$. 
\end{definition}

\begin{proposition}  \label{proposition; homogeneous cluster creterion}
  A $\mathcal{G}$-graded cluster algebra $\mathcal{A}_{|\sigma|}$ with initial seed $({\bm z}_\sigma, \sigma=(I,J,M))$ is homogeneous if and only if
  \begin{equation}\label{degreeaskernel}
    \sum_{i\in I} |z_i|M_{ij}=0, \quad \forall j\in J,
  \end{equation}
  where $z_i\in {\bm z}_\sigma$ and $|z_i|$ is the degree of $z_i$.
\end{proposition}

\begin{proof}
  If $\mathcal{A}_{|\sigma|}$ is homogeneous,  the equation \eqref{degreeaskernel} follows from the fact that the cluster variables $z_k'\in {\bm z}_{\sigma'}$ of the seed $\sigma'=\mu_k(\sigma)$ are homogeneous. To be more precise, the variable $z_k'$ is homogeneous if and only if the monomials in \eqref{equation;mutationvariable} have the same degree. Then we have:
  \[
  \sum_{j: M_{jk}>0}|z_j|M_{jk}=-\sum_{j: M_{jk}<0}|z_j|M_{jk},
  \]
  which is equivalent to \eqref{degreeaskernel}.
  
  For the other direction, by induction, all we need to show is
  \begin{equation}\label{mhc}
    \sum_{i\in I} |\mu_k(z_i)|\mu_k(M)_{ij}=0, \quad \forall k, j\in J.
  \end{equation}
  First all, note that $\mu_k(z_k)$ has degree:
  \[
    |\mu_k(z_k)|=-|z_k|+\frac{1}{2}\sum_{i\in I}|z_i||M_{ik}|.
  \]
  Then for $j\neq k$, noting that $M_{kk}=0$, we have:
  \begin{align*}
    2\sum_{i\in I} |\mu_k(z_i)|\mu_k(M)_{ij}&=2\sum_{i\neq k} |z_i|\mu_k(M)_{ij}+2|\mu_k(z_k)|\mu_k(M)_{kj}\\
    &=\sum_{i\neq k} |z_i|\left(2M_{ij}+|M_{ik}|M_{kj}+M_{ik}|M_{kj}|\right)-\left( \sum_{i\in I}|z_i||M_{ik}|-2|z_k|\right)M_{kj}\\
    &=\sum_{i\neq k} |z_i|M_{ik}|M_{kj}|=|M_{kj}|\sum_{i\in I} |z_i|M_{ik} \\
    & =0.
  \end{align*}
  The relation~\eqref{degreeaskernel} is used when moving from the second line to the third, and again moving from the third line to the fourth. 
  For $j=k$, we have
  \begin{align*}
    \sum_{i\in I} |\mu_k(z_i)|\mu_k(M)_{ik}&=-\sum_{i\neq k} |z_i|M_{ik}=0.
  \end{align*}
  Thus we get \eqref{mhc}.
\end{proof}

\subsubsection{Homogeneity of cluster variables on double Bruhat cells}

Let $G$ be a simply connected semisimple Lie group. Consider the action of $H\times H$ on $\C[G]$, where $(h,h') \cdot f(x) = f(hxh')$ for $f\in \C[G]$ and $(h,h') \in H\times H$. Then $\C[G^{u,v}]$ has a natural $P\times P$-grading, where the $P\times P$-homogenous elements are $H\times H$-eigenvectors in $ \C[G^{u,v}]$. 

Thus for a fixed double reduced word $\mathbf{i}$ of $(u,v)\in W\times W$,  the generalized minor $\Delta_{\gamma,\delta}$ naturally has a degree $(\gamma,\delta)\in P\times P$: 
\[
    |\Delta_k|:=(u_k\omega_{|i_k|},v_k\omega_{|i_k|})\text{~for~}k\in [1,n]; \quad |\Delta_k|:=(\omega_{|k|},v^{-1}\omega_{|k|})\text{~for~}k\in[-r,-1].
\]
\begin{proposition}\cite[Lemma 4.22]{GSV}
  Let $({\bm z}_{\sigma(\mathbf{i})},\sigma(\mathbf{i})=(I,J={\bm e}(\mathbf{i}), M(\mathbf{i})))$ be the initial seed as in Section~\ref{doublecellcluster}. For any $k\in J$, we have:
  \[
    \sum_j u_j\omega_{|i_j|}M(\mathbf{i})_{jk}=0, \quad \sum_j v_j\omega_{|i_j|}M(\mathbf{i})_{jk}=0.
  \]
\end{proposition}
By Theorem \ref{Thm2.10}, we define a grading $|\cdot|$ for the cluster algebra $\mathcal{A}_{|\sigma(\mathbf{i})|}$ by:
\[
    |z_k|:=|\Delta_k|\text{~for~}k\in [-r,-1]\cup [1,n].
\]
Thus it follows immediately: 
\begin{corollary}\label{corollary;homog}
	Let $\mathbf{i}$ be a double reduced word of $(u,v)$ and $({\bm z}_\sigma, \sigma(\mathbf{i})=(I,J, M(\mathbf{i})))$ be the initial seed as in Section~\ref{doublecellcluster}. The $P\times P$-graded cluster algebra $(\mathcal{A}_{|\sigma(\mathbf{i})|}, |\cdot|)$ is homogeneous.
\end{corollary}

Now assume $(u,v) = (w_0,e)$. Recall that following~\eqref{twist map equation} we defined the twisted minors
\[
  \Delta^\zeta_{u\omega_i,\omega}:=\Delta_{u\omega_i,\omega_i}\circ \zeta,
\]
Note that the twisted minors $\Delta^\zeta_{u\omega_i,\omega_i}$'s are homogeneous:
\begin{proposition}
\label{proposition;twisthomog}
Let $f\in \C[G^{w_0,e}]$ be homogenous of degree $(\gamma,\delta)\in P\times P$. Then $f\circ \zeta$ is homogenous of degree $(-w_0 \gamma, -\delta)$. In particular, the function $\Delta^\zeta_{u\omega_i,\omega_i}$ is of degree $(-w_0u\omega_i,-\omega_i)$.
%
%
\end{proposition}
\begin{proof}
 By the uniqueness of the Gauss decomposition, we have for all $g\in G_0$ and $h\in H$:
  \begin{align*}
    [hg]_{\geqslant0}&=h[g]_{\geqslant 0}, \text{~since~} hg=[hg]_-[hg]_{\geqslant 0}=h[g]_-h^{-1}\cdot h[g]_{\geqslant 0};\\
    [gh]_{\geqslant0}&=[g]_{\geqslant 0}h, \text{~since~} gh=[gh]_-[gh]_{\geqslant 0}=[g]_-\cdot [g]_{\geqslant 0}h.
  \end{align*}
  Thus one computes for $x\in G^{w_0,e}$ and $h\in H$,
  \begin{align*}
    f\circ\zeta (hx)&=f([\overline{w_0}^{-1}hx]_{\geqslant 0}^\theta)=f((\overline{w_0}^{-1}h\overline{w_0})^{-1}[\overline{w_0}^{-1}x]_{\geqslant 0}^\theta)=h^{-w_0\gamma}(f\circ\zeta)(x);\\
   f\circ \zeta (xh)&=f([\overline{w_0}^{-1}xh]_{\geqslant 0}^\theta)=f([\overline{w_0}^{-1}x]_{\geqslant 0}^\theta\cdot h^{-1})=f\circ \zeta (x)h^{-\delta},
  \end{align*}
  which gives us the degree of $f\circ \zeta$.   
\end{proof}

Thus by Corollary~\ref{corollary;homog} and Proposition~\ref{proposition;twisthomog}, if $f\in \C[G^{w_0,e}]$ is a cluster variable or the twist of a cluster variable, then $f$ is a $H\times H$-eigenvector.

\section{Positivity and polyhedral parameterizations of canonical bases} \label{PositivitySection}

In this section we briefly recall the notions of a positive structure and the tropicalization functor, and describe a fragment of the theory of geometric crystals. More detailed discussions of these subjects can be found in Sections 3.1, 4.2, and 6.3 of \cite{BKII}. In the next section we will connect these notions to our more analytic perspective, of scaled families of functions on the real manifold $AN_-$.

Highlights of this Section are Definition~\ref{pvwp} of a positive variety with potential, Definition~\ref{definition; tropicalization of positive variety} which introduces the notion of tropicalization, Definition~\ref{definition; BK potential} of the Berenstein-Kazhdan (BK)  potential on a double Bruhat cell $G^{w_0, e}$, Corollary~\ref{corollary;domination} explaining that cluster variables on $G^{w_0, e}$ are dominated by the BK potential, Remark~\ref{stringremark} stating that tropicalization of $G^{w_0, e}$ in twisted cluster variables gives rise to the sting cone of $G^\vee$ and Theorem~\ref{theorem;ABHL2maintheorem} which establishes an isomorphism (over $\mathbb{R}$) of sting cones for $G$ and $G^\vee$.

We will make the abbreviation $\C^{\times n}:= (\C^\times)^n$. We fix a reductive algebraic group $G$ with compact form $K$, as in previous sections.

\subsection{Positivity theory and tropicalization}\label{section; positivity and trop}

Throughout, we frequently view characters $\gamma\in X^*(H)$ of a complex algebraic torus $H$ as regular functions on $H$. In particular, the standard coordinates $z_1,\dots ,z_m$ on  $\C^{\times m}$ are identified with the standard basis of $X^*(\C^{\times m})$.


\begin{definition}\label{definition; theta coordinate}
    A \emph{toric chart} on irreducible complex algebraic variety $A$ is an open embedding $\theta\colon H \to A$ of a complex algebraic torus $H$. Given a toric chart $\theta\colon \C^{\times m} \to A$, the function $z_i\circ \theta\n \in \C(A)$ is a \emph{$\theta$-coordinate on $A$}. We often write $z_i = z_i\circ \theta\n$.
\end{definition}


\begin{definition}
  Let $H$ and $S$ be complex algebraic tori. A rational function $f\in  \C(H) $ is \emph{positive} if it can be written in a subtraction free form
  \begin{equation}\label{positive rational function}
    f = \frac{f'}{f''}, \quad f' = \sum_{\gamma\in X^*(H)} A_\gamma \gamma, \quad f'' = \sum_{\delta \in X^*(H)} B_\delta \delta
  \end{equation}
  for $A_\gamma, B_\delta\geq 0$ with all but finitely many of the coefficients $A_\gamma, B_\delta$ equal to zero. A rational map $F\colon H\to S$ is {\em positive} if $F^\gamma$ is positive for all $\gamma \in X^*(S)$. 
\end{definition}

Characters are positive functions and homomorphisms are positive maps, but not conversely.

\begin{definition}\label{pvwp}
  A \emph{positive variety with potential} is a triple $(A, \Phi_A, \Theta_A)$, where $A$ is an irreducible complex algebraic variety, $\Phi_A$ is a rational function on $A$ called a \emph{potential}, and $\Theta_A$ is a non-empty set of toric charts on $A$, such that:
  \begin{itemize}
    \item[(1)] there exists a toric chart $\theta\in \Theta_A$, such that $\Phi_A\circ \theta$ is positive; and
    \item[(2)] for any pair $\theta,\theta'\in\Theta_A$, the compositions $\theta^{-1}\circ \theta'$ and $(\theta')^{-1}\circ \theta$ are positive.
  \end{itemize}
  A {\em positive} rational function on a positive variety with potential $(A,\Phi_A,\Theta_A)$ is a rational function $f\in \C(A)$ such that $f\circ \theta$ is positive for some (equivalently, any) $\theta \in \Theta_A$.
\end{definition}

Toric charts $\theta,\theta'$ on $A$ that satisfy (2) are {\em positive equivalent toric charts}. Positive equivalence is an equivalence relation on the set of all toric charts on $A$. Thus, condition (2) requires that $\Theta_A$ is a subset of an equivalence class. 

The following special cases deserve their own names. If $\Phi_A=0$, then $(A,\Theta_A) = (A,0,\Theta_A)$ is a {\em positive variety}. If $\Theta_A=\{\theta\}$ is a singleton, then $(A,\Phi_A, \theta) = (A,\Phi_A, \{\theta\})$ is a {\em framed positive variety with potential}. 
If $\Phi_A = 0$ and $\Theta_A=\{\theta\}$, then $(A,\theta) = (A,0,\{\theta\})$ is a \emph{framed positive variety}. Every complex algebraic torus has a natural framed positive variety structure.

\begin{example}\label{posClust}
Let $A$ be an irreducible complex algebraic variety and assume that $\C[A]$ is isomorphic to a cluster algebra $\mathcal{A}_{|\sigma|}$.  Let $({\bm z}_{\sigma'}, \sigma')\in |\sigma|$ be a labeled seed for $\mathcal{A}_{|\sigma|}$. By the Laurent Phenomenon~\cite{BFZ}, there is an inclusion of rings
\begin{align*}
\C[A] \hookrightarrow \C[{\bm z}_\sigma^\pm] = \C[z_1^\pm, \dots, z_n^\pm]
\end{align*}
that induces a toric chart $\theta_{\sigma'} \colon \C^{\times n} \to A$. 
Since the mutation equations~\eqref{equation;mutationvariable} are subtraction free, toric charts $\theta_\sigma$, $\theta_{\sigma'}$ associated to mutation equivalent seeds $({\bm z}_{\sigma}, \sigma)$ and $({\bm z}_{\sigma'}, \sigma')$ are positive equivalent.
\end{example}

\begin{definition}  
  Let $(A,\Phi_A,\Theta_A)$ be a positive variety with potential. A rational function $f\in\C(A)$ is {\em dominated by $\Phi_A$}, denoted $f\prec \Phi_A$, if there exists positive rational functions $f^+,f^-$ on $(A,\Phi_A,\Theta_A)$ and a polynomial $p(\Phi_A)$ in $\Phi_A$ with positive real coefficients, such that $f=f^+-f^-$ and both
    \[
        p(\Phi_A)-f^+,\quad \text{and}\quad p(\Phi_A)-f^-
    \]
are positive with respect to $\Theta_A$.
\end{definition}

\begin{definition}
  A {\em map of positive varieties with potential} $f\colon(A,\Phi_A,\Theta_A)\to (B,\Phi_B,\Theta_B)$ is a rational map $f\colon A\to B$ such that:
  \begin{itemize}
    \item[(1)] for some (equivalently, any) $\theta_A\in \Theta_A$ and $\theta_B\in \Theta_B$,  $\theta_B^{-1}\circ f\circ \theta_A$ is positive, and 
    \item[(2)] $f^*\Phi_B\prec \Phi_A$.
  \end{itemize}
\end{definition}
When $\Phi_A=0$ and $\Phi_B=0$, then the second condition holds automatically.

We now recall the tropicalization construction; for a full discussion see~\cite{BKII}. The {\em tropicalization} of a framed positive variety $(A,\theta_A\colon H\to A )$ is 
\[
    (A,\theta_A )^t = X_*(H).
\] 
The {\em tropicalization} of a positive rational function $f$ on a framed positive variety $(A,\theta_A)$ is 
\[
    f^t \colon (A,\theta_A)^t \to \Z, \quad f^t = \min_{\substack{\gamma\in X^*(H)\\ A_\gamma \neq 0}}\{\gamma\} - \min_{\substack{\delta\in X^*(H)\\ B_\delta \neq 0}}\{\delta\}
\]
where $f\circ \theta_A$ is assumed to have the form~\eqref{positive rational function} and $H$ is the domain of $\theta_A$. The {\em tropicalization} of a map of framed positive varieties $f\colon (A,\theta_A) \to (B,\theta_B)$ is the map $f^t \colon (A,\theta_A)^t \to (B,\theta_B)^t$ uniquely defined by the property that $\langle \gamma, f^t \rangle = (f^\gamma)^t$ for all $\gamma \in X^*(S)$, where $S$ is the domain of $\theta_B$. If $f$ is a homomorphism, then $f^t=X_*(f)$.

If $k$ is a nonnegative integer, then $f^t(kp) = kf^t(p)$ for all $p \in X_*(H)$. Define the real extension of $f^t$,
\begin{align*}
f^t \colon X_*(H)\otimes_{\Z} \R & \to X_*(S)\otimes_{\Z} \R, \quad f^t(p\otimes x) = \left\{\begin{array}{ll}
    f^t(p)\otimes x & \text{ if }x\geq 0,\\
    f^t(-p)\otimes (-x) & \text{else.}
\end{array}\right.
\end{align*}
Here we have overloaded the notation $f^t$, but the meaning will be clear from context.
This map $f^t$ is piecewise $\R$-linear and the linearity chambers are cones.


\begin{definition} \label{definition; tropicalization of positive variety}
Let $(A,\Phi, \theta_A\colon H \to A)$ be a framed positive variety with potential. The \emph{tropicalization of $A$} (and its real points) are the cones
\begin{align}
  \label{cone definition}
  (A,\Phi, \theta_A)^t & = \{ x\in X_*(H) \mid \Phi^t(x) \geqslant 0 \}; \\
  \label{real cone definition}
  (A,\Phi, \theta_A)^t_{\mathbb{R}} & = \{ x\in X_*(H)\otimes_{\Z} \mathbb{R} \mid \Phi^t(x)\geqslant 0\}.
\end{align}
For any $\delta\geqslant 0$, the \emph{$\delta$-interior} of the tropicalization is
\begin{equation}
    (A,\Phi, \theta_A)^t_\R(\delta) = \{x\in X_*(H)\otimes_{\Z} \mathbb{R} \mid \Phi^t(x)> \delta\}.
\end{equation}
\end{definition}

The space $(A,\Phi, \theta_A)^t_{\mathbb{R}}$ has a natural topology defined as follows. An isomorphism $H\cong \C^{\times m}$ induces an isomorphism of sets $(A, \theta_A)^t_{\mathbb{R}} \cong  \R^m$. The induced topology on  $(A, \theta_A)^t_{\mathbb{R}}$ is independent of the choice of isomorphism $H\cong \C^{\times m}$.
The cone $(A,\Phi, \theta_A)^t_{\mathbb{R}}\subset (A, \theta_A)^t_{\mathbb{R}}$ has the subspace topology.
With respect to this topology, the set $(A,\Phi, \theta_A)^t_\R(0)$ is the topological interior of $(A,\Phi, \theta_A)^t_\R$. 
If $\Phi \circ \theta_A$ is a regular function on $H$, then $(A,\Phi, \theta_A)^t$ is a polyhedral cone. 

If $f\colon (A,\Phi_A, \theta_A) \to (B,\Phi_B, \theta_B)$
is a map of framed positive varieties with potential, then 
the image of $(A,\Phi_A, \theta_A)^t$ under  $f^t$ is contained in $(B,\Phi_B, \theta_B)^t$. The \emph{tropicalization} of $f$ is the resulting piecewise $\Z$-linear map  
\[
  f^t\colon (A,\Phi_A, \theta_A)^t\to (B,\Phi_B, \theta_B)^t.
\]
The real extension of $f^t$ defines a piecewise $\R$-linear map  $f^t \colon (A,\Phi_A, \theta_A)^t_{\mathbb{R}}\to (B,\Phi_B, \theta_B)^t_{\mathbb{R}}$. 

Tropicalization is functorial, in that it respects composition of maps of framed positive varieties with potential \cite[Theorem~4.12]{BKII}.

\begin{example}
Extending Example~\ref{posClust}, let $A^\vee$ be a variety with $\C[A^\vee] \cong \mathcal{A}_{|\sigma|}^\vee$, and assume a skew symmetrizer ${\bm D}$ of the mutation matrix of $\sigma$ has been fixed. Let $\Theta$ and $\Theta^\vee$ denote the collection of toric charts on $A$ and $A^\vee$ arising from cluster seeds of $\mathcal{A}_{|\sigma|}$ and $\mathcal{A}_{|\sigma|}^\vee$. Then, as in~\eqref{comparisongeneral}, for each $({\bm z}_{\sigma'},\sigma')\in |\sigma|$, we have a map of positive varieties
\[
(A,\Theta) \xrightarrow{\varPsi_{\sigma'}} (A^\vee, \Theta^\vee).
\]
Due to~\cite[Proposition~4.7]{ABHL2}, the maps $\varPsi_{\sigma'}$ and $\varPsi_{\sigma}$  coincide, after tropicalization. In other words, for any $\theta\in \Theta$ and $\theta^\vee\in \Theta^\vee$, the diagram commutes:
\begin{equation} \label{comparisonCompat}
\begin{tikzcd}
(A,\theta)^t \ar[r,"\varPsi_\sigma^t"] \ar[d, equal] & (A^\vee,\theta^\vee)^t \ar[d, equal] \\
(A,\theta)^t \ar[r,"\varPsi_{\sigma'}^t"] &(A^\vee,\theta^\vee)^t.
\end{tikzcd}
\end{equation}
\end{example}

\subsection{Positive structures on the double Bruhat cell \texorpdfstring{$G^{w_0,e}$}{Gwe} }\label{section;chartsDBC}

In this section we introduce several important toric charts on $G^{w_0,e}$ and define the Berenstein-Kazhdan (BK) potential. This gives us a positive variety with potential $(G^{w_0,e},\Phi_{BK}, \Theta(G^{w_0,e}))$.

Assume temporarily that $G$ is of the form~\eqref{equation;niceformG}. Let $\mathbf{i}$ be a double reduced word for $(w_0,e)$ (i.e. a reduced word for $w_0$ written in the alphabet $\{-1,\dots,-r\}$). Let $({\bm z}_\sigma,\sigma) \in |\sigma(\ii)|$. Then
\[
    {\bm z}_\sigma = \{z_{-r},\dots,z_{-1},z_1,\dots, z_m\},
\]
where the frozen cluster variables with negative indices are the principal minors $z_{-i} = \Delta_{\omega_i}.$
Extend the cluster variables $z_i\in {\bm z}_\sigma$ to functions on $Z\times G_{sc}$ by setting $z_i(z,g) = z_i(g)$. Choose characters $\gamma_{1},\dots, \gamma_{\tilde{r}-r} \in X^*_+(H)\cap \z(\g)^*$ so that the collection
\begin{equation}\label{equation;collection of characters}
    \omega_1,\dots, \omega_r, \gamma_{1},\dots, \gamma_{\tilde{r}-r}
\end{equation}
forms a $\Z$-basis for $X^*(H)$. We use the notation $z_{-(r+j)} : = \Delta_{\gamma_j} \in \C[G^{w_0,e}]$. By Example~\ref{posClust}, the birational evaluation map
\begin{align*}
G^{w_0,e} \to \C^{\tilde{r}+m} \ :\  g \mapsto (z_{-\tilde{r}}(g),\dots, z_{-r-1}(g),z_{-r}(g),\dots, z_{-1}(g), z_1(g),\dots, z_{m}(g))
\end{align*}
induces a toric chart which we denote
\begin{equation}\label{equation;clusterchart}
\theta_{\sigma} \colon \C^{\times (\tilde{r}+m)} \hookrightarrow G^{w_0,e}.
\end{equation}
Recall the twist map $\zeta\colon G^{w_0,e} \to G^{w_0,e}$ introduced in \eqref{twist map equation}. For any $\theta_\sigma$, we get a toric chart 
\begin{equation}\label{twisted cluster charts}
    \theta_\sigma^\zeta = \zeta\circ \theta_\sigma
\end{equation}
which is positively equivalent to $\theta_\sigma$.

\begin{definition} \label{definition;unreduced cluster chart}
  For $G$ of the form~\eqref{equation;niceformG}, any toric chart $\theta\colon \C^{\times(\tilde{r}+m)}\to G^{w_0,e}$ of the form~\eqref{equation;clusterchart} is a \emph{cluster chart on $G^{w_0,e}$}. Any toric chart $\theta$ of the form~\eqref{twisted cluster charts} is a \emph{twisted cluster chart on $G^{w_0,e}$.} If $\theta$ is either a cluster chart or a twisted cluster chart, we will say that $\theta$ is a (twisted) cluster chart.
\end{definition}

\begin{remark} \label{alwaysappearremark}
Let $\ii$ be any double reduced word for $(w_0,e)$. It follows from the definition of  $\mathcal{A}_{|\sigma(\mathbf{i})|}$ that the functions
    $\Delta_{w_0\omega_i,\omega_i}$ and $\Delta_{\omega_i}$, $i \in [i,\tilde r]$, are $\theta_\sigma$-coordinates for any $\sigma \in |\sigma(\ii)|$. By~\eqref{equation;twistfrozen},
   the functions $\Delta_{w_0\omega_i,\omega_i}\n$ and $\Delta_{\omega_i}\n$, $i \in [i,\tilde r]$, are $\theta_\sigma^\zeta$-coordinates for any $\sigma \in |\sigma(\ii)|$.
\end{remark}

Now, let $G$ be any connected reductive algebraic group. Identify $H\times L^{w_0,e} \cong G^{w_0,e}$ using multiplication in $G$. For a double reduced word $\mathbf{i}$ of $(w_0,e)$, let $({\bm z}_{\overline{\sigma}}, \overline{\sigma})$ be a labeled seed of the cluster algebra $\mathcal{A}_{|\overline{\sigma}(\mathbf{i})|}\subseteq \C[L^{u,v}]$ as in Theorem~\ref{Thm2.10}. Define the toric chart
\begin{equation} \label{equation;toricchartDBC}
\theta_{\overline{\sigma}}\colon \C^{\times\tilde{r}} \times \C^{\times m} \to H \times L^{w_0,e}
\end{equation}
which is the product of an isomorphism 
\begin{equation} \label{equation;splitH}
\C^{\times \tilde{r}}\cong H
\end{equation} and the toric chart $\theta_{\overline{\sigma}}\colon \C^{\times m} \to L^{w_0,e}$. In what follows we will assume a choice of isomorphism~\eqref{equation;splitH} has been fixed.

\begin{definition}\label{definition;reduced cluster chart} Any toric chart $\theta\colon \C^{\times \tilde{r}} \times \C^{\times m}\to G^{w_0,e}$ of the form~\eqref{equation;toricchartDBC} is a \emph{reduced cluster chart}. If $\theta=\zeta\circ \theta_{\overline{\sigma}}$ for some reduced cluster chart $\theta_{\overline{\sigma}}$, then $\theta$ is a \emph{twisted reduced cluster chart}. In this case, the Langlands dual of $\theta$ is defined to be $\zeta^\vee\circ \theta_{\overline{\sigma}^\vee}$, the twist of the Langlands dual of $\theta_{\overline{\sigma}}$.
\end{definition}

\begin{remark} \label{reducedchangeofcoords}
The reduced and unreduced charts in Definitions~\ref{definition;unreduced cluster chart} and~\ref{definition;reduced cluster chart} are used for different purposes. Assume $G=Z\times G_{sc}$ of the form~\eqref{equation;niceformG}, and let $H_{sc}$ be the Cartan subgroup of $G_{sc}$. Then, the unreduced charts are more convenient to describe the Poisson brackets that arise in Theorem~\ref{theorem;brackettheorem}. On the other hand, the reduced charts are defined on $G^{w_0,e}$ for all reductive $G$, and are needed in Section~\ref{section;geometriccrystals}.  

Let $({\bm z}_\sigma,\sigma)$ be a seed for $\C[G_{sc}^{w_0,e}]$, and let $({\bm z}_{\overline{\sigma}},\overline{\sigma})$ be the corresponding reduced seed for $\C[L^{w_0,e}]$.
If $\theta = \theta_\sigma$ and $\theta^{red}=\theta_{\overline{\sigma}}$  is the corresponding reduced cluster chart on $G^{w_0,e}\cong Z\times H_{sc} \times L^{w_0,e}$, then the $\theta$-coordinates are related to the $\theta^{red}$-coordinates by a Laurent monomial change of coordinates as follows.  Let $i^*\colon \C[G_{sc}^{w_0,e}]\to \C[L^{w_0,e}]$ be the projection dual to the inclusion $L^{w_0,e}\hookrightarrow G_{sc}^{w_0,e}$. Then, define 
\begin{equation}
\label{isomorphismredcoords}
\begin{split}
\C[Z]\otimes \C[G^{w_0,e}_{sc}] & \to \C[Z]\otimes \C[H_{sc}] \otimes \C[L^{w_0,e}] \\
f\otimes z_k & \mapsto f \otimes |z_k|_1 \otimes i^*z_k,
\end{split}
\end{equation}
where we write $|z_k| = (|z_k|_1,|z_k|_2) \in P\times P$, using the grading of Corollary~\ref{corollary;homog}. It is straightforward to show that~\eqref{isomorphismredcoords} is an algebra isomorphism, and does not depend on the choice of seed $\sigma$. Similarly, if $\theta=\zeta\circ \theta_\sigma$ is a twisted cluster chart and $\theta^{red}= \zeta\circ \theta_{\overline{\sigma}}$, then~\eqref{isomorphismredcoords} takes $\theta$-coordinates to Laurent monomials in $\theta^{red}$ coordinates. In particular, in either case the $\theta^{red}$-coordinates $z_i$ are $P\times P$-homogenous.
\end{remark}


\begin{definition}[Notation] \label{definition;notation for pos} For a connected reductive algebraic group $G$, denote by $\Theta(G^{w_0,e})$ the set of all  reduced cluster charts arising from all double reduced word $\mathbf{i}$ of $(w_0,e)$. Moreover, if $G$ is of the form~\eqref{equation;niceformG}, we extend $\Theta(G^{w_0,e})$ to include all (twisted) cluster charts and (twisted) reduced cluster charts arising from all double reduced words $\mathbf{i}$ of $(w_0,e)$, all the choices of characters $\gamma_i$, and all isomorphisms $H\cong \C^{\times\tilde{r}}$. In particular, if $({\bm z}_{\sigma},\sigma)$ is mutation equivalent to $({\bm z}_{\sigma(\mathbf{i})},\sigma(\mathbf{i}))$, then $\theta_{\sigma}\in \Theta(G^{w_0,e})$.
\end{definition}

\begin{definition} \label{definition; BK potential}
  Let $G$ be a connected reductive algebraic group, and assume the character lattice $X^*(H)$ contains the fundamental weights $\omega_i$. The function
  \begin{equation}\label{BK potential equation}
    \Phi_{BK} = \sum_{i\in {\bm I}}\frac{\Delta_{w_0\omega_i,s_i\omega_i}+\Delta_{w_0s_i\omega_i,\omega_i}}{\Delta_{w_0\omega_i,\omega_i}}
  \end{equation}
  is a regular function on $G^{w_0,e}$. Following its description in~\cite[Corollary~1.25]{BKII}, $\Phi_{BK}$ is called the \emph{Berenstein-Kazhdan (BK) potential}. 
\end{definition}

More generally, let $G$ be any connected reductive algebraic group, and let $\widehat{G}\to G$ be a covering group on which $\Phi_{BK}$ is defined. Then the function $\Phi_{BK}$ is invariant under translations by the center $Z(\widehat{G})$, and so descends to a well defined function on $G$. This function will also be denoted $\Phi_{BK}$. 

The following is a restatement of \cite[Proposition~4.9]{ABHL1} and \cite[Lemma~3.36]{BKII}.

\begin{proposition}
The triple $(G^{w_0,e},\Phi_{BK},\Theta(G^{w_0,e}))$ is a positive variety with potential.
\end{proposition}

By tropicalization, for any $\theta\in \Theta(G^{w_0,e})$, we get a polyhedral cone $(G^{w_0,e},\Phi_{BK},\theta)^t$ called a \emph{Berenstein-Kazhdan (BK) cone}.

\subsection{Domination of functions on \texorpdfstring{$G^{w_0,e}$}{Gw0,e}}\label{section; domination}

This section describes a family of functions on $G^{w_0,e}$ that are dominated by  $\Phi_{BK}$. This    technical property is exploited to describe the limiting behavior of a family of Poisson brackets in Theorem~\ref{theorem;bracketconvergence}.

Let $G$ be a reductive algebraic group and $\mathfrak{g}$ its Lie  algebra as before. Recall that $N_-$ is the unipotent radical of the Borel $B_-$. A \emph{$N_-\times N_-$-variety} is a variety equipped with a left action of $N_-$ and a right action of $N_-$, such that the two actions commute.

Since $N_-$ is unipotent, the exponential map $\exp\colon \mathfrak{n}_-\to N_-$ is algebraic. Thus for any $(F,F') \in\mathfrak{n}_-\times \mathfrak{n}_-$, the map 
\begin{equation}\label{equation;algebraicaction} 
    \C\times A \to A, \quad  
    (t, a) \mapsto \exp(-tF)\cdot a\cdot \exp(tF') 
\end{equation} 
is algebraic. The action of the fundamental vector field of $(F,F')\in \mathfrak{n}_-\times \mathfrak{n}_-$ on $f\in C^\infty(A,\C)$ is given by
\[
(F\cdot f\cdot F')(a) = \left.\frac{d}{dt}\right|_{t=0} f\left(
 \exp(-tF)\cdot a\cdot \exp(tF') 
\right).
\]
Since the map~\eqref{equation;algebraicaction} is algebraic, the action of $(F,F')$ restricts to a derivation of $\C[A]$.

\begin{lemma} \label{lemma;domination}
  Let $A$ be a $N_-\times N_-$-variety and $(A,\Phi,\Theta)$ be a positive variety with potential. Let $\{a_i\}_{i=1}^n$ be a set of positive rational functions on $(A,\Phi,\Theta)$ and let $(F,F')\in \mathfrak{n}_-\times\mathfrak{n}_-$. If 
  \[
    \dfrac{F \cdot a_i\cdot F'}{a_i}\prec \Phi,\quad \forall i\in [1,n],
  \]
  then $(F \cdot f\cdot F')/f\prec \Phi$ for any  subtraction free Laurent polynomial $f:=f(a_1,\ldots,a_n)$ in the functions $a_i$.
\end{lemma}

\begin{proof}
  First of all, since the Lie algebra $\mathfrak{n}_-\times \mathfrak{n}_-$ acts by derivations, for a Laurent monomial $a_1^{m_1}\cdots a_n^{m_n}$ and any positive real number $c>0$, we have:
  \[
    \frac{F\cdot (c a_1^{m_1}\cdots a_n^{m_n})\cdot F'}{ca_1^{m_1}\cdots a_n^{m_n}}=\sum_i m_i\frac{F \cdot a_i\cdot F'}{a_i}\prec \Phi.
  \]
  Next, denote by $f=f_1+\cdots +f_m$ a linear combination of Laurent monomials in the functions $a_i$ with positive coefficients. By the first step, we know $(F\cdot f_i\cdot F')/f_i\prec \Phi$. In other words, one can write
  \[
    \frac{F\cdot f_i\cdot F'}{f_i}=f_i^+-f_i^-    
  \]
  where $p_i(\Phi)-f_i^+$ and $p_i(\Phi)-f_i^-$ are positive with respect to $\Theta$, for some polynomial $p_i$ in $\Phi$ with positive coefficients. Then we have: 
  \[
    \frac{F\cdot f\cdot F'}{f}=\sum_i \frac{f_i}{f}\cdot \frac{F\cdot f_i\cdot F'}{f_i}=\sum_i \frac{f_i}{f}\cdot f_i^+ -\sum_i \frac{f_i}{f}\cdot f_i^- .
  \]
Then, putting $p(\Phi)=\sum_i p_i(\Phi)$, one has
 \[
    p(\Phi)-\sum \frac{f_i}{f}\cdot f_i^+=\sum \frac{f_i}{f}\left(p(\Phi)- f_i^+\right),\quad p\Phi-\sum \frac{f_i}{f}\cdot f_i^-=\sum \frac{f_i}{f}\left(p(\Phi)- f_i^-\right)
 \]
 which is positive with respect to $\Theta$. Thus $(F \cdot f\cdot F')/f\prec \Phi$.
\end{proof}


Our prototypical example of a $N_-\times N_-$-variety is the double Bruhat cell $G^{w_0,e}$ with action $(u,g,u')\mapsto ugu'$, where $g\in G^{w_0,e}$ and $u,u'\in N_-$. 

\begin{proposition} \label{proposition;dominationABHL1}
  Let $G$ be an algebraic group of the form~\eqref{equation;niceformG}. For the positive variety with potential $(G^{w_0,e},\Phi_{BK},\Theta(G^{w_0,e}))$, we have
  \[
    \dfrac{F \cdot \Delta \cdot F'}{\Delta}\prec \Phi_{BK},\quad  F, F' \in \mathfrak{n}_-
  \]
  for all generalized minors $\Delta$. In particular, if $\mathbf{i}$ is a double reduced word for $(w_0,e)$, then $(F\cdot z \cdot F')/z \prec \Phi_{BK}$ for all $\theta_{\sigma(\mathbf{i})}$-coordinates $z$.
\end{proposition}

\begin{proof}
  This was shown in a slightly weaker form in \cite[Theorem~4.13]{ABHL1}; the proof extends to our context with minor modifications.
\end{proof}

\begin{corollary} \label{corollary;domination}
    Let $G$ be a connected reductive algebraic group, of the form~\eqref{equation;niceformG}. Let $\theta\in \Theta(G^{w_0,e})$, and let $z$ be a Laurent monomial in $\theta$-coordinates on $G^{w_0,e}$.
	Then 
	\[
	\frac{F \cdot z \cdot F'}{z}\prec \Phi_{BK},
	\]
	for all $(F,F') \in \mathfrak{n}_-\times \mathfrak{n}_-$.
\end{corollary}


\begin{proof}
First, assume that $\theta=\theta_{\sigma(\mathbf{i})}^\zeta$, for some cluster chart $\theta_{\sigma(\mathbf{i})}$ given by a double reduced word $\mathbf{i}$ for $(w_0,e)$. Let $\{z_{-\tilde{r}},\dots,z_{-1},z_1,\dots,z_m\}$ be the set of $\theta_{\sigma(\mathbf{i})}$-coordinates; these are simply generalized minors on $G^{w_0,e}$. By~\cite[Theorem~5.8]{BZ01}, any $\theta_{\sigma(\mathbf{i})}$-coordinate $z_i$ can be written as a subtraction free polynomial in ``factorization parameters'' $t_i$. In turn, by~\cite[Theorem~1.9]{FZ} these $t_i$ can be written as Laurent monomials in the $\theta_{\sigma(\mathbf{i})}^\zeta$-coordinates. So we have
\[
z_i = f(z_{-\tilde{r}}\circ \zeta,\dots,z_{m}\circ \zeta)
\]
for some subtraction free Laurent polynomial $f$. Precomposing both sides by the involution $\zeta$ gives $z_i\circ \zeta = f(z_{-\tilde{r}},\dots, z_m)$. Therefore the $\theta$-coordinate $z_i\circ \zeta$ is a subtraction free Laurent polynomial in the $\theta_{\sigma(\mathbf{i})}$ coordinates, and the claim follows from Lemma~\ref{lemma;domination} and Proposition~\ref{proposition;dominationABHL1}.

Now, assume that $\theta$ is any cluster chart. Then by Theorem~\ref{Thm2.10} and the Laurent phenomenon for cluster algebras, the $\theta$-coordinates can be written as subtraction free Laurent polynomials in $\theta_{\sigma(\mathbf{i})}$-coordinates, for some double reduced word $\mathbf{i}$ for $(w_0,e)$. Similarly, if $\theta$ is a twisted cluster chart, then the $\theta$-coordinates are subtraction free Laurent polynomials in $\theta_{\sigma(\mathbf{i})}^\zeta$-coordinates, for some $\mathbf{i}$.
In the first case, the claim then follows from Lemma~\ref{lemma;domination} and Proposition~\ref{proposition;dominationABHL1}. In the second case, the claim follows from Lemma~\ref{lemma;domination}, Proposition~\ref{proposition;dominationABHL1}, and the first paragraph of this proof.
\end{proof}

\subsection{Geometric crystals}\label{section;geometriccrystals}

The authors of \cite{BKII} defined geometric crystals as a geometric analogue of Kashiwara's crystals. A (positive decorated) geometric crystal  is a positive variety with potential $(X,\Phi,\Theta)$ with several additional rational structure maps. Though our main examples originate from this theory, we will only use a fragment of the additional structure.

For a connected reductive group $G$, the positive variety with potential $(G^{w_0,e}, \Phi_{BK},\Theta(G^{w_0,e}))$ can be enriched with the structure of a geometric crystal, as in~\cite[Theorem~6.15]{BKII}. Of the additional geometric crystal structure maps, there are two we are interested in:
\begin{align}
\label{hw map}
  \hw &\colon G^{w_0,e}\cong H \times L^{  w_0,e}\to H, \quad  (h,z)\mapsto h,\\
  \label{wt map}
  \wt &\colon G^{ w_0,e}\to H, \quad z \mapsto [z]_0.
\end{align}
They are called the \emph{highest weight} and \emph{weight} maps, respectively.  Using~\eqref{equation;RDBCdefining} they satisfy the following equalities.
\begin{equation}\label{equation;hwwtcoordinates}
\hw(g)^{w_0\gamma} = \Delta_{w_0\gamma,\gamma}(g),\qquad \wt(g)^{\gamma} = \Delta_{\gamma,\gamma},
\end{equation}
for any $\gamma\in X^*(H)$, whenever the right hand side is defined. Combining this with Remark~\ref{alwaysappearremark} yields the following lemma. 

\begin{lemma}\label{lemma; wt hw homomorphism property}
	Let $\theta\in\Theta(G^{w_0,e})$. Then, $\wt \circ \theta$ and $\hw \circ \theta$ 
	are group homomorphisms. Moreover, the induced maps on character lattices, $X^*(\wt\circ \theta)$ and $X^*(\hw\circ \theta)$, are injective with torsion-free cokernels. In particular, $\hw,\wt \colon (G^{w_0,e}, \Phi_{BK},\Theta(G^{w_0,e})) \to H$ are positive maps and their tropicalizations with respect to any $\theta \in \Theta(G^{w_0,e})$ are $\Z$-linear.
\end{lemma}

Let $G^\vee$ be the Langlands dual group of $G$. Then $(G^{\vee;w_0,e},\Phi_{BK}^\vee, \Theta(G^{\vee;w_0,e}))$ together with the structure maps
\[
\hw^\vee \colon G^{\vee;w_0,e} \to H^\vee,\qquad \wt^\vee \colon G^{\vee;w_0,e} \to H^\vee
\]
captures aspects of the representation theory of $G$, according to the following theorem. It is a shadow of \cite[Theorem~6.15]{BKII}, obtained by forgetting parts of the Kashiwara crystal structure.

\begin{theorem} \label{theorem;BKmaintheorem} Let $G$ be a connected reductive algebraic group, let $\theta^\vee\in \Theta(G^{\vee;w_0,e})$ and consider the framed positive variety with potential $(G^{\vee;w_0,e},\Phi_{BK}^\vee, \theta^\vee)$. Let $\lambda,\nu\in X^*(H)=X_*(H^\vee)$ be characters of $H$. If $\lambda\in X^*_+(H)$ is dominant, then the cardinality of 
\begin{equation} \label{equation;countingdimension}
((\hw^\vee)^t)\n(\lambda) \cap ((\wt^\vee)^t)\n(\nu) \subset (G^{\vee;w_0,e},\Phi_{BK}^\vee, \theta^\vee)^t
\end{equation}
 is equal to the dimension of the $\nu$-weight space in $V_\lambda$, the irreducible $G$ module with high weight $\lambda$. If $\lambda\not \in X^*_+(H)$, then $((\hw^\vee)^t)\n(\lambda) = \emptyset$. 
 What is more, the sets $((\hw^\vee)^t)\n(\lambda)$ and \eqref{equation;countingdimension} are the lattice points of bounded convex polytopes.
\end{theorem}

Part of the last statement of Theorem~\ref{theorem;BKmaintheorem} can be seen from~Lemma~\ref{lemma; wt hw homomorphism property}. In any $\theta^\vee\in \Theta(G^{w_0,e})$, maps $\hw^t$ and $\wt^t$  extend to linear maps out of $(G^{\vee;w_0,e},\theta^\vee)^t$. The fibers of these maps in $(G^{\vee;w_0,e},\Phi_{BK}^\vee, \theta^\vee)^t$ are the intersection of an affine hyperplane with the polyhedral cone $(G^{\vee;w_0,e},\Phi_{BK}^\vee, \theta^\vee)^t$. A description of these polytopes, as well as their connection to canonical bases, will be given in the next section.

\subsection{Polyhedral parametrizations of canonical bases}\label{section;canonical bases}

Let $G$ be a reductive algebraic group. According to \cite{BZ01}, for each reduced word $\ii$ of $w_0$ there is a polyhedral cone called the \emph{string cone}. Its integral points parametrize the canonical basis of the quantized universal enveloping algebra $ U_q(\nnn)$. The name ``string cone'' comes from the interpretation of points of the cone as strings of operators on $U_q(\nnn)$. 
In this section, we show how to recover each string cone as the tropicalization of $L^{w_0,e}$, with respect to a specific toric chart.
\begin{definition}
  Let $G$ be a connected reductive group, define the function
  \begin{equation}\label{BZp}
      \Phi_{L}=\sum_{i\in [1,r]}\Delta_{w_0\omega_i,s_i\omega_i}.
  \end{equation}
  It is a well defined regular function on $L^{w_0,e}$ under the identification \eqref{equation;coverL}.
\end{definition}

Next, we will introduce toric charts for $L^{w_0,e}$. Recall~\eqref{equation;RDBCdefining} that an element $x\in G^{w_0,e}$ belongs to $L^{w_0,e}$ if and only if $\Delta_{w_0\omega_i,\omega_i}(x)=1$ for all $i\in [1,r]$. Consider any chart $\theta\colon \mathbb{C}^{\times (m+\tilde{r})}\to G^{w_0,e}$ in $\Theta(G^{w_0,e})$ as Definition~\ref{definition;notation for pos}. 
Because the functions $\Delta_{w_0\omega_i,\omega_i}$ are all Laurent monomials in $\theta$-coordinates, the preimage $\theta\n(L^{w_0,e})$ is a subtorus of $\C^{m+\tilde{r}}$ of dimension $m$.
%
%
%
%
Then
\[
\overline{\theta}:=\theta|_{\theta\n(L^{w_0,e})}\colon \theta\n(L^{w_0,e}) \to L^{w_0,e}
\]
is a toric chart for $L^{w_0,e}$. Denote by $\Theta(L^{w_0,e})$ all toric charts for $L^{w_0,e}$ arising this way. 

\begin{remark}\label{reducechart}
    If the chart $\theta\in \Theta(G^{w_0,e})$ is a reduced cluster chart, then by~\eqref{equation;toricchartDBC} the chart $\overline{\theta}$ for $L^{w_0,e}$ is nothing but the cluster chart as in Example \ref{posClust}. What is more, by \cite[Theorem 4.7, Eq(6.1)]{BZ01}, the following map
    \[
        \eta\colon (L^{e,w_0})^T\hookrightarrow  G^{w_0,e}\xrightarrow{\zeta} G^{w_0,e}\cong H\times L^{w_0,e}\xrightarrow{\pr} L^{w_0,e}
    \]
    is a biregular isomorphism, where $\pr$ is the natural projection, $\zeta$ is the twist, and $(\cdot)^T\colon G\to G$ is the transpose anti-automorphism of~\cite[Equation 2.1]{FZ}. The map $\eta$ is  called ``twist'' in \cite{BZ01}. Therefore, if  $\theta\in \Theta(G^{w_0,e})$ is a twisted reduced chart, $\overline{\theta}$ can be viewed as the composition of a toric chart for $(L^{e,w_0})^T$ and the map $\eta$. 
\end{remark}

\begin{proposition}The triple $(L^{w_0,e}, \Phi_{L}, \Theta(L^{w_0,e}))$ is a positive variety with potential. 
  The natural projection $\pr\colon G^{w_0,e}\cong H\times L^{w_0,e}\to L^{w_0,e}$ gives rise to a morphism of positive varieties with potential from $(G^{w_0,e}, \Phi_{BK},  \Theta(G^{w_0,e}))$ to $(L^{w_0,e}, \Phi_{L},  \Theta(L^{w_0,e}))$. For any $\theta\in \Theta(G^{w_0,e})$, it induces a surjective map
    \[
        \pr^t\colon (G^{w_0,e}, \Phi_{BK}, \theta)^t \twoheadrightarrow  (L^{w_0,e}, \Phi_{L}, \overline{\theta})^t.
    \]
\end{proposition}

\begin{proof} 
 It is enough to consider the case when $\theta$ is the twisted reduced cluster chart $\zeta\circ \theta_{\overline{\sigma(\mathbf{i})}}$, as in \eqref{equation;toricchartDBC}.  Write $z_1,\dots,z_m$ for the $\theta$-coordinates which factor through the projection $G^{w_0,e} \to L^{w_0,e}$.     By \cite[Theorem~4.8, Theorem~5.8]{BZ01}, any generalized minor $\Delta_{\gamma,\delta}$ can be written (as a function on $L^{w_0,e}$), as
    \begin{equation}\label{itrail}
        \Delta_{\gamma,\delta}(\theta(z_1,\dots,z_m))=\sum_{\pi\in S_{\gamma,\delta}}N_{\pi}z_1^{c_1(\pi)}\cdots z_{m}^{c_m(\pi)},  
    \end{equation}
    where $S_{\gamma,\delta}$ is a finite index set, $N_{\pi}\in \mathbb{R}_{>0}$, and $c_{k}(\pi)\in \mathbb{Z}$. 
 
The first statement follows immediately from the positive expression~\eqref{itrail} for the generalized minors $\Delta_{w_0\omega_i,s_i\omega_i}$.
To show that $\pr$ is a morphism of positive varieties with potential, it is enough to show that
    \[
        \Phi_{BK}-\Phi_{L}\circ \pr
    \]
    is positive with respect to $\theta$. By \cite[Proposition~5.16]{ABHL2}, for any $(h,x)\in H\times L^{w_0,e}$ one has
    \[
        \Phi_{BK}(hx)=\sum_{i\in [1,r]} \left( \Delta_{w_0\omega_i,s_i\omega_i}(x)+h^{-w_0\alpha_i}\Delta_{w_0s_i\omega_i,\omega_i}(x)\right).
    \]
    Then
    \begin{equation} \label{BK and BZ}
    \Phi_{BK}(hx) -\Phi_{L}(x) = \sum_{i\in [1,r]} h^{-w_0\alpha_i}\Delta_{w_0s_i\omega_i,\omega_i}(x),
    \end{equation}
    and so $\Phi_{BK}-\Phi_{L}\circ \pr$ is positive with respect to $\theta$. 
    
    It remains to show that $\pr^t\colon (G^{w_0,e}, \Phi_{BK}, \theta)^t \to  (L^{w_0,e}, \Phi_{L}, \overline{\theta})^t$ is surjective.
   If $(t_1,\dots,t_m) \in (L^{w_0,e}, \Phi_{L}, \overline{\theta})^t$, we must find $\lambda^\vee\in X_*(H)$ such that $(\lambda^\vee,t_1,\dots,t_m) \in (H\times L^{w_0,e}, \Phi_{BK}, \overline{\theta})^t$. By~\eqref{BK and BZ}, one has $(\lambda^\vee,t_1,\dots,t_m)\in (H\times L^{w_0,e}, \Phi_{BK}, \overline{\theta})^t$ if and only if
    \begin{equation}\label{conesurjection}
    \sum_{k=1}^m c_k(\pi) t_k-\langle w_0\alpha_i,\lambda^\vee\rangle + \Phi_{L}^t(t_1,\dots,t_m) \geqslant 0,
    \end{equation}
    for all $\pi\in S_{w_0s_i\omega_i,\omega_i}$ and $i \in [1,r]$. Let $\lambda^\vee=\sum  n_i\omega_i^\vee$ for $n_i\geqslant 0$. Then we have $-\langle w_0\alpha_i,\lambda^\vee\rangle=\sum n_j\langle -w_0\alpha_i,\omega_j^\vee\rangle\geqslant 0$ since $-w_0\alpha_i$ is a simple root. By picking the $n_i$ sufficiently large one ensures the inequalities~\eqref{conesurjection} all hold.
%
\end{proof}


\begin{remark} \label{stringremark}
  If one chooses $\theta$ to be the twisted reduced chart $\zeta\circ \theta_{\overline{\sigma}(\mathbf{i})}$ associated to a double reduced word $\mathbf{i}$ of $(w_0,e)$, then 
  the cone $(L^{w_0,e}, \Phi_{L}, \overline{\theta})^t$ coincides with the string cone associated with $\mathbf{i}$ defined in \cite{BZ01,FFL,FLP}, as follows. 
  By Remark \ref{reducechart} and  \cite[Theorem 4.8]{BZ01}, the chart $\overline{\theta}$ is positively equivalent to the factorization chart $x_{\mathbf{i}}$ defined in \cite[Eq (4.10)]{BZ01}. 
  Moreover, the $\overline{\theta}$-coordinates are Laurent monomials in the coordinates of $x_{\mathbf{i}}$, and vice versa. 
  Thus the cone $(L^{w_0,e},\Phi_{L}, \overline{\theta})^t$ is unimodularly isomorphic to the one in \cite[Theorem 3.10]{BZ01}.
  What is more, for $\lambda^\vee\in X_*(H)$, the following maps are injective:
\begin{equation}\label{coneinj}
    (\hw^{t})\n(\lambda^\vee)\hookrightarrow (G^{w_0,e}, \Phi_{BK}, \theta)^t \twoheadrightarrow  (L^{w_0,e}, \Phi_{L}, \overline{\theta})^t.
\end{equation}
So one may view the fibers $(\hw^{t})\n(\lambda^\vee)$ as subsets of the string cone; these are sometimes known as \emph{string polytopes}.
\end{remark}

\subsection{Comparison map on double Bruhat cells}
\label{section;comparisonDBC}

In this section, we will discuss comparison maps between $G^{w_0,e}$ and $G^{\vee;w_0,e}$, in the spirit of Langlands dual cluster algebras. 

Fix a double reduced word $\mathbf{i}$ for $(w_0,e)$, and recall that $\C[L^{\vee;w_0,e}]\cong \mathcal{A}^\vee_{|\overline{\sigma}(\mathbf{i})|}$ is the Langlands dual cluster algebra of $\C[L^{w_0,e}]\cong \mathcal{A}_{|\overline{\sigma}(\mathbf{i})|}$. Recall from~\eqref{equation;skewsymmetrizer} that the matrix $\overline{\bm{D}}$ 
is a skew-symmetrizer of $\overline{M}(\mathbf{i})$. Then for each labeled seed $(\bm{z}_{\overline{\sigma}},\overline{\sigma}) \in |\overline{\sigma}(\ii)|$, one has the map
\[
\varPsi_{\overline{\sigma}}^*\colon \C[L^{\vee;w_0,e}] \to \C[L^{w_0,e}] \ :\ z_j^\vee\mapsto z_j^{d_{|i_j|}}
\]
as in~\eqref{comparisongeneral}. 
%
%
 We extend $\varPsi_{\overline{\sigma}}$ to a map $\varPsi_{\sigma}\colon G^{w_0,e}\to G^{\vee;w_0,e}$, by setting
\[
	\varPsi_\sigma^*:=\varPsi_H^* \otimes\varPsi_{\overline{\sigma}}^*  \colon \mathbb{C}[G^{\vee;w_0,e}]\cong \mathbb{C}[H^{\vee}]\otimes \mathbb{C}[L^{\vee;w_0,e}] \to \mathbb{C}[G^{w_0,e}]\cong \mathbb{C}[H]\otimes \mathbb{C}[L^{w_0,e}].
\]
If $\theta=\theta_{\overline{\sigma}}$ is a reduced cluster chart for $G^{w_0,e}$, and $\theta^\vee= \theta_{\overline{\sigma}^\vee}$ its dual chart, one has a map of positive varieties
\begin{equation} \label{comparisonreduced}
\varPsi_\theta:= \varPsi_{\sigma} \colon (G^{w_0,e}, \theta) \to (G^{\vee;w_0,e},\theta^\vee).
\end{equation}

Note that $\varPsi^t_{\theta}\colon (G^{w_0,e},\theta)^t \to (G^{w_0,e},\theta^\vee)^t$ is linear and lifts the comparison map $\psi_\h$ of~\eqref{equation;comparisonmap}, in the sense that the following diagrams commute.
\begin{equation} \label{comparisonlifts}
	\begin{tikzcd}[row sep=2.5em]
	(G^{w_0,e},\theta)^t \ar[r,"\varPsi_{\theta}^t"] \ar[d, "\hw^t"] & (G^{\vee;w_0,e},\theta^\vee)^t \ar[d, "(\hw^\vee)^t"] \\
	X_*(H) \ar[r, "\psi_\h"] & X_*(H^\vee)
	\end{tikzcd}
	\quad
	\begin{tikzcd}[row sep=2.5em]
	(G^{w_0,e},\theta)^t \ar[r,"\varPsi_{\theta}^t"] \ar[d, "\wt^t"] & (G^{\vee;w_0,e},\theta^\vee)^t \ar[d, "(\wt^\vee)^t"]\\
	X_*(H) \ar[r, "\psi_\h"] & X_*(H^\vee)
	\end{tikzcd}
\end{equation}
According to the following, the tropicalized comparison maps are compatible with the twist map.

\begin{theorem}\cite[Theorem~5.8]{ABHL2} \label{twistComparison}
Let $\theta, \theta'$ be reduced cluster charts on $G^{w_0,e}$. Let $\theta^\vee$ and ${\theta'}^\vee$ be the corresponding Langlands dual charts. Then, the following diagram commutes.
	\[
		\begin{tikzcd}[row sep=2.5em]
		(G^{w_0,e},\theta)^t \ar[r,"\varPsi_{\theta}^t"] \ar[d,"\zeta^t"] & (G^{\vee;w_0,e},{\theta}^\vee)^t \ar[d,"(\zeta^\vee)^t"]  \\
		(G^{w_0,e},\theta)^t \ar[r,"\varPsi_{\theta}^t"] & (G^{\vee;w_0,e}, \theta^\vee)^t 
		\end{tikzcd}
	\]
\end{theorem}

Additionally, the tropicalized comparison maps preserves the Berenstein-Kazhdan cone.


\begin{theorem}\cite[Theorem 5.21]{ABHL2}\label{theorem;ABHL2maintheorem} Assume $\theta=\theta_{\overline{\sigma}}$ is a reduced cluster chart for $G^{w_0,e}$, and $\theta^\vee= \theta_{\overline{\sigma}^\vee}$ is its dual chart.
	Then the comparison map $\varPsi_\theta^t$ restricts to an injective map
	\[
		\varPsi_\theta^t\colon (G^{w_0,e},\Phi_{BK},\theta)^t \hookrightarrow (G^{\vee;w_0,e},\Phi_{BK}^\vee, \theta^\vee)^t,
	\]
	which extends to a linear isomorphism of real BK cones: 
		\[
		\varPsi_\theta^t\colon (G^{w_0,e},\Phi_{BK},\theta)^t_{\mathbb{R}} \xrightarrow{\sim} (G^{\vee;w_0,e},\Phi_{BK}^\vee, \theta^\vee)^t_{\mathbb{R}}.
	\]
\end{theorem}

\section{Partial tropicalization} \label{PTSection}

This section recalls the definition and properties of the partial tropicalization of $K^*$ from \cite{ABHL1, ABHL2, AHLL}. In the process, it extends the results of \cite{ABHL1, ABHL2, AHLL} to arbitrary (twisted) cluster charts on $G^{w_0,e}$ (the results in \cite{ABHL1, ABHL2, AHLL} were only stated for cluster charts associated to reduced words). 

Some of the key results of this Section are as follows: Lemma~\ref{lemma;scaling domination} stating that functions dominated by the BK potential exponentially decay in the tropical limit, Lemma~\ref{lemma;wtpt properties} which describes the torus action on the partial tropicalization, Lemma~\ref{lemma;hwpt properties} which describes symplectic leaves of the partial tropicalization and Theorem~\ref{good coordinates} which introduces Darboux coordinates for the Poisson bracket $\pi_{-\infty}$.

\subsection{Coordinates on \texorpdfstring{$K^*$}{K*} from cluster charts on \texorpdfstring{$G^{w_0,e}$}{Gw0e}}\label{ChartsonKstarSection}

Fix a choice of $\theta \in \Theta(G^{w_0,e})$ and consider the framed positive variety with potential $(G^{w_0,e},\Phi_{BK}, \theta)$. Recall~\eqref{wt map} the weight map, $\wt \colon G^{w_0,e} \to H$, and~\eqref{hw map} the highest weight map, $\hw \colon G^{w_0,e} \to H$. 
By Lemma~\ref{lemma; wt hw homomorphism property}, their tropicalizations are $\Z$-linear maps
\begin{equation*}
	\begin{split}
		\wt^t  = (\wt\circ \theta)^t & \colon (G^{w_0,e},\theta)^t \to X_*(H), \\
		\hw^t  = (\hw\circ \theta)^t & \colon (G^{w_0,e},\theta)^t \to X_*(H). \\
	\end{split}
\end{equation*}
 Moreover, $\wt^t = X_*(\wt\circ \theta)$ and $\hw^t= X_*(\hw\circ \theta)$. Recall that $A\leq H$ is the real Lie subgroup  $A = \exp(\a)$. As sets, $K^*\cap G^{w_0,e} = \wt\n(A)$. By Lemma~\ref{lemma; wt hw homomorphism property},
\[
	\theta\n(K^*\cap G^{w_0,e})= (\wt\circ\theta)\n(A)
\] 
is a connected real Lie subgroup of $\C^{\times(\tilde r +m)}$. Let $\T_\theta$ denote the maximal compact subtorus of $(\wt\circ\theta)\n(A)$. When it is clear from context we write $\T =\T_\theta$. Since the maximal compact subgroup of $A$ is $\{e\}$, we have $\T \subset \ker(\wt \circ \theta)$ and  $\dim_\R\T = m$.

In what follows we will be concerned with the manifold $(G^{w_0,e},\theta)^t_\R\times \T_\theta$. 
If $z$ is a positive function on $(G^{w_0,e},\theta)$, then we abuse notation and write $z^t$ for the map $(G^{w_0,e},\theta)^t_\R\times \T_\theta\xrightarrow{\pr_1} (G^{w_0,e},\theta)^t_\R \xrightarrow{z^t} \R$.

\begin{definition} \label{definition;PTcoordtrans}
Let $\theta,\theta'\in \Theta(G^{w_0,e})$, and consider the manifolds $(G^{w_0,e},\theta)^t_\R \times \T_\theta$ and $(G^{w_0,e},\theta')^t_\R \times \T_{\theta'}$. For each open linearity chamber $C$ of the piecewise linear bijection
\[
X_*(\C^{\times(\tilde{r}+m)}) \otimes \R = (G^{w_0,e},\theta)^t_\R \xrightarrow{\id^t} (G^{w_0,e},\theta')^t_\R = X_*(\C^{\times(\tilde{r}+m)}) \otimes \R 
\]
consider the unique linear map 
\[
X_*(\C^{\times(\tilde{r}+m)}) \otimes \R \to X_*(\C^{\times(\tilde{r}+m)}) \otimes \R
\]
which agrees with $(\id^t)|_C$ on $C$. Let 
\begin{equation*}
e^{\id_C} \colon \C^{\times(\tilde{r}+m)} \to \C^{\times(\tilde{r}+m)}
\end{equation*}
 denote the associated group homomorphism. Because $(\wt\circ \id)^t = \wt^t$, we have $e^{\id_C}(\T_\theta) = \T_{\theta'}$. Define
\begin{equation}
\label{PTcoordtrans}
\id^{PT} = \coprod_{C}((\id^t)|_C \times e^{\id_C}) \colon \coprod_{C} C\times \T_\theta \to (G^{w_0,e},\theta')^t_\R \times \T_{\theta'},
\end{equation}
where the disjoint union ranges over all the open linearity chambers of $\id^t$. The map $\id^{PT}$ is the \emph{partially tropicalized change of coordinates}.
\end{definition}

 For all $s \neq 0$, there is an isomorphism of real Lie groups
\begin{equation}\label{equation;change of coordinates I}
	\mathfrak{P}_s \colon (G^{w_0,e},\theta)^t_\R \times \T \to (\wt\circ\theta)\n(A)
\end{equation}
defined by the property that 
\begin{equation}\label{equation;change of coordinates II}
	e^{\frac{s}{2}\langle \gamma,x\rangle}t^\gamma = \mathfrak{P}_s(x,t)^\gamma, \quad \forall \gamma \in X^*\C^{\times(\tilde r + m)},\, (x,t) \in (G^{w_0,e},\theta)^t_\R \times \T.
\end{equation}
Note that since $\T$ is contained in $\ker(\wt \circ \theta)$,
\begin{equation}\label{equation;XH vanishes on big T}
	t^{\langle \gamma,\wt^t\rangle} = 1, \quad \forall t\in \T, \gamma \in X^*(H).
\end{equation}
Recall that the standard coordinates $z_i$ on $\C^{\times(\tilde r +m)}$ are identified with  the standard basis of  $X^*(\C^{\times(\tilde r +m)})$.   With this identification, define the system of polar  coordinates 
$\lambda_i, \varphi_i$  on $(G^{w_0,e},\theta)^t_\R \times \T$:
\begin{equation}\label{equation; cluster coordinates on K^* II}
	\left\{ \begin{array}{rl}
		\lambda_i(x,t) &= \langle z_i,x\rangle, \\
		e^{\sqrt{-1}\varphi_i(x,t)} &= t^{z_i}. 
	\end{array}\right.
\end{equation}
The coordinates $\varphi_i$ are defined (modulo $2\pi$) for the indices $i$ such that $z_i$ 
takes on values outside of $\R_{>0}$ on the subset $G^{w_0,e}\cap K^*$. In other words, $\varphi_i$ are defined for indices $i$ such that  $z_i$ (viewed as a character of $\C^{\times(\tilde r +m)}$) is not in the image of  $ X^*(H)$ under $X^*(\wt\circ \theta)$.

\begin{definition}\label{definition;detropicalization}\label{equation;PiSTheta}
	The \emph{detropicalization map}, $L_s^\theta$, is the composition 
\[
	L_s^\theta := \theta \circ \mathfrak{P}_s \colon (G^{w_0,e},\theta)^t_\R \times \T \to  G^{w_0,e}\cap K^*.
\]
\end{definition}

\subsection{The definition of partial tropicalization} \label{section; definition of partialtrop}

Let $\theta \in\Theta(G^{w_0,e})$ be a (twisted) cluster chart on $G^{w_0,e}$ as in the previous section. By Corollary~\ref{corollary;homog} and Proposition~\ref{proposition;twisthomog}, there is $\Z$-module homomorphism,
\begin{equation}\label{equation; degree map}
	|\cdot | \colon X^*(\C^{\times(\tilde r + m)}) \to X^*(H) \times X^*(H),
\end{equation}
defined by the property that  $|z_i| = (\alpha_i, \beta_i)$ for all $i$, where $(\alpha_i,\beta_i)$ is the homogeneous degree of the $\theta$-coordinate $z_i$. If $|\gamma| = (\alpha,\beta)$, denote $|\gamma|_1 = \alpha$ and $|\gamma|_2 = \beta$.

\begin{proposition}\label{prop; coefficients}
Assume $G$ is of the form~\eqref{equation;niceformG}, and let $\theta \in  \Theta(G^{w_0,e})$.  Let $\{\cdot,\cdot\}^\C_{K^*}$ denote the $\C$-linear extension of the Poisson bracket of $\pi_{K^*}$. Restrict the $\theta$-coordinates $z_i \in \C[G^{w_0,e}]$ to functions on $G^{w_0,e}\cap K^*$. Then, there exists  $c_{z_i,z_j}$, $c_{z_i,\overline{z}_j}, c_{\overline{z}_i,\overline{z}_j}\in \C$, such that
    \begin{equation}\label{equation;bracketswithc} 
        \{z_i,z_j\}^\C_{K^*} =  z_iz_j c_{z_i,z_j}; \quad \quad \{\overline{z}_i,\overline{z}_j\}^\C_{K^*} =  \overline{z}_i\overline{z}_j c_{\overline{z}_i,\overline{z}_j}; \quad \quad
        \{z_i,\overline{z}_j\}^\C_{K^*}  =  z_i \overline{z}_j (c_{z_i,\overline{z}_j}+ f_{z_i,\overline{z}_j}), 
    \end{equation}
    where  $f_{z_i,\overline{z}_j}$ is a $\C$-linear combination of
    \begin{equation}\label{equation;domterms}
        \frac{(E_{-\alpha}\cdot z_i)(E_{-\alpha}\cdot \overline{z}_j)}{z_i \overline{z}_j},  \frac{(z_i \cdot E_{-\alpha})(\overline{z}_j\cdot E_{-\alpha})}{z_i\overline{z}_j}, \quad \alpha \in R_+.
    \end{equation}
    Moreover, 
    \begin{equation}\label{equation;mixedcoefficient}
    c_{z_i,\overline{z}_j} = \frac{\sqrt{-1}}{2} \Big( (|z_i|_1,|\overline{z}_j|_1)_{\mathfrak{h}^*} - (|z_i|_2,|\overline{z}_j|_2)_{\mathfrak{h}^*} \Big), \quad c_{z_i,z_j} =-c_{\overline{z}_i,\overline{z}_j} \in \sqrt{-1}\Q. 
\end{equation}
\end{proposition}

\begin{proof}
First consider brackets of the form $\{z_i,z_j\}=\{z_i,z_j\}^\C_{K^*}$. Consider the holomorphic Poisson-Lie structure $\pi_G^{hol}=(\mathrm{r}^{hol})^L-(\mathrm{r}^{hol})^R$ on $G$ given by the $r$-matrix $\mathrm{r}^{hol} = \frac{1}{\sqrt{-1}} \mathrm{r}$, where $\mathrm{r}$ is given by \eqref{equation;rmatrixholomorphic}.
 The bivector $\pi_{G}^{hol}$ restricts to a holomorphic Poisson structure on the double Bruhat cell $G^{w_0,e}$. Therefore, for holomorphic functions on $G^{w_0,e}$, the brackets $\{z_i,z_j\}^\C_{K^*}$ and $-\sqrt{-1}\{z_i,z_j\}_{G^{w_0,e}}$ are equal.
  By \cite[Theorem~4.18]{GSV} and \cite[Theorem~3.1]{GSV03}, the bracket of two $\theta$-variables on $G^{w_0,e}$ is log canonical, with rational structure constant.

Next, consider brackets of the form $\{\overline{z}_i,\overline{z}_j\}$. Since the bivector $\pi_{K^*}\in \Gamma(\wedge^2 T K^*)$ is real, we have $\{\overline{z}_i,\overline{z}_j\}= \overline{\{z_i,z_j\}}$. The brackets $\{z_i,z_j\}$ are pure imaginary, and so this amounts to $c_{z_i,z_j}  =-c_{\overline{z}_i,\overline{z}_j}$.

Finally, consider the mixed brackets $\{z_i,\overline{z}_j\}$. The claim follows directly from the expression for the bracket in Lemma~\ref{lemma;formulaformixedbrackets}, as well as the $P\times P$-homogeneity of (twisted) cluster variables established in Corollary~\ref{corollary;homog} and Proposition~\ref{proposition;twisthomog}.
\end{proof}

\begin{definition}\label{definition;PTdef}\label{equation;PTdef}
  Assume $G$ is of the form~\eqref{equation;niceformG}, and let $\theta \in  \Theta(G^{w_0,e})$. The \emph{partial tropicalization of $K^*$ with respect to} $\theta$ is the Poisson manifold $(PT(K^*,\theta),\pi_{-\infty}^\theta)$, where:
  \begin{enumerate}[label=(\arabic*)]
  	\item $PT(K^*,\theta) := (G^{w_0,e},\Phi_{BK},\theta)^t_\R(0) \times \T$.
  	\item $\pi_{-\infty}^\theta$ is defined such that:
  	\begin{equation}\label{equation;bracketPTdef}
  	\begin{split}
  		\{\lambda_i,\varphi_j\} =\sqrt{-1} (c_{z_i,\overline{z}_j} -c_{z_i,z_j}),\quad 
      \{\varphi_i,\varphi_j\}  =\{\lambda_i,\lambda_j\} = 0, 
  	\end{split}
  	\end{equation}
  	where:
  	\begin{enumerate}
  	\item  $\lambda_i,\varphi_j$ are the coordinates \eqref{equation; cluster coordinates on K^* II}, and
  	\item $c_{z_i,z_j}$, $c_{z_i,\overline{z}_j}$, and $c_{\overline{z}_i,\overline{z}_j}$ are as in~\eqref{equation;bracketswithc}.
  	\end{enumerate}
  \end{enumerate}
\end{definition}

\begin{remark}\label{specialformforK*}
	By the explicit formula in  \cite[Corollary 4.21]{GSV}, given a double reduced word $\mathbf{i}$ for $(w_0,e)$, in the chart $\theta_{\sigma(\mathbf{i})}$, the coefficient $c_{z_i,z_j}$ is given by
	\[
		c_{z_i,z_j}=-\frac{\sqrt{-1}}{2} \Big( (|z_i|_1,|z_j|_1)_{\mathfrak{h}^*} - (|z_i|_2,|z_j|_2)_{\mathfrak{h}^*} \Big)
	\]
	whenever $i<j$.
	For $i>j$, we have $c_{z_i,z_j}:=-c_{z_j,z_i}$. 
\end{remark} 

The following gives an interpretation of coordinate transformations, after partial tropicalization. It follows from~\cite[Theorem~6.23]{ABHL1}.

\begin{theorem} \label{theorem;PTpoissonisomorphism}
Assume $G$ is of the form~\eqref{equation;niceformG}, and let $\theta, \theta' \in  \Theta(G^{w_0,e})$. Recall the map $\id^{PT}$ defined in~\eqref{PTcoordtrans}. Then, the restriction of $\id^{PT}$ to a densely defined map
\[
\id^{PT} \colon PT(K^*,\theta)\to PT(K^*,\theta')
\]
has $(\id^{PT})_*(\pi_{-\infty}^\theta) = \pi_{-\infty}^{\theta'}$ wherever it is defined.
\end{theorem}
%

\subsection{Partial tropicalization as a limit}  \label{section; partialtrop is limit}

For $\delta>0$, the \emph{$\delta$-interior of $PT(K^*,\theta)$} is
\[
  PT(K^*,\theta,\delta) = (G^{w_0,e},\Phi_{BK},\theta)^t_\R(\delta) \times \T.
\]
The following results show that for all $\delta>0$, $\pi_{s}^\theta := (L_s^\theta)^*(s\pi_{K^*})$ converges uniformly to $\pi_{-\infty}^\theta$ on $PT(K^*,\theta,\delta)$ as $s\to -\infty$. First, some notation:

\begin{notation}[Big-O notation]\label{notation;bigo}
	A family of differential forms or multi-vector fields on $PT(K^*,\theta,\delta)$ parameterized by $s$ is $O(e^{s\delta})$ if its coefficients are $O(e^{s\delta})$ as functions on $(-\infty,0)\times PT(K^*,\theta,\delta)$ when it is written in the coordinates $\lambda_i$, $\varphi_i$~\eqref{equation; cluster coordinates on K^* II}. 
\end{notation}

\begin{lemma} \cite[Lemma~6.17]{ABHL1} \label{lemma;scaling domination}
Let $\theta\in \Theta(G^{w_0,e})$. If $f\in \C(G^{w_0,e})$ is dominated by $\Phi_{BK}$, then $(f\circ L_s^\theta)|_{PT(K^*,\theta,\delta)} = O(e^{s\delta})$ for all $\delta>0$.
\end{lemma}

\begin{theorem}\label{theorem;bracketconvergence}
  Assume $G$ is of the form~\eqref{equation;niceformG}, and let $\theta \in  \Theta(G^{w_0,e})$. Then, for all $\delta>0$
  \[
   \pi^\theta_s\vert_{PT(K^*,\theta,\delta)} = \pi_{-\infty}^\theta + O(e^{s\delta}).
  \]
\end{theorem}

\begin{proof}
By Corollary~\ref{corollary;domination} the functions $f_{z_i,\overline{z}_j}$ in~\eqref{equation;bracketswithc}, considered as functions on $G^{w_0,e}$, are dominated by the potential $\Phi_{BK}$, up to complex conjugation of some terms. The analytical significance of this is explained by Lemma~\ref{lemma;scaling domination}, and the proof is then exactly as in \cite[Theorem~6.18]{ABHL1}.
\end{proof}

\subsection{Structure maps and symplectic leaves} \label{PTstructureSection}

 The structure maps of the partial tropicalization are:
\begin{equation} \label{equation;hwwtPT}
 \begin{tikzcd}
 \hw^{PT} \colon  (G^{w_0,e},\theta)^t_\R \times \T  \ar[r,"\pr_1"] & (G^{w_0,e},\theta)^t_\R \ar[r,"\hw^t"] & X_*(H)\otimes_\Z \R \ar[rr,"-\sqrt{-1} \psi_\h" ]&& \t^*, \\
 \wt^{PT} \colon  (G^{w_0,e},\theta)^t_\R \times \T  \ar[r,"\pr_1"] & (G^{w_0,e},\theta)^t_\R \ar[r,"\wt^t"] &  X_*(H)\otimes_\Z \R \ar[rr,"-\sqrt{-1} \psi_\h"]&& \t^*.
 \end{tikzcd}
 \end{equation}
Although they are defined on $(G^{w_0,e},\theta)^t_\R \times \T$, $\hw^{PT}$ and $\wt^{PT}$ will sometimes denote their restrictions to the subspace $PT(K^*,\theta)$. 

\begin{lemma}\label{lemma;wtPT commuting diagram}
Let $\theta\in \Theta(G^{w_0,e})$. Then, the following diagram commutes for all $s\neq 0$.
\begin{equation}\label{equation;Tcommutativediagram}
\begin{tikzcd}
(G^{w_0,e},\theta)^t_\R \times \T \ar[d,"\wt^{PT}"] \ar[r,"L_s^\theta"] & K^*\cap G^{w_0,e} \ar[d," \wt"] \\
\ttt^* \ar[r,"E_s"] & A
\end{tikzcd}
\end{equation}
\end{lemma}

\begin{proof} For $(x,t) \in  (G^{w_0,e},\theta)^t_\R \times \T$,
\[
	E_s\circ \wt^{PT}(x,t) = \exp\left(\frac{s\sqrt{-1}}{2}\psi_\hhh\n(-\sqrt{-1}\psi_\hhh\circ \wt^t (x))\right)= \exp\left(\frac{s}{2}\wt^t (x)\right)
\]
Thus it suffices to prove
\begin{equation*}
	\exp\left(\frac{s}{2}\wt^t (x) \right) = \wt(L_s^\theta(x,t)), \quad \forall (x,t) \in (G^{w_0,e},\theta)^t_\R \times \T.
\end{equation*}
	Let $\gamma \in X^*(H)$ and $(x,t) \in (G^{w_0,e},\theta)^t_\R \times \T$ arbitrary. Applying the definitions, 
	\begin{equation*}
		\begin{split}
			\wt(L_s^\theta(x,t))^\gamma & = (\wt \circ \theta \circ \mathfrak{P}_s(x,t))^\gamma\\
		 	& = \mathfrak{P}_s(x,t)^{\langle \gamma,\wt^t\rangle} \\
		 	& = \exp\left(\frac{s}{2}\wt^t (x) \right)^\gamma t^{\langle \gamma,\wt^t\rangle}\\
		 	& = \exp\left(\frac{s}{2}\wt^t (x) \right)^\gamma.
		\end{split}
	\end{equation*}
	The fourth equality follows by~\eqref{equation;XH vanishes on big T}. The lemma follows since $\gamma$ is arbitrary.
\end{proof}

The degree map (cf.~\eqref{equation; degree map}) of a (twisted) cluster chart $\theta$ determines a homomorphism 
\[
	\varphi_\theta \colon H \times H \to \C^{\times(\tilde r + m)}, \quad \varphi_\theta(h_1,h_2)^\gamma = h_1^{|\gamma|_1}h_2^{|\gamma|_2} \quad \forall \gamma \in X^*\C^{\times(\tilde r + m)}.
\]
This in turn determines an action of $H\times H$ on $\C^{\times(\tilde r + m)}$ by multiplication such that $\theta$ is $H\times H$-equivariant. The image of the anti-diagonal subgroup $\{(h,h\n) \in H\times H \mid h \in H\} \cong H$ is contained in the kernel of $\wt\circ \theta$ since $|\langle \gamma,\wt^t\rangle|_1 = |\langle \gamma,\wt^t\rangle|_2$ for all $\gamma \in X^*(H)$. Thus, the image of the composition 
\[
	 \iota_\theta\colon T \hookrightarrow H \xrightarrow{h \mapsto (h,h\n)} H\times H \xrightarrow{\varphi_\theta} \C^{\times(\tilde r + m)}
\]
is contained in $\T$. Let $T$ act on $(G^{w_0,e},\theta)^t_\R\times \T$ by translation on $\T$ with respect to $\iota_\theta$, and consider the dressing (conjugation) action of $T$ on $K^*\cap G^{w_0,e}$. Then, $\mathfrak{P}_s$ and  $L_s^\theta = \theta \circ \mathfrak{P}_s$ are each $T$-equivariant.

Recall that the action of $T$ on $K^*$ by conjugation coincides with the dressing action. For all $s\neq 0$, the dressing action of $T$ on $(K^*,s\pi_{K^*})$ is Hamiltonian with moment map $E_s\n \circ [\cdot]_0$ (see the discussion preceding Proposition~\ref{moser flow equivariance}). According to the following lemma, this Hamiltonian action survives the limit $s\to -\infty$.

\begin{lemma}\label{lemma;wtpt properties} Let $G$ be of the form \eqref{equation;niceformG}, and let $\theta\in \Theta(G^{w_0,e})$. The kernel of the homomorphism $\iota_\theta\colon T \to \T$ is the center of $K$. Moreover:
\begin{enumerate}[label=(\arabic*)]
    \item For all $s\neq 0$, the action of $T$ on $(G^{w_0,e},\theta)^t_\R \times \T$ defined by $\iota_\theta$ is Hamiltonian with respect to $\pi_s^\theta$ with moment map $\wt^{PT}$.
    \item The action of $T$ on $PT(K^*,\theta)$ defined by $\iota_\theta$ is Hamiltonian with respect to $\pi_{-\infty}^\theta$ with moment map $\wt^{PT}$.
\end{enumerate}
    
\end{lemma}

\begin{proof}
To see that the kernel of $\iota_\theta$ is the center of $K$, it suffices to observe that $L_s^\theta$ is $T$-equivariant and the kernel of the dressing action of $T$ on $K^*$ equals the center of $K$.

The restriction of $[\cdot]_0$ to $G^{w_0,e}\cap K^*$ equals $\wt$. Since $L_s^\theta$ is $T$-equivariant, the action of $T$ on $(G^{w_0,e},\theta)^t_\R \times \T$ defined by $\iota_\theta$ is Hamiltonian with respect to $\pi_s^\theta$ with moment map $E_s\n\circ \wt\circ L_s^\theta$. By~\eqref{equation;Tcommutativediagram}, $E_s\n\circ \wt\circ L_s^\theta = \wt^{PT}$, which completes the proof of item 1.

Since $\iota_\theta$ does not depend on $s$,  fundamental vector fields $\underline{X} \in \mathfrak{X}((G^{w_0,e},\theta)^t_\R \times \T)$, $X\in \ttt$, for the action defined by $\iota_\theta$ do not depend on $s$. By Theorem \ref{theorem;bracketconvergence},
\[
\underline{X} = \lim_{s\to -\infty} (\pi_s^\theta)^\sharp(d\langle \wt^{PT},X\rangle) = (\pi_{-\infty}^\theta)^\sharp(d\langle \wt^{PT},X\rangle)
\]
at points in $PT(K^*,\theta, \delta)$. This completes the proof of item 2. 
\end{proof}


\begin{lemma}\label{lemma;hwpt properties}
	Let $\theta\in \Theta(G^{w_0,e})$. Then:
	\begin{enumerate}[label=(\arabic*)]
		\item \label{fiberssymplleaves}The fibers of $\hw^{PT}\colon PT(K^*,\theta) \to \ttt^*$ are the symplectic leaves of $(PT(K^*,\theta), \pi_{-\infty}^\theta)$. 
		\item The image of $\hw^{PT}\colon PT(K^*,\theta) \to \ttt^*$ is $\mathring{\t}_+^*$.
		\item For all $\lambda \in \mathring{\t}_+^*$, let $\omega_{-\infty}^\theta$ denote the symplectic structure on $(\hw^{PT})\n(\lambda)$ defined by $\pi_{-\infty}^\theta$.  Then, 
\begin{equation}
    \Vol((\hw^{PT})\n(\lambda),\omega_{-\infty}^\theta) = \Vol(\mathcal{O}_\lambda,\omega_\lambda).
\end{equation}
	\end{enumerate}
\end{lemma}

\begin{proof}
	The proof of Item 1.~is the same as the proof of \cite[Proposition 6.3]{ABHL2} which handles the special case $\theta=\theta_{\sigma(\mathbf{i})}$. Item 2.~follows from Theorem~\ref{theorem;BKmaintheorem}. Item 3.~is Theorem~6.5 and Remark~6.6 of \cite{ABHL2}. 
\end{proof}

The next lemma is crucial for estimating Hamiltonian flows in Section 5.  It was proved for the special case $\theta=\theta_{\sigma(\mathbf{i})}$ in Theorem~6.11 of \cite{ABHL2}. The proof in general is the same.

\begin{lemma}\label{lemma;hwpt casimir approximation}
	Let $\theta\in \Theta(G^{w_0,e})$. Then, For all $X \in \ttt$,
        \begin{equation}
            \langle \hw^{PT},X \rangle = \langle  \mathcal{S} \circ E_s\n \circ L_s^\theta,X\rangle + O(e^{s\delta}),
        \end{equation}
        as a function on $PT(K^*,\theta,\delta)$.  Moreover, 
        \begin{equation}
            d\langle \hw^{PT},X \rangle = d\langle  \mathcal{S} \circ E_s\n \circ L_s^\theta,X\rangle + O(e^{s\delta}),
        \end{equation}
        as linear operators.
\end{lemma}


\subsection{Properties of the partial tropicalization}  \label{section; properties of partialtrop}

This section establishes basic results about $\pi_{-\infty}^\theta$. Section~\ref{section; brackets in special charts} describes the special case $\theta= \theta_{\sigma(\mathbf{i})}$ or $\theta_{\sigma(\mathbf{i})}^\zeta$, where $\mathbf{i}$ is a double reduced word for $(w_0,e)$. Versions of these results were proved in \cite{ABHL2}. Section~\ref{section;bracketsgeneral} describes the case for general $\theta\in \Theta(G^{w_0,e})$.

\subsubsection{Brackets in special charts}\label{section; brackets in special charts}
%
%

\begin{theorem} \label{theorem;brackettheorem} 
Let $\mathbf{i} = (i_1,\dots, i_m)$ be a double reduced word for $(w_0,e)$. 
 If $\theta=\theta_{\sigma(\mathbf{i})}$, then for all $j \in [-r, -1]\cup [1,m]$ and $k\in [1,m]$:
	\begin{align*}
	\{ \lambda_j,\varphi_k\} & = 0 & \text{ if } j \geqslant k, \\
	\{ \lambda_j,\varphi_k\} & =  (\omega_{i_j},\omega_{i_k})_{\h^*}- (\omega_{i_j},s_{i_{j+1}} \cdots s_{i_k} \omega_{i_k})_{\h^*}  & \text{ if } j<k.
	\end{align*}
	If $j<-r$, then $\{\lambda_j,\varphi_k\} = 0$ for all $k\in [1,m]$. 
	In particular, if $z_j=\Delta_{w_0\omega_{|i_j|},\omega_{|i_j|}}$ is a $\theta$-coordinate and $\lambda_j = z_j^t$, then $\{\lambda_j,\varphi_k\} = 0$ for all $k\in[1,m]$.
\end{theorem}

\begin{proof}
We refer to the definition~\eqref{equation;bracketPTdef} of $\pi_{-\infty}^\theta$ and use invariance of $(\cdot,\cdot)_{\h^*}$ under the Weyl group. The coefficient $c_{z_j,\overline{z}_k}$ is given in~\eqref{equation;mixedcoefficient}. The coefficient $c_{z_j,z_k}$ is given in Remark~\ref{specialformforK*}. The last statement can be shown directly from the previous ones; it also follows from Lemma~\ref{lemma;hwpt properties}\ref{fiberssymplleaves}.
\end{proof}

\begin{corollary} \label{corollary;triangularbracket}
Let $\mathbf{i} = (i_1,\dots, i_m)$ be a double reduced word for $(w_0,e)$, and let $\theta=\theta_{\sigma(\mathbf{i})}$. For $j\in [1,m]$, let $j^- = \max\{k\in [-r,-1]\cup [1,m] \mid |i_{k}| =|i_j|\}$. Let $B$ denote the $m\times m$ matrix with entries
\[
B_{j,k} = \{\lambda_{j^-},\varphi_{k}\}.
\]
Then $B = X\n Y$ where $Y$ is an upper triangular unimodular matrix with nonnegative entries and $X$ is the diagonal matrix with $X_{j,j} = d_{|i_j|}$. 

Consequently, there is a Lie group automorphism $C_\varphi\colon \T_\theta \to \T_\theta$ so that, if $\varphi'_{k} = \varphi_k\circ C_\varphi$, then the matrix $B'$ with $B'_{j,k} =\{d_{|i_j|} \lambda_{j^-},\varphi'_{k}\}$ is equal to the identity matrix.
\end{corollary}

\begin{proof}
 Then, by Theorem~\ref{theorem;brackettheorem} and invariance of $(\cdot,\cdot)_{\h^*}$ with respect to the Weyl group,
\begin{align*}
B_{j,j} = \{\lambda_{j^-},\varphi_{j} \} & = (\omega_{i_{j^-}},\omega_{i_{j}})_{\h^*} - (\omega_{i_j},s_{1+i_{j^-}} \cdots s_{i_{j}} \omega_{i_{j}})_{\h^*} \\
& = (\omega_{i_j},\omega_{i_{j}})_{\h^*} - (\omega_{i_j},s_{i_{j}} \omega_{i_{j}})_{\h^*} \\
& = (\omega_{i_j},\alpha_{i_{j}})_{\h^*} = 1/d_{|i_j|}.
\end{align*}
If $j>j^-\geqslant k$, then $B_{j,k} = 0$ by Theorem~\ref{theorem;brackettheorem}. If $j>k>j^-$, then again by invariance
\[
B_{j,k} = \{\lambda_{j^-},\varphi_{k} \} = (\omega_{i_{j^-}},\omega_{i_{k}})_{\h^*} - (\omega_{i_j},s_{1+i_{j^-}} \cdots s_{i_{k}} \omega_{i_{k}})_{\h^*} =0.
\]
Finally, if $j<k$ then
\[
B_{j,k} = \{\lambda_{j^-},\varphi_{k} \} = (\omega_{i_{j^-}}, \omega_{i_k}-s_{1+i_{j^-}} \cdots s_{i_{k}} \omega_{i_{k}}) = c (\omega_{i_j}, \alpha_{i_j})_{\h^*}
\]
for some nonnegative integer $c$, since $\omega_{i_k}-s_{1+i_{j^-}} \cdots s_{i_{k}} \omega_{i_{k}}$ is a positive integer linear combination of simple roots $\alpha_i$ and $(\omega_{i_j},\alpha_i) = \delta_{i_j,i}$. The desired automorphism $C_\varphi$ can be constructed by reducing $Y$ to the identity matrix by multiplying by a unimodular matrix on the right.
\end{proof}

\subsubsection{Brackets in general charts and relationship with canonical bases} \label{section;bracketsgeneral}

By the discussion of Section~\ref{section;geometriccrystals}, the cone factor of $PT(K^*,\theta)$ is related to the representation theory of $G^\vee$. Ultimately, we want to connect the Poisson geometry of $\mathfrak{k}^*$ with the representation theory of $K\subset G$.
%

\begin{proposition} \label{goodformeasy} 
Assume $G$ is of the form~\eqref{equation;niceformG}.
Let $\mathbf{i}$ be a double reduced word for $(w_0,e)$, let $\theta=\theta_{\overline{\sigma}(\mathbf{i})}\in \Theta(G^{w_0,e})$, and let $\theta^\vee = \theta_{\overline{\sigma}(\mathbf{i})^\vee}\in \Theta(G^{\vee;w_0,e})$. 
Let $\lambda_j,\varphi_k$ and $\lambda_j^\vee,\varphi_k^\vee$ be coordinates on $(G^{w_0,e},\theta)^t_\R \times \T_\theta$ and $(G^{\vee; w_0,e},\theta^\vee)^t_\R \times \T_{\theta^\vee}$, respectively, as in~\eqref{equation; cluster coordinates on K^* II}. Identify $\T=\T_{\theta^\vee}=\T_\theta$ by setting $\varphi_k = \varphi^\vee_k$ for all $k$ where $\varphi_k$ is defined. 

Given the fixed isomorphism $\C^{\tilde{r}}\cong H$ as in~\eqref{equation;splitH}, let $X_1,\dots,X_{\tilde{r}} $ be the associated basis of $X_*(H)$. 
Let $x^\vee_j = \lambda^\vee_{j^-}$ for $j\in [1,m]$ and $x^\vee_{-j} = \langle X_j, (\hw^\vee)^t \rangle$. 
There are coordinates $\upsilon^\vee_1,\dots,\upsilon^\vee_m$, modulo $2\pi$, so that
\begin{enumerate}[label=(\arabic*)]
\item  Under the constant Poisson structure
\[
(\varPsi_\theta^t \times \id_\T)_* (\pi^\theta_{-\infty})
\]
on $(G^{\vee; w_0,e},\Phi^\vee_{BK},\theta^\vee)^t_\R(0) \times \T$, one has
\label{bracketitem1easy} 
		\begin{align*} &\{x^\vee_{j},\upsilon^\vee_k\}  = 0 &  \text{ for all } j \in [-\tilde{r},-1], k \in [1,m]; \\
		&\{x^\vee_j,\upsilon^\vee_k\} = \delta_{j,k}  & \text{ for all } j,k \in [1,m].
		\end{align*}
\item \label{bracketitem0easy} $x^\vee_{-k}\circ (\varPsi^t_\theta \times \id_{\T}) = \langle X_k, \hw^{PT}\rangle$ for all $k\in[1,\tilde{r}]$.
\end{enumerate}
\end{proposition}

\begin{proof}
Let $\hat{\theta} =\theta_{\sigma(\mathbf{i})}$ be the unreduced cluster chart on $G^{w_0,e}$ which is related to $\theta$ as in Remark~\ref{reducedchangeofcoords}. 
 For each $\hat{\theta}$-coordinate $\hat{z}_j$, view the corresponding coordinates $\hat{\lambda}_j, \hat{\varphi}_j$ as functions on $PT(K^*,\theta)$ by precomposing with
\[
\id^{PT}\colon PT(K^*,\theta)\to PT(K^*,\hat{\theta});
\]
notice that in this case $\id^{PT}$ is globally defined because the coordinate transformation~\eqref{isomorphismredcoords} is Laurent monomial.  Let $\hat{\varphi}_k'$ be related to $\hat{\varphi}_k$ as in Corollary~\ref{corollary;triangularbracket}, and set $\upsilon^\vee_k =\hat{\varphi}_k'$, for $k=1,\dots, m$.

For $\hat{z}_j = \Delta_{u_{j}\omega_{|i_j|},\omega_{|i_j|}}$, write $u_j \omega_{|i_j|} = \sum_{i=1}^r c_i w_0\omega_i$ for $c_i\in \Z$. If $z_j = \hat{z}_j |_{L^{w_0,e}}$, we have 
\[
\lambda_j = \hat{\lambda}_j -\left(\sum_{i=1}^r c_i \Delta_{w_0\omega_i,\omega_i}^t\right).
\]
Due to the last statement of Theorem~\ref{theorem;brackettheorem}, the functions $\Delta_{w_0\omega_i,\omega_i}^t$ are Casimirs of $\pi_{-\infty}^{\theta}$. 
By Corollary~\ref{corollary;triangularbracket}, if $j>0$ then
\[
\{x^\vee_j\circ \varPsi_\theta^t,\upsilon_k\}  = \{\lambda_{j^-}\circ \varPsi_\theta^t, \hat{\varphi}_k' \} = \{d_{|i_j|} \lambda_{j^-}, \hat{\varphi}_k'\} = \delta_{j,k},
\]
as desired. If $j<0$ then $\{x_j,\upsilon_k\} = 0$ by~\ref{bracketitem0easy} together with Lemma~\ref{lemma;hwpt properties}\ref{fiberssymplleaves}. This establishes condition~\ref{bracketitem1easy}. The condition~\ref{bracketitem0easy} follows from the definition of $x^\vee_{-k}$ and the commutivity of the first diagram in~\eqref{comparisonlifts}.
%
%
%
%
%
%
%
%
%
\end{proof}

The following extends Proposition~\ref{goodformeasy} to include arbitrary (twisted) reduced cluster charts.

\begin{theorem}\label{good coordinates}
Let $\mathbf{i} = (i_1,\dots, i_m)$ be a double reduced word for $(w_0,e)$ and let $(\bm{z}_\sigma,\sigma)\in |\sigma(\mathbf{i})|$. Let $\theta=\theta_{\overline{\sigma}}$ or $\zeta\circ \theta_{\overline{\sigma}}$, and let $\theta^\vee = \theta_{\overline{\sigma}^\vee}$ or $\zeta^\vee \circ  \theta_{\overline{\sigma}^\vee}$, respectively. Let $\lambda_j,\varphi_k$ be coordinates on $(G^{w_0,e},\theta)^t_\R \times \T_\theta$ as in~\eqref{equation; cluster coordinates on K^* II}.

For the isomorphism $\C^{\tilde{r}}\cong H$ as in~\eqref{equation;splitH}, let $X_1,\dots,X_{\tilde{r}} $ be the associated basis of $X_*(H)$. There are linear coordinates $x_{-\tilde{r}},\dots,x_{-1},x_1,\dots,x_m$ on $(G^{w_0,e},\theta)^t_\R$, and coordinates $\upsilon_1,\dots,\upsilon_m$ modulo $2\pi$ on $\T_\theta$, so that:
\begin{enumerate}[label=(\arabic*)]
\item \label{bracketitem1}  Under the constant Poisson structure $\pi^\theta_{-\infty}$
on $PT(K^*,\theta)$, 
%
%
%
%
%
%
%
%
		\begin{align*} &\{x_{j},\upsilon_k\}  = 0 &  \text{ for all } j \in [-\tilde{r},-1], k \in [1,m]; \\
		&\{x_j,\upsilon_k\} = \delta_{j,k}  & \text{ for all } j,k \in [1,m].
		\end{align*}
		\item \label{bracketitem-1} \label{bracketitem0} $x_{-j} = \langle X_j, \hw^{PT}  \rangle$ for all $j\in[1,\tilde{r}]$.
 \item \label{bracketitem3}There is a unimodular map $\R^{\tilde{r}+m} \cong (G^{\vee;w_0,e},\theta^\vee)^t_\R$ which takes the image of $PT(K^*,\theta)$ under
 \begin{equation} \label{momentxmap}
 (q,t) \mapsto (x_{-\tilde{r}}(q),\dots,x_{-1}(q),x_1(q),\dots,x_m(q))\in \R^{\tilde{r}+m}, \quad (q,t)\in PT(K^*,\theta).
 \end{equation}
  to  $(G^{\vee;w_0,e},\Phi_{BK}^\vee,\theta^\vee)^t_\R(0)$.
\end{enumerate}
\end{theorem}

\begin{proof} 
Write $\tilde{\theta} = \theta_{\overline{\sigma}(\mathbf{i})}$ and $\tilde{\theta}^\vee=\theta_{\overline{\sigma}(\mathbf{i})^\vee}$. Identify $\T_\theta \cong \T_{\theta^\vee}$ and $\tilde{\T} \cong \tilde{\T}^\vee$ by putting $\varphi_k = \varphi_k^\vee$ and $\tilde{\varphi}_k = \tilde{\varphi}_k^\vee$, using notation for coordinate functions as in~\eqref{equation; cluster coordinates on K^* II}.
Due to the commutivity of~\eqref{twistComparison}, as well as Theorem~\ref{comparisonCompat} (if $\theta$ is a twisted reduced cluster chart), the diagram commutes: 
\begin{equation} \label{PTcommutingdiagramhard}
\begin{tikzcd}
PT(K,\theta)  \ar[r,"\varPsi^t_\theta\times \id_{\T_\theta}"] \ar[d, "\id^{PT}"] & (G^{\vee ; w_0,e},\Phi_{BK}^\vee,\theta^\vee)^t_\R(0) \times \T_{\theta^\vee}  \ar[d,"\id^{PT}"] \\
PT(K,\tilde{\theta})  \ar[r,"\varPsi^t_{\tilde{\theta}} \times \id_{\T_{\tilde{\theta}}}"] & (G^{\vee ; w_0,e},\Phi_{BK}^\vee,\tilde{\theta}^\vee)^t_\R(0) \times \T_{\tilde{\theta}^\vee}
\end{tikzcd}
\end{equation}
wherever the vertical arrows are defined.
 Let $C\subset (G^{w_0,e},\Phi_{BK},\theta)_\R^t(0)$ be an open linearity chamber of the map $\id^t \circ \varPsi^t_\theta = \varPsi^t_{\tilde{\theta}} \circ \id^t$. Consider the restriction of the map in~\eqref{PTcommutingdiagramhard} to $C\times \T_\theta$:
 \[
C \times \T_\theta \to (G^{\vee ; w_0,e},\tilde{\theta}^\vee)^t_\R \times \T_{\tilde{\theta}^\vee}.
\]
Linearly extend this to a map
\[
F\colon PT(K^*,\theta) \to (G^{\vee ; w_0,e},\tilde{\theta}^\vee)^t_\R \times \T_{\tilde{\theta}^\vee}.
\]
Notice that, by following the two sides of the diagram~\eqref{PTcommutingdiagramhard}, one can decompose $F$ as 
\[
F=F_2 \circ (\varPsi^t_\theta \times \id_{\T_\theta}) = (\varPsi^t_{\tilde{\theta}} \times \id_{\T_{\tilde{\theta}}}) \circ F_1,
\]
so that, when restricted to $C\times \T_\theta$ (resp. $\varPhi^t_\theta(C) \times \T_{\theta^\vee}$), the map $F_1$ agrees with $\id^{PT}$ (resp. $F_2$ agrees with $\id^{PT}$). Note that each $F_i$ is the product of a unimodular map $f_i$ in the first factor with a Lie group isomorphism $e^{f_i}$ in the second.
Let $x^\vee_j$ and $\upsilon^\vee_k$ be the functions on $(G^{\vee ; w_0,e},\tilde{\theta}^\vee)^t_\R \times \T_{\tilde{\theta}^\vee}$ constructed in Proposition~\ref{goodformeasy}. Let 
\[
x_j = x^\vee_j\circ F,\qquad \upsilon_k = \upsilon^\vee_k \circ F. 
\]
We will show that $x_j$ and $\upsilon_k$ have the desired properties.


First, consider the unique constant Poisson structure $\tilde{\pi}$ on $(G^{\vee ; w_0,e},\tilde{\theta}^\vee)^t_\R \times \T_{\tilde{\theta}^\vee}$ which coincides with $(\varPsi_{\tilde{\theta}}^t\times \id_{\T_{\tilde{\theta}}})_* (\pi_{-\infty}^{\tilde{\theta}})$ on the open subset $(G^{\vee ; w_0,e},\Phi_{BK}^\vee,\tilde{\theta}^\vee)^t_\R(0) \times \T_{\tilde{\theta}^\vee}$. By Theorem~\ref{theorem;PTpoissonisomorphism}, the map $F$ is a Poisson map on the open subset $C\times \T_\theta$ of its domain. Since $F$ is the product of a linear map and a fixed Lie group isomorphism, and because both  $\pi^{\theta}_{-\infty}$ and $\tilde{\pi}$ are constant, the map $F$ is a Poisson map on its entire domain. So, by by Proposition~\ref{goodformeasy}\ref{bracketitem1easy}, the Poisson brackets $\{x_j,\upsilon_k\}$ with respect to $\pi^\theta_{-\infty}$ have the desired form~\ref{bracketitem1}.

Second, by definition $x_{-j} = \langle X_j, (\hw^\vee)^t\rangle \circ F$. Restricting to the open set $C$, one then has
\[
x_{-j} = \langle X_j, (\hw^\vee)^t\rangle \circ (\varPsi_{\tilde{\theta}}^t \times \id_{\T_{\tilde{\theta}}}) \circ \id^{PT}.
\]
By Proposition~\ref{goodformeasy}\ref{bracketitem0easy}, $x_{-j} = \langle X_k, \hw^{PT}  \rangle$ on $C\times \T_\theta$. If two linear maps agree on an open subset of their domain, they are equal, and so $x_{-j}$ has the desired form~\ref{bracketitem0} on all of $PT(K^*,\theta)$.

Finally, the map~\eqref{momentxmap} can be decomposed
\[
PT(K^*,\theta) \xrightarrow{\pr_1} (G^{w_0,e},\theta)^t_\R \xrightarrow{\varPsi^t_\theta} (G^{\vee;w_0,e},\theta^\vee)^t_\R(0) \xrightarrow{f_2} (G^{\vee;w_0,e},\tilde{\theta}^\vee)^t_\R \cong \R^{\tilde{r}+m}.
\]
The isomorphism $ (G^{\vee;w_0,e},\tilde{\theta}^\vee)^t_\R \cong \R^{\tilde{r}+m}
$ given by the choice of coordinates $x^\vee_j$ is unimodular, and the map $f_2$ is unimodular. Also, $\varPsi^t_\theta\circ \pr_1(PT(K^*,\theta)) = (G^{\vee;w_0,e},\Phi^\vee_{BK},\theta^\vee)^t_\R(0)$. This gives condition~\ref{bracketitem3}.
%
%
%
\end{proof}

\section{Main Results: Construction of action-angle coordinates} \label{mainresultssection}

We are now able to prove our main results. We first construct action-angle coordinates on the space $K\times \mathring{\ttt}_+^*$ in Sections \ref{section; statement of main theorem} and \ref{section; proof of main theorem}. These are then used in Section \ref{section; charts on multiplicity free} to build action-angle coordinates on multiplicity free spaces through some standard reduction arguments.

\subsection{Big action-angle coordinate charts on the model space \texorpdfstring{$K\times \mathring{\ttt}_+^*$}{K*t}}\label{section; statement of main theorem}

Fix an arbitrary (twisted) reduced cluster chart $\theta$ on $G^{w_0,e}$. As in~\eqref{equation;splitH} and Theorem~\ref{good coordinates}, we have a fixed isomorphism $\C^{\times \tilde{r}} \cong H$ and denote $X_1,\dots,X_{\tilde{r}}$ the associated basis of $X_*(H)$.
Let:
\begin{equation}
\label{goodcoords1}
x_{-\tilde{r}},\dots,x_{-1},x_1, \dots, x_m, \upsilon_1, \dots, \upsilon_m
\end{equation} be the linear coordinates on $(G^{w_0,e},\theta)^t_\R \times \T$ given by Theorem~\ref{good coordinates}. Let
\begin{equation} \label{goodcoords2}
\Upsilon_1, \dots , \Upsilon_{\tilde r}
\end{equation} denote coordinates on $T$ defined modulo $2\pi$ by the dual basis of $X_1, \dots , X_{\tilde r}$. 

As $\theta$ is now fixed, we will frequently suppress it from notation (e.g.~$L_s^\theta = L_s$). Moreover, denote
\[
\mathcal{C} := (G^{w_0,e},\Phi_{BK},\theta)^t_\R(0), \qquad \mathcal{C}(\delta) := (G^{w_0,e},\Phi_{BK},\theta)^t_\R(\delta).
\]
With this notation, $PT(K^*,\theta) = \mathcal{C}\times \T$.

\begin{definition}\label{defintion; symplectic structure on toric space}
	Given the coordinates $x_{-i},x_j, \upsilon_j,\Upsilon_i$ and $\mathcal{C}$ as defined above (for a particular choice of $\theta$), let
	\[
		\omega_{-\infty} := d\upsilon_1\wedge dx_1   + \dots + d\upsilon_m\wedge dx_m  + d\Upsilon_1\wedge dx_{-1} + \dots + d\Upsilon_{\tilde r}\wedge dx_{-\tilde r} .
	\]
\end{definition}
The symplectic manifold $(\mathcal{C}\times \T\times T,\omega_{-\infty})$ has the following immediate properties:
\begin{enumerate}[label=(\Alph*)]
	\item\label{sympPT property A} The map $(\mathcal{C}\times \T\times T,\omega_{-\infty}) \to (\mathcal{C}\times \T,\pi_{-\infty})$, $(x,t,t') \mapsto (x,t)$, is Poisson. This follows from Theorem~\ref{good coordinates}.
	\item\label{sympPT property B} The action of $T$ on $\mathcal{C}\times \T\times T$ defined by $t\cdot(x,t',t'') = (x,t',tt'')$ is Hamiltonian with respect to $\omega_{-\infty}$ with moment map $\hw^{PT}$ (where $\hw^{PT}(x,t,t') := \hw^{PT}(x,t)$). This follows from Theorem~\ref{good coordinates}.
	\item\label{sympPT property C} The action of $T$ on $\mathcal{C}\times \T\times T$ defined by $t\cdot (x,t',t'') = (x,\iota(t)t',t'')$  is Hamiltonian with respect to $\omega_{-\infty}$ with moment map $\wt^{PT}$ (where $\wt^{PT}(x,t,t') := \wt^{PT}(x,t)$). This follows from Lemma~\ref{lemma;wtpt properties}. 
		\item\label{sympPT property D} The two $T$-actions commute, so $(\mathcal{C}\times \T\times T,\omega_{-\infty},(\hw^{PT},\wt^{PT}))$ is a Hamiltonian $T\times T$-manifold.
\end{enumerate}

Recall the Hamiltonian $K\times T$-manifold $(K \times \mathring{\ttt}_+^*,\omega_{\rm{can}},(\mu_L,\mu_R) )$ from Example~\ref{T^*K cross section example}. 

\begin{theorem}\label{main theorem}
For all $\delta >0$  and any bounded open subset $U\subset\mathcal{C}(\delta)$ there exists $T\times T$-equivariant symplectic embeddings $(U \times \T \times T,\omega_{-\infty}) \hookrightarrow (K \times \mathring{\ttt}_+^*,\omega_{\mathrm{can}})$ such that the following diagrams commute. 
    \begin{equation}\label{main theorem 1.b, diagram}
    \begin{tikzcd}
        U \times \T \times T \ar[d,"\hw^{PT}"] \ar[r,hookrightarrow] & K \times \mathring{\ttt}_+^* \ar[d,"-\mu_R"] \\
        \ttt_+^*  & \ttt_+^*\ar[l,"="]
    \end{tikzcd}\quad  \quad 
    \begin{tikzcd}
        U \times \T \times T \ar[d,"\wt^{PT}"] \ar[r,hookrightarrow] & K \times \mathring{\ttt}_+^* \ar[d,"\mu_L"] \\
        \ttt^*  & \kk^*\ar[l,"\pr_{\ttt^*}"]
    \end{tikzcd}
\end{equation}
\end{theorem}

\subsection{Proof of Theorem \ref{main theorem}}  \label{section; proof of main theorem}

Throughout the proof of Theorem \ref{main theorem} we use big-O notation analogous to Notation~\ref{notation;bigo}, now used in reference to the coordinates $x_{-i},x_j,\Upsilon_k,\upsilon_l$ on $\mathcal{C}\times \T \times T$.

\noindent\emph{Proof of Theorem \ref{main theorem}:} Fix $\delta >0$  and let  $U$ be a bounded open subset of $\mathcal{C}(\delta)$ as in the statement of Theorem \ref{main theorem}. 

\noindent \textbf{Step 1 (delinearize $K\times\mathring{\ttt}_+^*$):} Let $(K\times \mathring{\ttt}_+^*,\Omega^s,\Psi^s)$ be the delinearization of $(K\times \mathring{\ttt}_+^*,\omega_{\mathrm{can}},(\mu_L,\mu_R))$ as described in Example~\ref{example; delinearization of T^*K cross section}. Recall that~\eqref{delinearization moser vf} defines an isotopy $\phi_s \colon K\times \mathring{\ttt}_+^* \to K\times \mathring{\ttt}_+^*$ such that   $\phi_s^*\Omega^s=\omega_{\mathrm{can}}$.  By Proposition \ref{completeness of moser flow}, $\phi_s$ is defined for all $s$. By Proposition \ref{moser flow equivariance}, 
it suffices to construct $T\times T$-equivariant symplectic embeddings of $(U\times \T\times T,\omega_{-\infty})$ into $(K \times \mathring{\ttt}_+^*,\Omega^s)$ such that the following diagrams commute.
\begin{equation}\label{WTS diagrams}
    \begin{tikzcd}
        U \times \T \times T \ar[d,"\hw^{PT}"] \ar[r,hookrightarrow] & K \times \mathring{\ttt}_+^* \ar[d,"-\mu_R"] \\
        \ttt_+^*  & \ttt_+^* \ar[l,"="]
    \end{tikzcd}\quad  \quad 
    \begin{tikzcd}
        U \times \T \times T \ar[dd,"\wt^{PT}"] \ar[r,hookrightarrow] & K \times \mathring{\ttt}_+^* \ar[d,"\Psi^s"] \\
        & K^* \ar[d,"\wt"]\\
        \ttt^* & A\ar[l,"E_s\n"]
    \end{tikzcd}
\end{equation}

\noindent \textbf{Step 2 (trivialize  $\Psi^s$):} 
Consider the $T$-invariant orthogonal complement $\ttt^\perp \subset \kk$ of $\ttt$ and define a map $\ttt^\perp \times \mathring{\ttt}_+^* \to K^*$ by sending $(Z,\xi) \mapsto E_s(\Ad_{e^Z}^*\xi)$. For a sufficiently small neighbourhood $V$ of $0\in\ttt^\perp$, the restriction of this map to $V \times \mathring{\ttt}_+^*$  is a tubular neighbourhood of the submanifold $E_s(\mathring{\ttt}_+^*)$. Note that $E_s(\mathring{\ttt}_+^*)$ does not depend on $s$. This tubular neighbourhood embedding is $T$-equivariant with respect to the coadjoint and dressing action: for all $t\in T$ and $(Z,\xi) \in \ttt^\perp\times \mathring{\ttt}_+^*$,
\[
	t\cdot (Z,\xi)  = (\Ad_t Z,\xi) \mapsto E_s(\Ad_{\exp(\Ad_t Z)}^*\xi) = E_s(\Ad_t^*\Ad_{e^Z}^*\xi) = tE_s(\Ad_{e^Z}^*\xi) t\n.
\]
The moment map $\Psi^s$ trivializes over the tubular neighbourhood $V\times \mathring{\ttt}_+^*$ as follows:
\begin{equation}
\label{equation;trivializationnbd}
    \begin{tikzcd}
    	& (Z,\xi,t)\ar[r] &  (e^Zt,\xi)\\
      (Z,\xi,t)\ar[d] & V \times \mathring{\ttt}_+^* \times T \ar[r,hookrightarrow]\ar[d,"\pr"] & K \times \mathring{\ttt}_+^* \ar[d,"\Psi^s"]  \\
      (Z,\xi)& V\times \mathring{\ttt}_+^*\ar[r,hookrightarrow] & K^*\\
      & (Z,\xi) \ar[r] & E_s(\Ad_{e^Z}^*\xi)
    \end{tikzcd}
\end{equation}
This trivialization is $T\times T$-equivariant with respect to the actions 
\begin{equation} \label{equation;actionmodelspace}
	\begin{split}
		T\times T \times V \times \mathring{\ttt}_+^* \times T \to V \times \mathring{\ttt}_+^* \times T; & \quad (t',t'',Z,\xi,t) \mapsto (\Ad_{t'}Z,\xi,t't(t'')\n), \\
		T\times T \times K \times \mathring{\ttt}_+^* \to K \times \mathring{\ttt}_+^*; & \quad (t',t'',k,\xi) \mapsto (t'k(t'')\n,\xi).
	\end{split}
\end{equation}

\noindent \textbf{Step 3 (detropicalize):} Let $L_s = L_s^\theta \colon (G^{w_0,e},\theta)^t_\R \times \T \to K^*$ be the detropicalization map (Definition \ref{definition;detropicalization}). Define
\[
	\mathcal{L}_s := L_s \times \id_T \colon (G^{w_0,e},\theta)^t_\R \times \T \times T \to K^* \times T.
\]

\begin{lemma}\label{intersection lemma}\cite[Lemma	 4.1, Lemma 4.2, Proposition 4.3]{AHLL}
    For all $\delta >0$ and any neighbourhood $V$ of $0\in \ttt^\perp$,  there exists $s_0<0$ such that for all $s\leqslant s_0$, 
    \begin{equation} \label{embedLscontrol}
        L_s(\CC(\delta) \times \T) \subset V \times \mathring{\ttt}_+^*
    \end{equation}
    where $V \times \mathring{\ttt}_+^*$ has been identified with its image in $K^*$ under the embedding from Step 2. What is more, if $p\in \CC(\delta) \times \T$ and $L_s(p) = (Z,\xi)\in V\times \mathring{\ttt}_+^*$ under~\eqref{embedLscontrol}, then $Z$ is $O(e^{s\delta})$.
\end{lemma}
Thus, for $s \ll 0$, we have embeddings such that the following diagram commutes.
\begin{equation}\label{step 3 diagram}
    \begin{tikzcd}
      \CC(\delta) \times \T\times T \ar[d,"\pr"]\ar[r,"\mathcal{L}_s",hookrightarrow] & V \times \mathring{\ttt}_+^* \times T \ar[r,hookrightarrow]\ar[d,"\pr"] & K \times \mathring{\ttt}_+^* \ar[d,"\Psi^s"]  \\
      \CC(\delta) \times \T \ar[r,"L_s",hookrightarrow]& V\times \mathring{\ttt}_+^*\ar[r,hookrightarrow] & K^*
    \end{tikzcd}
\end{equation}  
Since $L_s$ is $T$-equivariant, the map $\mathcal{L}_s$ is $T\times T$-equivariant.

\noindent \textbf{Step 4 (Poisson bracket estimates):} Let $\mathfrak{L}_\alpha\colon K^*\to \kk$ denote the Legendre transform of $\alpha \in \Omega^1(K^*)$ (see Section \ref{section;legendre transform}). Given $\beta \in \Omega^1_\C(K^*)$, denote  $\beta_R,\beta_I\in \Omega^1(K^*)$ such that $\beta = \beta_R + \sqrt{-1}\beta_I$. Define $\mathfrak{L}_\beta^\C\colon K^*\to \g$ by complexifying the equation defining $\mathfrak{L}_\beta$. Then,
$\mathfrak{L}_\beta^\C = \mathfrak{L}_{\beta_R} + \sqrt{-1}\mathfrak{L}_{\beta_I}$.

\begin{lemma}\label{lemma;Thimmlemma} 
Let $z$ be a Laurent monomial in the $\theta$-coordinates $z_i$ with $\gamma = |z|_1$. Let
\[
\beta=\frac{dz}{z},\qquad \beta'= \frac{d\Delta_{\gamma,\gamma}}{\Delta_{\gamma,\gamma}}.
\]
Then, as functions $\mathcal{C}(\delta)\times \T \to \kk$,
\[
    \mathfrak{L}_{\beta_R}\circ L_s = \mathfrak{L}_{\beta'}\circ L_s + O(e^{s\delta}) \quad \text{and} \quad \mathfrak{L}_{\beta_I}\circ L_s =  O(e^{s\delta}).
\]
\end{lemma}

\begin{proof} Given $X \in \a \nnn_-$, let $X^R$ denote the right invariant vector field on $K^*$ with $X^R(e) = X$. By definition of the Legendre transform, 
\begin{align*}
     \llangle X, \mathfrak{L}_{\beta_R}-\mathfrak{L}_{\beta'} \rrangle_s + \sqrt{-1}\llangle X, \mathfrak{L}_{\beta_I} \rrangle_s & = \llangle \theta^R(X^R), \mathfrak{L}_{\beta_R}-\mathfrak{L}_{\beta'} \rrangle_s + \sqrt{-1}\llangle \theta^R(X^R), \mathfrak{L}_{\beta_I} \rrangle_s   \\
     & = \beta(X^R) - \beta'(X^R) \\ 
     & = \frac{X\cdot \Delta_{\gamma,\gamma}}{\Delta_{\gamma,\gamma}} - \frac{X\cdot z}{z}.
\end{align*}
By~\eqref{homogeneitygenminor}, if $X \in \a$ then this function is identically $0$. If $X\in \nnn_-$, then applying Corollary~\ref{corollary;domination} and Lemma~\ref{lemma;scaling domination} to $z$ and $\Delta_{\gamma,\gamma}$,
\[
    \left(\frac{X\cdot \Delta_{\gamma,\gamma}}{\Delta_{\gamma,\gamma}} - \frac{X\cdot z}{z}\right)(L_s(p)) = 
      O(e^{s\delta}).
\]
The same argument proves a similar result for the pair of forms $\overline{\beta} = \beta_R - \sqrt{-1}\beta_I$ and $\overline{\beta}' = \beta'$. Combining the resulting two equations completes the proof.
\end{proof}

\begin{lemma} \label{lemma;ThimmbracketsI}
Let $\left\{- , -\right\}_{s}$ denote the Poisson bracket of $\mathcal{L}_s^*\Omega^s$. Then, for all $i,k \in [1,\tilde r]$ and $j \in [1,m]$:
\begin{enumerate}[label=(\roman*)]
    \item $\{ x_{-i}, \Upsilon_k\}_{s}  = \delta_{i,k} + O(e^{s\delta})$,
    \item $\{ x_j, \Upsilon_k\}_{s}  \in \Z +  O(e^{s\delta})$, 
	\item $\{\upsilon_j,\Upsilon_k\}_{s}  =  O(e^{s\delta})$,
\end{enumerate}
 as real-valued functions on $\mathcal{C}(\delta) \times \T \times T$.
\end{lemma}

\begin{proof}  If $f$ is a real valued function defined on $\mathcal{C} \times \T \times T$ such that $f(x,t,t') = f(x,t)$, then by \eqref{step 3 diagram} and Proposition \ref{proposition;legendre transform},
\begin{equation}\label{eq; lem 6.3 proof eq 1}
    \begin{split}
        \{f,\Upsilon_k\}_s(p) & = \{f\circ \mathcal{L}_s\n,\Upsilon_k\circ \mathcal{L}_s\n\}_{\Omega^s}(\mathcal{L}_s(p)) \\
    & = \{f\circ L_s\n \circ \Psi^s,\Upsilon_k\circ \mathcal{L}_s\n\}_{\Omega^s}(\mathcal{L}_s(p)). \\
    & = d(\Upsilon_k\circ \mathcal{L}_s\n)\left(X_{f \circ L_s\n \circ \Psi^s}\right)_{\mathcal{L}_s(p)}\\
    & = d(\Upsilon_k\circ \mathcal{L}_s\n)\left(\underline{\mathfrak{L}_{f\circ L_s\n}(\Psi^s(\mathcal{L}_s(p)))})_{\mathcal{L}_s(p)}\right)\\
    \end{split}
\end{equation}

\noindent \underline{Equation (i):} First, consider the case where $f = x_{-i}$, $i \in [1,\tilde r]$. By Lemma~\ref{lemma;hwpt casimir approximation}, for all $X \in \ttt$,
\begin{equation}
    x_{-i}\circ L_s\n = \langle \hw^{PT}\circ L_s\n,X_i \rangle = \langle  \mathcal{S} \circ E_s\n ,X_i\rangle + O(e^{s\delta}).
\end{equation}
Since $\mathcal{S} \circ E_s\n \circ \Psi^s$ is the moment map for the Thimm torus action on $K\times \mathring{\ttt}_+^*$, it follows again by Lemma~\ref{lemma;hwpt casimir approximation} that
\begin{equation}
    X_{x_{-i}\circ L_s\n \circ \Psi^s} = \underline{X_i} + O(e^{s\delta}) ,
\end{equation}
where $\underline{X_i}$ is the fundamental vector field of $X_i$ for the Thimm torus action. Setting $f = x_{-i}$ in~\eqref{eq; lem 6.3 proof eq 1},
\begin{equation}
    \begin{split}
        \{f,\Upsilon_k\}_s(p) & = d(\Upsilon_k\circ \mathcal{L}_s\n)\left(X_{f \circ L_s\n \circ \Psi^s}\right)_{\mathcal{L}_s(p)}\\
        & = d(\Upsilon_k\circ \mathcal{L}_s\n)(\underline{X_i}_{\mathcal{L}_s(p)}) + O(e^{s\delta})\\
        & = \delta_{i,k} + O(e^{s\delta}).
    \end{split}
\end{equation}

\noindent \underline{Equation (ii):} Let $z = z_j$, $j \in [1,m]$, $\gamma = |z|_1$, $\beta = dz/z$, and $\beta' = d\Delta_{\gamma,\gamma}/\Delta_{\gamma,\gamma}$. Then,
\begin{equation}\label{eq; lem 6.3 proof eq 2}
    \beta_R = \frac{s}{2}d(x_j\circ L_s\n), \quad \beta_I = d(\upsilon_j\circ L_s\n), \quad \beta_R' = \frac{s}{2}d(g\circ L_s\n), \quad \beta_I' = 0,
\end{equation}
where $g = \langle \gamma, \wt^t\rangle$.  Note that 
\begin{equation}\label{eq; lem 6.3 proof eq 3}
\begin{split}
    g\circ L_s\n \circ \Psi^s & = \langle \gamma,\wt^t\circ L_s\n\circ \Psi^s\rangle \\
    & = \langle -\sqrt{-1}\psi_\hhh \circ \wt^t \circ L_s\n\circ \Psi^s, -\sqrt{-1}\psi_\hhh \gamma\rangle \\
    & = \langle E_s\n \circ \wt\circ \Psi^s, -\sqrt{-1}\psi_\hhh \gamma\rangle.
\end{split}
\end{equation}
Since $E_s\n \circ \wt \circ \Psi^s$ is the moment map for the action of $T$ on $K\times \mathring{\ttt}_+^*$ as the maximal torus, 
\begin{equation*}
    \frac{2}{s}\underline{\mathfrak{L}_{\beta_R'}(\Psi^s(m))}_m = \underline{X}_m,
\end{equation*}
where  $\underline{X}_m$ is the fundamental vector field of $X = -\sqrt{-1}\psi_\hhh (\gamma)\in \ttt$ with respect to the same action.

 Setting $f = x_j$ in~\eqref{eq; lem 6.3 proof eq 1}, and combining with Lemma~\ref{lemma;Thimmlemma},
\begin{equation*}
    \begin{split}
        \{x_j,\Upsilon_k\}_s(p) & = \frac{2}{s} d(\Upsilon_k\circ \mathcal{L}_s\n)\left(\underline{\mathfrak{L}_{\beta_R}(\Psi^s(\mathcal{L}_s(p)))}_{\mathcal{L}_s(p)}\right)\\
        & = \frac{2}{s} d(\Upsilon_k\circ \mathcal{L}_s\n)\left(\underline{\mathfrak{L}_{\beta_R'}(\Psi^s(\mathcal{L}_s(p)))}_{\mathcal{L}_s(p)}\right) + O(e^{s\delta})\\
        & =  d(\Upsilon_k\circ \mathcal{L}_s\n)\left(\underline{X}_{\mathcal{L}_s(p)}\right) + O(e^{s\delta})\\
        & \in \Z + O(e^{s\delta}).\\
    \end{split}
\end{equation*}
In the last step, we have that $d(\Upsilon_k\circ \mathcal{L}_s\n)\left(\underline{X}_{\mathcal{L}_s(p)}\right) \in \Z$ since $\underline{X}_{\mathcal{L}_s(p)}$ is the fundamental vector field of the integral element  $-\sqrt{-1}\psi_\hhh (\gamma)\in \ttt$ acting as an element of the maximal torus. The action of the maximal torus on the Upsilon coordinates is as in~\eqref{equation;actionmodelspace}. (In fact, one can easily say exactly which integer this formula evaluates to, but we do not need a precise value).

\noindent \underline{Equation (iii):} The same argument as in the previous case, now applied to $\beta_I$, shows that $\{\upsilon_j,\Upsilon_k\}_{s}  =  O(e^{s\delta})$. This completes the proof.
\end{proof}

The following completes the description of the Poisson brackets of $\mathcal{L}_s^*\Omega^s$.

\begin{lemma} \label{lemma;ThimmbracketsII}
Let $\left\{- , -\right\}_{s}$ denote the Poisson bracket of $\mathcal{L}_s^*\Omega^s$. Then, for all $i,j \in [1,\tilde r]$,
\begin{align*}
	\{ \Upsilon_i, \Upsilon_j\}_{s}  = O(e^{s\delta}),
\end{align*}
 as real-valued functions on $\mathcal{C}(\delta) \times \T \times T$.
\end{lemma}

\begin{proof}
The brackets $\{ \Upsilon_i, \Upsilon_j\}_{s}$ can be approximated using the formula for $(\Omega^s)\n$ provided in~\eqref{equation; delinearization of T^*K cross section}. First, note that $P_k(d\Upsilon_i) =0$ for all $i,k \in [1,\tilde r]$. Second, recall from Remark~\ref{remark; quasitriangular r matrix} that $\pi_K =  \Lambda_0^L - \Lambda_0^R$ where $\Lambda_0 = \sum_{\alpha \in R_+} E_\alpha \wedge E_{-\alpha} $. Third, note that the denominator of the first term of~\eqref{equation; delinearization of T^*K cross section} does not approach zero as $s\to -\infty$.
For a fixed $p \in \mathcal{C}(\delta) \times \T$ and $t\in T$, by Lemma~\ref{intersection lemma} we may take $s$ sufficiently negative that we can express $\mathcal{L}_s(p,t) = (e^Zt,\xi) \in K\times \mathring{\ttt}^*_+$ in terms of the trivialization \eqref{equation;trivializationnbd}. Here $Z\in \ttt^\perp$ and $\xi\in \mathring{\ttt}_+^*$ are elements which depend on $s$.
Noting that $E_\alpha,~E_{-\alpha}\in \ttt^\perp +\sqrt{-1} \ttt^\perp$, it then suffices to show that, for any $Y \in \ttt^\perp$, 
\begin{equation} \label{suffUpsilonbrackets}
(Y^L\cdot \Upsilon_i)(e^Zt) = O(e^{s\delta}),\qquad (Y^R\cdot \Upsilon_i)(e^Zt) = O(e^{s\delta}),
\end{equation}
where $Y^L$ and $Y^R$ denote the left- and right-invariant vector fields on $K$ whose value at $e$ is $Y$.

There exists unique $Z'\in \ttt^\perp$ and $X\in \ttt$ so that for $q\in \R$, $\log (e^Z e^{q\Ad_t Y}) =\log( e^{Z+qZ'} e^{qX}) \mod q^2$. 
Due to Lemma~\ref{intersection lemma}, the element $Z$ is $O(e^{s\delta})$. Therefore, by the Baker-Campbell-Hausdorff formula,
\[
Z + q(\Ad_t Y + O(e^{s\delta}))  = Z + q(Z' +X + O(e^{s\delta})) \mod q^2.
\]
By taking orthogonal projection $\kk \to \ttt$, we see that $X = O(e^{s\delta})$. 

Now, computing from the definition,
\begin{align*}
(Y^L\cdot \Upsilon_i)(e^Zt) & = \frac{d}{dq}\Big|_{q=0} \Upsilon_i(e^Z t e^{qY}) \\
& = \frac{d}{dq}\Big|_{q=0} \Upsilon_i(e^Z e^{q\Ad_t Y} t) \\
& = \frac{d}{dq}\Big|_{q=0} \Upsilon_i(e^Z e^{q\Ad_t Y} ) \\
& = \frac{d}{dq}\Big|_{q=0} \Upsilon_i(e^{Z+qZ'} e^{qX} ) \\
& = d\Upsilon_i(X)  = O(e^{s\delta}),
\end{align*}
as desired. The proof for the action of $Y^R$ is essentially the same. 
\end{proof}

\begin{corollary} \label{corollary;bestcoords}
The coordinates~\eqref{goodcoords1} and~\eqref{goodcoords2} can be chosen such that, on $\mathcal{C}(\delta) \times \T \times T$, 
\begin{align*}
\mathcal{L}_s^*\Omega^s 
& = \omega_{-\infty} + O(e^{s\delta}).
	\end{align*}
\end{corollary}

\begin{proof}
By Theorem~\ref{good coordinates} and Theorem~\ref{theorem;bracketconvergence}, the coordinates~\eqref{goodcoords1} can be chosen so that, under $\mathcal{L}_s^* \Omega^s$,
\[
\{x_j,\upsilon_k\}_s = \delta_{j,k} +O(e^{s\delta}), \quad \{x_j, x_k\}_s = \{x_j, x_{-k} \}_s = \{x_{-j}, x_{-k} \}_s = \{\upsilon_j,\upsilon_k\}_s = O(e^{s\delta})
\]
for $j,k>0$.
By Lemma~\ref{lemma;ThimmbracketsI}, one has for this choice of coordinates,
\[
\{ x_{-i}, \Upsilon_k\}_{s}  = \delta_{i,k} + O(e^{s\delta}),\quad \{ x_j, \Upsilon_k\}_{s}  = N_{j,k} +  O(e^{s\delta}), \quad \{\upsilon_j,\Upsilon_k\}_{s}  =  O(e^{s\delta}),
\]
where $N_{j,k} \in \Z$.
Finally, by Lemma~\ref{lemma;ThimmbracketsII}, 
\[
\{ \Upsilon_i, \Upsilon_j\}_{s}  = O(e^{s\delta}).
\]
Consider the unimodular change of coordinates given by replacing
\[
x_j := x_j-\sum_{k=1}^{\tilde{r}}N_{j,k} x_{-k}.
\]
It is easy to check that the new coordinates $x_{-i},x_j,\Upsilon_i,\upsilon_j$ satisfy all the above equalities except now
\[
\{ x_j, \Upsilon_k\}_{s}  =  O(e^{s\delta}).
\]
This completes the proof.
\end{proof}

In what follows we will assume without loss of generality we have chosen the coordinates~\eqref{goodcoords1} and~\eqref{goodcoords2} so that the conclusion of Corollary~\ref{corollary;bestcoords} holds.

\noindent \textbf{Step 5 (correction maps):}  Unfortunately, the map $\mathcal{L}_s$ does not make the diagrams \eqref{WTS diagrams} commute. It is therefore necessary to precompose  $\mathcal{L}_s$ with a ``correction map''.

\begin{lemma}\label{correction maps}
    For all $\delta >0$ and any bounded open subset $U \subset \CC(\delta )$, there exists $T\times T$-equivariant embeddings $G_s \colon U\times \T \times T \to \CC(\delta) \times \T \times T$ such that the following diagrams commute.
\begin{equation} \label{equation;verybigsquare}
    \begin{tikzcd}
      U \times \T\times T \ar[d,"\pr"]\ar[r,"G_s",hookrightarrow] & \CC(\delta) \times \T\times T \ar[r,"\mathcal{L}_s",hookrightarrow] & K\times \mathring{\ttt}_+^*\ar[d,"\Psi^s"] \\
      U \times \T \ar[d,"\hw^{PT}"] & & K^* \ar[d,"E_s\n"]\\
      \mathring{\ttt}_+^* & & \kk^* \ar[ll,"\mathcal{S}"] \\
    \end{tikzcd}
\end{equation}  
\begin{equation} \label{equation;otherverybigsquare}
    \begin{tikzcd}
      U \times \T\times T \ar[d,"\pr"]\ar[r,"G_s",hookrightarrow] & \CC(\delta) \times \T\times T \ar[r,"\mathcal{L}_s",hookrightarrow] & K\times \mathring{\ttt}_+^*\ar[d,"\Psi^s"] \\
      U \times \T \ar[d,"\wt^{PT}"] & & K^* \ar[d,"\wt"]\\
      \mathring{\ttt}_+^* & & \kk^* \ar[ll,"E_s\n"] \\
    \end{tikzcd}
\end{equation}  
Moreover, with respect to the coordinates $x_{-i},x_k,\upsilon_k,\Upsilon_i$, the Jacobian $(G_s)_*$ has  block-form
\begin{equation}
    (G_s)_* = \left(\begin{array}{cccc}
         I_{\tilde{r} \times \tilde{r}} + O(e^{s\delta})& O(e^{s\delta})&  O(e^{s\delta})& 0\\
          0 & I_{m\times m} &0  &0  \\
         
         0 & 0 & I_{m\times m}&0 \\
         0 & 0 & 0 & I_{\tilde{r}\times \tilde{r}} \\
    \end{array}\right),
\end{equation}
where $O(e^{s\delta})$ denotes a matrix whose entries are $O(e^{s\delta})$.
\end{lemma}

\begin{proof}
Since the diagram  \eqref{step 3 diagram} commutes, it suffices to find a $T$-equivariant map $g_s\colon U\times \T \to \mathcal{C}(\delta)\times \T$ such that
	that the following diagrams commute,
\begin{equation}\label{correction map big diagram}
    \begin{tikzcd}
      U \times \T \ar[d,"\hw^{PT}"]\ar[r,"g_s"] & \mathcal{C}(\delta) \times \T \ar[r,"L_s"] & K^* \ar[d,"E_s\n"]\\
      \mathring{\ttt}_+^* & & \kk^*\ar[ll,"\mathcal{S}"] \\
    \end{tikzcd}
\end{equation}
\begin{equation}\label{correction map big diagram II}
    \begin{tikzcd}
      U \times \T \ar[d,"\wt^{PT}"]\ar[r,"g_s"] & \mathcal{C}(\delta) \times \T \ar[r,"L_s"] & K^* \ar[d,"\wt"]\\
      \mathring{\ttt}_+^* & & \kk^*\ar[ll,"E_s\n"] \\
    \end{tikzcd}
\end{equation}
and, with respect to the coordinates $x_{-i},x_k,\upsilon_k$, the Jacobian of $g_s$ has block-form
\begin{equation}\label{diagram; correction map proof second diagram}
	(g_s)_* = \left(
	\begin{array}{ccc}
         I_{\tilde{r} \times \tilde{r}} + O(e^{s\delta})& O(e^{s\delta})& O(e^{s\delta})\\
         0 &I_{m\times m} & 0    \\
         0 & 0 & I_{m\times m} \\
	\end{array}\right).
\end{equation}
The map $G_s = g_s \times \id_T\colon U \times \T \times T \to \mathcal{C}(\delta) \times \T \times T$ will then have the desired properties. 

For every bounded open set $U \subset \mathcal{C}(\delta)$, it is possible to find a convex bounded set $C \subset \mathring{\ttt}_+^*$ such that
\[
    U \times \T \subset (\hw^{PT})\n(C)\cap (\mathcal{C}(\delta) \times \T).
\]
Thus it is sufficient to construct $g_s$ on subsets of the form $ (\hw^{PT})\n(C)\cap (\mathcal{C}(\delta) \times \T)$ for any convex bounded set $C \subset \mathring{\ttt}_+^*$ .

Such a map $g_s$ was constructed in \cite[Lemma 3.1, Lemma 3.2, Proposition 3.3]{AHLL} for the case where $C = \{p\}$ for an arbitrary  element $p\in \mathring{\ttt}_+^*$. The same implicit function theorem argument used in \cite[ Lemma 3.2]{AHLL} can be easily extended to construct  $g_s$ with the properties described above for any convex bounded open set $C\subset \mathring{\ttt}_+^*$.

Although it was not mentioned in \cite{AHLL}, it is easy to see  that~\eqref{correction map big diagram II} commutes. By an argument similar to the proof of Lemma~\ref{lemma;wtPT commuting diagram}, it suffices to show that $\wt^t\circ g_s = \wt^t$. This is easy to see looking at the definition of $g_s$.
\end{proof}

To see that the left diagram in \eqref{WTS diagrams} commutes, observe that by \eqref{equation;verybigsquare}, \eqref{linearization equation}, and since $\mathcal{S}\circ \mu_L= -\mu_R$, 
\[
    \hw^{PT} \circ \pr  = \mathcal{S} \circ E_s\n \circ \Psi^s \circ \mathcal{L}_s \circ G_s = \mathcal{S} \circ\mu_L \circ \mathcal{L}_s \circ G_s = -\mu_R \circ \mathcal{L}_s \circ G_s.
\]

\noindent \textbf{Step 6 (set up Moser's trick):} Fix $\delta>0$ and let $U$ be a bounded open subset of $\mathcal{C}(\delta)$.
Let $\omega_s  = G_s^*\mathcal{L}_s^*\Omega^s \in\Omega^2(
U \times \T \times T)$. The goal is to construct a Moser flow that deforms $\omega_s$ to $\omega_{-\infty}$. For $s<0$, define a closed 2-form $\omega^\tau_s$ on $U \times \T \times T$ by the equation
\[
	\omega^\tau_s = (1-\tau)\omega_{-\infty} + \tau\omega_s,\quad \tau \in [0,1].
\] 

\begin{lemma}\label{mosers trick lemma 1}
    For all $\delta>0$, $\tau\in [0,1]$, and $s\ll 0$, the form $\omega^\tau_s$ is non-degenerate on $U \times \T \times T$.
\end{lemma}

\begin{proof}
	By Corollary~\ref{corollary;bestcoords}, $\mathcal{L}_s^*\Omega^s = \omega_{-\infty} + O(e^{s\delta})$.
	By Lemma \ref{correction maps}, 
	\[
		\omega_s = G_s^*\mathcal{L}_s^*\Omega^s = G_s^*(\omega_{-\infty} + O(e^{s\delta})) = \omega_{-\infty} + O(e^{s\delta}).
	\]
	Thus
	\[
	    \omega^\tau_s = (1-\tau)\omega_{-\infty} +\tau(\omega_{-\infty} + O(e^{s\delta})) = \omega_{-\infty} + O(e^{s\delta})
	\]
	is non-degenerate for $s\ll 0$.
\end{proof}

Define another closed 2-form $\alpha_s = \omega_s - \omega_{-\infty}$ on $U \times \T \times T$. Then $\alpha_s$ is $O(e^{s\delta})$.
 
\begin{lemma}\label{mosers trick lemma 2a}
Let $U$ be an open subset of $\mathcal{C}(\delta)$. Then $\alpha_s$ is exact on $U\times \T \times T $.
\end{lemma}
\begin{proof}
The form $\omega_{-\infty}$ is exact, and so it suffices to show that $\mathcal{L}_s^* \Omega^s$ is exact. Consider the involution 
\begin{equation}
\label{baronKt}
 K \times \mathring{\ttt}^*_+ \to K \times \mathring{\ttt}^*_+\colon \qquad (k,\xi)\mapsto (\overline{k},\xi).
\end{equation}
 By naturality of the canonical symplectic form on $T^* K$, the involution of $(K\times \ttt^*,\omega_{\mathrm{can}})$ which maps $(k,\xi)\mapsto (\overline{k},-\xi)$ is symplectic. Therefore~\eqref{baronKt} is anti-symplectic on $(K\times \mathring{\ttt}^*_+,\omega_{\mathrm{can}})$. 
 Consider also the involution 
\begin{equation}
\label{baronPT}
\overline{(\cdot)}\colon  \mathcal{C}(\delta) \times \T \times T \to \mathcal{C}(\delta) \times \T \times T, \qquad \overline{
(x_{-i}, x_j ,\Upsilon_i,\upsilon_j)} =(x_{-i},x_j,-\Upsilon_i,-\upsilon_j).
\end{equation}
 Then $\mathcal{L}_s$ intertwines the involutions~\eqref{baronPT} and~\eqref{baronKt}, and so by \cite[Section~3.4]{AMW}, the involution~\eqref{baronKt} is anti-symplectic on $(K\times \mathring{\ttt^*_+}, \Omega^s)$. Consequently, one has $\overline{(\cdot)}^* \mathcal{L}_s^*\Omega^s = -\mathcal{L}_s^*\Omega^s$.

On the other hand, for any cohomology class $[\omega]\in H^2(\mathcal{C}(\delta)\times \T \times T)\cong H^2(\T\times T)$, one has $\overline{(\cdot)}^* [\omega] = [\omega]$.
Therefore the class $[\mathcal{L}_s^*\Omega^s]\in H^2(\mathcal{C}(\delta)\times \T \times T)$ is equal to $0$, and hence $\mathcal{L}_s^*\Omega^s$ is exact.
\end{proof}

\begin{lemma}
\label{torushtpy}
Let $n$ be a positive integer, let $\T= (S^1)^n$ and fix a subtorus $\T' \subset \T$. Assume $\sigma_s \in \Omega^k(\T)$ be a family of exact smooth $\T'$-invariant forms, parametrized by $s<0$. Assume there is some $\delta>0$ such that $\sigma_s = O (e^{s\delta})$.
Then, there exists a family $\gamma_s\in \Omega^{k-1} (\T)$ of smooth $\T'$-invariant forms, which satisfy $d\gamma_s= \sigma_s$ and $\gamma_s = O (e^{s\delta})$.
\end{lemma}

\begin{proof}
Let us first introduce some notation. For any positive integer $l$, we denote multi-indices by capital letters, eg $J=(j_1,\dots, j_l)$. Write 
\[
e^{2\pi \sqrt{-1} \varphi} = (e^{2\pi\sqrt{-1}\varphi_1}, \dots, e^{2\pi\sqrt{-1}\varphi_n}) \in \T= (S^1)^n
\]
for points of $\T$, and consider the forms $d\varphi_J = d\varphi_{j_1}\wedge \cdots \wedge d\varphi_{j_l}$. For any form $\omega\in \Omega^l(\T)$, write $\omega = \sum_J \omega_J d\varphi_J$. For $m= (m_1,\dots, m_n) \in \Z^n$, let $m\cdot \varphi = m_1 \varphi_1 + \cdots +m_n \varphi_n$, and let
\[
\hat{\omega}_J(m) = \int_{\T} \omega_J  e^{-2\pi\sqrt{-1}m \cdot \varphi} 
\]
be the $m^{th}$ Fourier coefficient of $\omega_J$. Define
\[
\omega[m] = e^{2\pi\sqrt{-1}m \cdot \varphi} \sum_J \hat{\omega}_J(m) d\varphi_J.
\]
Then $\omega = \sum_{m\in \Z^n} \omega[m]$, where the sum on the right converges uniformly because $\omega$ is smooth. 

We consider the form $\sigma_s\in \Omega^k(\T)$, writing $\sigma=\sigma_s$ for simplicity. For each $0\ne m = (m_1,\dots, m_n) \in \Z^n$, choose $j(m)\in \{1,\dots,n\}$ so that $m_{j(m)}\ne 0$. Define the forms
\[
\gamma[m] : = \frac{1}{2\pi\sqrt{-1}m_{j(m)}} \iota_{\frac{\partial}{\partial\varphi_{j(m)}}} \sigma[m],
\]
and let $\hat{\gamma}_J(m)$ be the coefficient of $d\varphi_{J}$ in $e^{-2\pi\sqrt{-1}m \cdot\varphi} \gamma[m]$, so
\[
e^{-2\pi\sqrt{-1}m \cdot\varphi} \gamma[m] = \sum_J \hat{\gamma}_J(m) d\varphi_J.
\]

Because the Fourier series of $\sigma$ converges uniformly, we have
\begin{align}
\nonumber \sum_{m\ne 0} |\hat{\gamma}_J(m)| &  \le  \sum_{m \ne 0} \frac{1}{2\pi m_{j(m)}} \sum_K |\hat{\sigma}_K(m)| \\
& =\label{firstapprox} \sum_K \sum_{m\ne 0} \frac{|\hat{\sigma}_K(m)|}{2\pi m_{j(m)}} \\
& \nonumber < \sum_K \sum_{m\ne 0} |\hat{\sigma}_K(m)| < \infty
\end{align}
and hence $\sum_{m\ne 0} \gamma[m]$ converges uniformly to a smooth $k-1$ form. Let us denote this form by $\gamma = \sum_{m\ne 0} \gamma[m]$. 

Since the exterior derivative $d$ preserves the Fourier mode $m$, we have 
\begin{equation}
\label{modeclosed}
d(\sigma[m]) = (d\sigma) [m] = 0.
\end{equation}
Applying Cartan's magic formula and \eqref{modeclosed}, we have $d\gamma[m] = \sigma[m]$ for all $0\ne m \in \Z^n$. So $\sigma- d\gamma = \sigma[0]$. Because $\sigma$ is exact, it follows that $\sigma[0] = \sum_K \hat{\sigma}_K (0) d\varphi_K$ is exact. But this can only be true if $\sigma[0]=0$. So $\sigma = d\gamma$. 

Let us control the size of $\gamma$. We compute as in \eqref{firstapprox} and use Parseval's theorem:
\begin{align*}
|\gamma_J|^2 & \le \sum_K \sum_m |\hat{\sigma}_K(m)|^2  = \sum_K \int_\T |\sigma_K|^2.
\end{align*}
Since $\sigma_s$ is $O(e^{s\delta})$, we conclude $\sum_K \int_\T |\sigma_K|^2$ is $O(e^{2s\delta})$. Hence $\gamma_s =\gamma$ is $O(e^{s\delta})$.
Finally, since $d\gamma_s=\sigma_s$ and $\sigma_s$ is $\T'$-invariant, by averaging $\gamma_s$ with respect to translations by $\T'$ we can assume that $\gamma_s$ is $\T'$-invariant as well. Because $\T$ and $\T'$ are compact this does not affect the conclusion that $\gamma_s$ is $O(e^{s\delta})$.
\end{proof}

\begin{lemma}\label{mosers trick lemma 2b}
    Let $U$ be a convex bounded open subset of $\mathcal{C}(\delta)$. Then there exists a $T\times T$-invariant form $\beta_s \in \Omega^1(U\times \T\times T)$ such that $d\beta_s = \alpha_s$ and $\beta_s$ is $O(e^{s\delta})$.
\end{lemma}

\begin{proof} 
Fix a point $p\in U$. The set $U$ is convex so we may define a straight line retract 
\begin{align*}
H \colon   [0,1]\times U \times \T\times T & \to U \times \T \times T \\
(q,x,\upsilon, \Upsilon) & \mapsto H_q(x,\upsilon,\Upsilon) = (p+q(x-p),\upsilon,\Upsilon)
\end{align*}
from $U\times \T \times T$ to $\{p\}\times \T \times T$.
Let $h\colon \Omega^\bullet(U\times \T\times T) \to \Omega^{\bullet-1}(U\times \T\times T)$ be the homotopy operator associated with $H$, so
\[
h\alpha_s= \int_{0}^{1} \left(\iota_{\frac{\partial}{\partial q}} H^* \alpha_s \right) dq \in \Omega^1(U \times \T \times T).
\]
Note that $H_q \colon U\times \T \times T \to U\times \T \times T$ is $T\times T$-equivariant for all $q\in [0,1]$, and that the form and $\alpha_s$ is $T\times T$-invariant. What is more, $\alpha_s$ is $O(e^{s\delta})$ and $U$ is bounded. Therefore the form $h\alpha_s$ is $T\times T$-invariant and is $O(e^{s\delta})$. Since $h$ is a homotopy operator and since $\alpha_s$ is closed, one has
\[
\alpha_s = d h\alpha_s +H_0^*\alpha_s.
\]
By Lemma~\ref{mosers trick lemma 2a}, the form $H_0^* \alpha_s$ is exact. It is $T\times T$-invariant, and is $O(e^{s\delta})$. So by Lemma \ref{torushtpy}, there is a $T\times T$-invariant $\gamma_s \in \Omega^1(\{p\}\times \T \times T)$ satisfying $d\gamma_s  = H_0^*\alpha_s$, which also satisfies $\gamma_s = \mathcal{O}(e^{s\delta})$.
Then $\beta_s = h\alpha_s + \gamma_s$ has the desired properties.
\end{proof}

\noindent \textbf{Step 7 (integrate the Moser vector field):} To complete the proof, notice that there exists $0<\delta'<\delta$ and a convex bounded open set $U'\subset \mathcal{C}(\delta')$ so that $U\subset \overline{U} \subset U'$. We replace $U$ with $U'$, and all the previous discussion holds with $\delta$ replaced by $\delta'$. We will then consider $U \times \T \times T$ as a submanifold of $U' \times \T \times T\subset \mathcal{C}(\delta') \times \T \times T$.

By Lemma~\ref{mosers trick lemma 1}, the form $\omega^\tau_s\in \Omega^2(U'\times \T\times T)$ is non-degenerate for all $\tau\in [0,1]$, for $s\ll 0$. For any such $s$, the equation 
\begin{equation}\label{moser equation}
    \iota_{X^\tau_s} \omega^\tau_s = -\beta_s
\end{equation}
defines a $\tau$-dependent vector field $X^\tau_s$, $\tau\in [0,1]$, on $U' \times \T \times T$. Denote by $\phi_s^\tau$ the flow of $X^\tau_s$ for all points $p\in U'\times \T \times T$ and all $\tau \in [0,1]$ for which it is defined. 

\begin{lemma}
The vector field $X_s^\tau$ is $O(e^{s\delta'})$ for all $\tau$.
\end{lemma}

\begin{proof}
By~\eqref{moser equation}, 
\[
\iota_{X^\tau_s} \omega^\tau_s = \iota_{X^\tau_s} \omega_{-\infty} +\tau \iota_{X^\tau_s} \alpha_s = -\beta_s = O(e^{s\delta'}).
\]
Since $\alpha_s$ is $O(e^{s\delta'})$ and $\omega_{-\infty}$ is nondegenerate, the vector field $X^\tau_s$ must be $O(e^{s\delta'})$.
\end{proof}

\begin{proposition}\label{moser integration}
There exists $s\ll 0$ and $x_0 \in (G^{w_0,e},\theta)^t_\R$ such that:
\begin{enumerate}
\item The flow 
\[
\phi_s^\tau|_{U \times \T \times T} \colon U\times \T \times T \to U'\times \T \times T
\]
is defined for all $\tau\in [0,1]$;
\item The time 1 flow $\phi_s^1$ satisfies $
(\phi_s^1)^*(\omega_s^1) = (\phi_s^1)^*(\omega_s) = \omega_{-\infty}
$
and is $T\times T$-equivariant;
\item There exists $x_0\in (G^{w_0,e},\theta)^t_\R$ so that the map 
\begin{align*}
\tilde{\phi}_s^1\colon U \times \T \times T & \to (G^{w_0,e},\theta)^t_\R \times \T \times T\\
(x,\upsilon)& \mapsto (x+x_0,\upsilon)
\end{align*}
has $\tilde{\phi}_s^1(U\times \T \times T) \subset U'\times \T \times T$, satisfies $(\tilde{\phi}_s^1)^*(\omega_s^1) = \omega_{-\infty}$, and additionally satisfies
\[
\hw^{PT} \circ \tilde{\phi}_s^1= \hw^{PT}, \text{ and } \wt^{PT}\circ \tilde{\phi}_s^1 = \wt^{PT}.
\]
\end{enumerate}
\end{proposition}

\begin{proof}
Let $R$ denote the minimum distance from $\partial U$ to $\partial U'$, where we consider the standard Euclidean metric on $(G^{w_0,e},\theta)^t_\R\cong \R^{\tilde{r}+m}$. Let $c>0$ denote the norm of the linear map 
\[
(\wt^{PT}\times \hw^{PT})\colon (G^{w_0,e},\theta)^t_\R \cong (G^{w_0,e},\theta)^t_\R\times\{1\} \to \ttt^*\times \ttt^*.
\]
Fix a linear section $\sigma$ of this map, and let $b=||\sigma||>0$.
Fix an invariant metric on $\T\times T$, and equip $(G^{w_0,e},\theta)^t_\R \times \T \times T$ with the product metric.
Since $X_s^\tau$ is $O(e^{s\delta})$, we may choose $s\ll 0$ such that $||X^\tau_s||< \min\{\frac{R}{2bc},\frac{R}{2}\}$ at all points of $U' \times \T \times T$. Then for all $(x,\upsilon) \in U\times \T\times T$ and all $\tau\in [0,1]$, we have that the distance from $\phi^\tau_s(x,\upsilon)$ to $(x,\upsilon)$ is less than $\frac{R}{2}$. Therefore the flow of $X_s^\tau$, restricted to $U\times \T\times T$, does not escape $U'\times \T \times T$, for $\tau \in [0,1]$. This establishes the first claim. 

The second claim is due to the standard Moser argument: By~\eqref{moser equation}, one has 
\[
d\iota_{X^\tau_s} \omega^\tau_s + d \beta_s = L_{X^\tau_s} \omega^\tau_s +
 \frac{\partial \omega_s^\tau}{\partial \tau} =0
\]
and therefore $(\phi_s^\tau)^*\omega_s^\tau = \omega_s^0 = \omega_{-\infty}$ wherever the flow is defined. Finally, since $\omega_s^\tau$ and $\beta_s$ are both $T\times T$-invariant, the flow of $X_s^\tau$ must be $T\times T$-equivariant.

For the third claim, notice that the map $(\wt^{PT},\hw^{PT}) \colon \mathcal{C}(\delta)\times \T \times T\to \ttt^*\times \ttt^*$ is a moment map for the $T\times T$ action, with respect to both  $\omega_{-\infty}$ and $\omega_s$  (cf. properties~\ref{sympPT property B} and~\ref{sympPT property C} and Lemma~\ref{correction maps}). By uniqueness of moment maps there is some $\hat{x}_0\in \ttt^*\times \ttt^*$ such that
\[
(\wt^{PT}\circ \phi_s^1,\hw^{PT}\circ \phi_s^1)+\hat{x}_0 = (\wt^{PT},\hw^{PT}).
\]
From $||X_s^\tau||< \frac{R}{2bc}$ we conclude that $||\hat{x}_0||<\frac{R}{2b}$.
Let $(x_0,1) = \sigma(\hat{x}_0) \in (G^{w_0,e},\theta)^t_\R \times \T \times T$. Then $||x_0||\le b||\hat{x}_0||\le R/2$. Therefore,  for $(x,\upsilon)\in U\times \T \times T$, the distance from $(x,\upsilon)$ to $\tilde{\phi}_s^1(x,\upsilon)$ is less than $R/2+R/2=R$, and hence $\tilde{\phi}_s^1(x,\upsilon)$ is contained in $U'\times \T \times T$. The form $\omega_{-\infty}$ is constant, and so a shift by $x_0$ preserves $\omega_{-\infty}$. Therefore $(\tilde{\phi}_s^1)^*(\omega_s^1) = \omega_{-\infty}$. The final claim is immediate from the construction.
\end{proof}

We then have the following embedding:
\[
\begin{tikzcd}
F\colon U \times \T \times T \ar[r,"\tilde{\phi}_s^1",hookrightarrow] & U' \times \T \times T \ar[r, "G_s",hookrightarrow] & C(\delta')\times \T \times T \ar[r,"\mathcal{L}_s",hookrightarrow] & K \times \ttt^*_+.
\end{tikzcd}
\]
It satisfies $F^* \omega_{can} = \omega_{-\infty}$, and makes the diagrams~\eqref{main theorem 1.b, diagram} commute.
This completes the proof of Theorem~\ref{main theorem}.

\subsection{Big action-angle coordinate charts on multiplicity free spaces}  \label{section; charts on multiplicity free}

Theorem \ref{main theorem} can be used to construct action-angle coordinates on big subsets of a large family of compact, connected, multiplicity free Hamiltonian $K$-manifolds.   Before stating the theorem, we define the domains of our action-angle coordinates.

Let $(M,\omega,\mu)$ be a compact, connected multiplicity free space with principal stratum $\mathring{\ttt}_+^*$, Kirwan polytope $\triangle = \mu(M) \cap \ttt_+^*$, and principal isotropy subgroup $L \leq T$. Let $\mathring{\triangle}$ denote the relative interior of $\triangle$. Recall from Section~\ref{section;multiplicity free spaces} that there is an associated toric $T/L$-manifold $(\mathring{\triangle} \times T/L,\omega_{\mathrm{std}},\pr\colon \mathring{\triangle} \times T/L \to \mathring{\triangle})$, where $\omega_{\mathrm{std}}$ was defined via a linear identification of $\Lie(T/L)^*$ with the affine subspace of $\ttt^*$ spanned by $\mathring{\triangle}$. In what follows, we will view $(\mathring{\triangle} \times T/L,\omega_{\mathrm{std}},\pr\colon \mathring{\triangle} \times T/L \to \mathring{\triangle})$ as a Hamiltonian $T$-manifold. 

Let $(\mathcal{C}\times \T\times T,\omega_{-\infty},(\hw^{PT},\wt^{PT}))$ be one of the Hamiltonian $T\times T$-manifolds defined in Section~\ref{section; statement of main theorem} (depending on the choices described therein). The open submanifolds $\mathcal{C}(\delta)\times \T\times T$, $\delta >0$, are also Hamiltonian $T\times T$-manifolds. We may form the product Hamiltonian $T\times T \times T$-manifold,
\[
	(\mathcal{C}(\delta) \times \T \times T \times \mathring{\triangle} \times T/L,\omega_{-\infty}\oplus \omega_{\mathrm{std}},(\hw^{PT},\wt^{PT},\pr)).
\] 
Our model space is the symplectic reduction of this space  with respect to the diagonal $T$ action generated by the moment map  $\pr-\hw ^{PT}$,
\[
	 M(\delta,\mathring{\triangle},L) := (\mathcal{C}(\delta) \times \T \times T \times \mathring{\triangle} \times T/L,\omega_{-\infty}\oplus \omega_{\mathrm{std}})\sslash_0 T.
\]
Provided that $\mathring{\triangle} \subset \mathring{\ttt}_+^*$ and $\delta>0$ is sufficiently small, $M(\delta,\mathring{\triangle},L)$ is non-empty. Since the action of $T$ is free, $M(\delta,\mathring{\triangle},L)$ is a smooth symplectic manifold. Denote the reduced symplectic form on $M(\delta,\mathring{\triangle},L)$ by $\omega_{\mathrm{red}}$. The moment map $(\hw^{PT},\wt^{PT})$ and the associated $T\times T$ action descend to  $M(\delta,\mathring{\triangle},L)$ so that we have a Hamiltonian $T\times T$-manifold,
\[
    (M(\delta,\mathring{\triangle},L),\omega_{\mathrm{red}},(\hw^{PT},\wt^{PT})).
\]
Note that we use the same notation for the maps induced on the quotient as for the maps $\hw^{PT}$ and $\wt^{PT}$.  By definition of the symplectic reduced space, there is a $T\times T$-equivariant diffeomorphism \[
        \mathcal{C}(\delta,\mathring{\triangle})\times \T \times (T/L) \xrightarrow{\cong} M(\delta,\mathring{\triangle},L), \quad (x,t,t'L) \mapsto [x,t,t',\hw^{PT}(x),1],
\]
where $\mathcal{C}(\delta,\mathring{\triangle})$ denotes the projection to $\mathcal{C}(\delta)$ of the $T\times T$-invariant subspace $(\hw^{PT})\n(\mathring{\triangle}) \subset \mathcal{C}(\delta) \times \T \times T$.

Finally, recall from Proposition~\ref{proposition about dense model space}  that  the dense subset $\mathcal{U}=(\mathcal{S}\circ \mu)\n(\mathring{\ttt}_+^*) \subset M$ is a Hamiltonian $K\times T$-manifold with moment map $(\mu,\mathcal{S}\circ \mu)$. The extra Hamiltonian $T$-action on $\mathcal{U}$ is the Thimm torus action. The embeddings constructed in the following theorem are action-angle coordinates by the description of $M(\delta,\mathring{\triangle},L)$ above.

\begin{theorem}\label{main corollary 1}
    Let $(M,\omega,\mu)$ be a compact\footnote{Those familiar should note that Theorem~\ref{main corollary 1} easily extends to the more general setting of \emph{convex} multiplicity free Hamiltonian $K$-manifolds  (cf. \cite[Definition 2.2]{knop2}).}, connected, multiplicity free Hamiltonian $K$-manifold such that $\mathring{\triangle}\subset \mathring{\ttt}_+^*$ and let $\mathcal{U}=(\mathcal{S}\circ \mu)\n(\mathring{\ttt}_+^*)$. For all $\varepsilon >0$, there exists $\delta >0$ and symplectic embeddings $(M(\delta,\mathring{\triangle},L),\omega_{\mathrm{red}}) \hookrightarrow (\mathcal{U},\omega)$ 
    such that:
    \begin{enumerate}[label=(\roman*)]
    	\item The symplectic volumes satisfy
    		\[
    			\Vol(M(\delta,\mathring{\triangle},L),\omega_{\mathrm{red}}) > \Vol(M,\omega) - \varepsilon. 
			\]

    	\item The embeddings $(M(\delta,\mathring{\triangle},L),\omega_{\mathrm{red}}) \hookrightarrow (U,\omega)$ are embeddings of Hamiltonian $T\times T$-manifolds:
    	\[
    	    (M(\delta,\mathring{\triangle},L),\omega_{\mathrm{red}},(\hw^{PT},\wt^{PT})) \hookrightarrow (\mathcal{U},\omega,(\mathcal{S}\circ \mu,\pr_{\ttt^*}\circ \mu)).
    	\]
    	\end{enumerate}    
\end{theorem}

\begin{proof} \noindent\textbf{Construction of the embeddings and proof of (ii):} The desired embeddings of Hamiltonian $T\times T$-manifolds are constructed by first constructing embeddings of Hamiltonian $T\times T \times T$-manifolds, then reducing by a certain diagonal copy of $T$.

First, we consider the following composition of an embedding and an isomorphism of Hamiltonian $T\times T \times T$-manifolds,
\begin{equation}
    \begin{split}
        &(\mathcal{C}(\delta,\mathring{\triangle}) \times \T \times T \times \mathring{\triangle} \times T/L,\omega_{-\infty}\oplus \omega_{\mathrm{std}},(\hw^{PT},\wt^{PT},\pr))\\
        & \quad \hookrightarrow (K\times \mathring{\ttt}_+^* \times \mathring{\triangle} \times T/L,\omega_{-\infty}\oplus \omega_{\mathrm{std}},(\mathcal{S}\circ \mu_L,\pr_{\ttt^*}\circ \mu_L,\pr))\\
        & \quad \cong (K\times \mathring{\ttt}_+^* \times W,\omega_{-\infty}\oplus \omega,(\mathcal{S}\circ \mu_L,\pr_{\ttt^*}\circ \mu_L,\mu)).\\
    \end{split}
\end{equation}
The  embedding exists by Theorem \ref{main theorem}: since $\mathring{\triangle}$ is the relative interior of a convex polytope and the fibers of $\hw^{PT}$ are bounded, the set  $\mathcal{C}(\delta,\mathring{\triangle}) \subset \mathcal{C}(\delta)$  is bounded.  The isomorphism follows by Lemma~\ref{model for sympelctic slice}.

Next, we apply a diagonal symplectic reduction to this composition, corresponding to the inclusion $T \to T\times T \times T$, $t \mapsto (t\n,1,t)$. On the left side, this action is generated by the moment map $\pr - \hw^{PT}$ and on the right it is generated by $\mu-\mathcal{S}\circ\mu_L$. Since the embeddings, torus actions, and moment maps above are all compatible with this reduction, we obtain an embedding,
\begin{equation}
    \begin{split}
        &(\mathcal{C}(\delta,\mathring{\triangle}) \times \T \times T \times \mathring{\triangle} \times T/L,\omega_{-\infty}\oplus \omega_{\mathrm{std}},(\hw^{PT},\wt^{PT},\pr))\sslash_0 T\\
        & \quad  \hookrightarrow (K\times \mathring{\ttt}_+^* \times W,\omega_{-\infty}\oplus \omega,(\mathcal{S}\circ \mu_L,\pr_{\ttt^*}\circ \mu_L,\mu))\sslash_0 T.\\
    \end{split}
\end{equation}
The reduced space on the right  embeds densely as a Hamiltonian $T\times T$ submanifold of $(\mathcal{U},\omega,(\mathcal{S}\circ \mu,\pr_{\ttt^*}\circ \mu))$ by Proposition~\ref{proposition about dense model space}. The space on the left is our model space. Thus, we have constructed the desired embedding of Hamiltonian $T\times T$-manifolds, 
\[
    (M(\delta,\mathring{\triangle},L),\omega_{\mathrm{red}},(\hw^{PT},\wt^{PT})) \hookrightarrow (\mathcal{U},\omega,(\mathcal{S}\circ \mu,\pr_{\ttt^*}\circ \mu)).
\]

\noindent\textbf{Proof of (i):} Let $\beta_M$ denote the Liouville measure of $(M,\omega)$ and let $\beta_W$ denote the Liouville measure of $(W,\omega)$ as a Hamiltonian $T$-manifold. Recall, e.g.\ from \cite{GP}, that
\begin{equation}
		\Vol(M,\omega) = \int_{\kk^*} \mu_*\beta_M = \int_{\ttt_+^*}\Vol(\mathcal{O}_\xi,\omega_\xi)\mu_*\beta_W.
\end{equation}
Let $d\lambda$ denote the measure on $\hw\n(\xi)\cap \mathcal{C}$ defined by the lattice $\frac{1}{2\pi}\Z^m$. By Lemma \ref{lemma;hwpt properties}, 
\[
\int_{\hw\n(\xi)\cap \mathcal{C}}d\lambda = \Vol(\mathcal{O}_\xi,\omega_\xi).
\]
By definition of $\mathcal{C}(\delta) \subset \mathcal{C}$, there exists a constant $c(\xi)>0$, depending continuously on $\xi\in \mathring{\ttt}_+^*$, such that
\[
\int_{\hw\n(\xi)\cap \mathcal{C}(\delta)}d\lambda \geq \Vol(\mathcal{O}_\xi,\omega_\xi) - c(\xi)\delta.
\]
Since the measure $\mu_*\beta_W$ is compactly supported and $c(\xi)$ depends continuously on $\xi\in \mathring{\ttt}_+^*$, there exists $C>0$ such that 
\[
	C \geq \int_{\ttt_+^*}c(\xi) \mu_*\beta_W .
\]
Finally, let $\beta$ denote the Liouville measure of $(M(\delta,\mathring{\triangle},L),\omega_{\mathrm{red}})$. By Lemma \ref{model for sympelctic slice}, the pushforward of the Liouville measure of $(\mathring{\triangle}\times T/L,\omega_{\mathrm{std}},\pr)$ to $\ttt^*$ by $\pr$ equals $\mu_*\beta_W$ (noting again that we have linearly identified $\Lie(T/L^*)$ with the affine subspace of $\ttt^*$ spanned by $\mathring{\triangle}$). The Duistermaat-Heckman measure of $(M(\delta,\mathring{\triangle},L),\omega_{\mathrm{red}})$ as a Hamiltonian $\T \times T$-manifold is therefore $d\lambda \times \mu_*\beta_W$, which is supported on $\mathcal{C}(\delta,\mathring{\triangle})$. 
Combining these facts and applying Fubini's theorem,
\begin{equation*}
	\begin{split}
		\Vol(M(\delta,\mathring{\triangle},T_W),\omega_{\mathrm{red}}) & = \int_{\mathcal{C}(\delta,\mathring{\triangle})}d\lambda \times \mu_*\beta_W\\ 
		& = \int_{\ttt_+^*}\left(\int_{\hw\n(\xi)\cap \mathcal{C}(\delta)}d\lambda\right) \mu_*\beta_W \\
		& \geq \int_{\ttt_+^*}(\Vol(\mathcal{O}_\xi,\omega_\xi) - c(\xi)\delta) \mu_*\beta_W \\
		& \geq \int_{\ttt_+^*}\Vol(\mathcal{O}_\xi ,\omega_\xi) \mu_*\beta_W - C\delta \\
		& = \Vol(M,\omega)- C\delta.
	\end{split}	
\end{equation*}
Letting $\delta < \varepsilon /C$ completes the proof.
\end{proof}

As discussed in the introduction, Theorem \ref{thm:intro_2} has a more direct proof than Theorem \ref{main corollary 1} since there cannot be nutation effects in the case of coadjoint orbits. Nonetheless, the following proof shows how Theorem \ref{thm:intro_2} can be derived as a corollary of Theorem \ref{main corollary 1}.

\begin{proof}[Proof of Theorem \ref{thm:intro_2}]
    Let $(M,\omega,\mu)=(\mathcal{O}_\lambda,\omega_\lambda,\iota)$ be a regular coadjoint orbit of $K$, equipped with the coadjoint action of $K$ with moment map the inclusion $\iota\colon\mathcal{O}_\lambda  \to \kk^*$. Then  $(\mathcal{O}_\lambda,\omega_\lambda,\iota)$ is a multiplicity free space with principal isotropy subgroup $L=T$ and Kirwan polytope $\triangle = \{\lambda\}$. Let $(\mathcal{C}\times \T\times T,\omega_{-\infty},(\hw^{PT},\wt^{PT}))$ be one of the Hamiltonian $T\times T$-manifolds defined in Section~\ref{section; statement of main theorem}.
    
    The model space is
    \[   
        M(\delta,\mathring{\triangle},L) \cong (\mathcal{C}(\delta) \times \T \times T ,\omega_{-\infty})\sslash_{\lambda} T \cong (\triangle_\lambda(\delta)\times \T,\omega_{\rm{red}}).
    \]
    Here $\triangle_\lambda(\delta)$ is the intersection of the subspace $ (-\sqrt{-1}\psi_\hhh \circ \hw^t)(\lambda)$ with $\mathcal{C}(\delta)$.  It follows from the definition of $\omega_{-\infty}$ (Definition \ref{defintion; symplectic structure on toric space}) that the reduced symplectic form $\omega_{\rm{red}}$ equals
	\[
		\omega_{\rm{std}} := d\upsilon_1\wedge dx_1   + \dots + d\upsilon_m\wedge dx_m.
	\]
    Thus, Theorem \ref{main corollary 1} gives us the desired embeddings $(\triangle_\lambda(\delta)\times \T,\omega_{\rm{std}}) \hookrightarrow (\mathcal{O}_\lambda,\omega_\lambda)$.
\end{proof}

\begin{remark}
    By Theorem \ref{main corollary 1}, we have constructed embeddings of Hamiltonian $T$-manifolds 
    \[
        (M(\delta,\mathring{\triangle},L),\omega_{\mathrm{red}},\wt^{PT}) \hookrightarrow (M,\omega,\pr_{\ttt^*}\circ \mu)
    \]
    where one recalls that $\pr_{\ttt^*}\circ \mu$ is the moment map for the action of $T$ on $M$ as the maximal torus of $K$.  Reduction by $T$ at an arbitrary element $\xi \in \ttt^*$ yields symplectic embeddings
    \[
        M(\delta,\mathring{\triangle},L)\sslash_\xi T \hookrightarrow M\sslash_\xi T.
    \]
    In the case of coadjoint orbits, the space on the right is a symplectic analogue of a weight variety and the space on the left has the form $\triangle_{\lambda,\xi}(\delta)\times \T/T$ where $\triangle_{\lambda,\xi}(\delta)$ is the intersection of $\triangle_{\lambda}(\delta)$ with the affine subspace $ (-\sqrt{-1}\psi_\hhh \circ \wt^t)(\xi)$.  It should be straightforward to show using Theorem \ref{theorem;BKmaintheorem} that such embeddings are volume exhausting, provided the space on the right is non-empty.
\end{remark}

\section{Main Results: Applications to Gromov width}\label{gromov width section}

Given two symplectic manifolds $(N,\tau)$ and $(M,\omega)$, a  \emph{symplectic embedding} of $(N,\tau)$ into $(M,\omega)$ is an injective immersion $\varphi\colon N \hookrightarrow M$ such that $\varphi^*\omega = \tau$.  The \emph{Gromov width} of a connected symplectic manifold $(M,\omega)$ of dimension $2n$ is
\[
	\mathrm{GWidth}(M,\omega)= \sup_{r>0} \left\{ \pi r^2 \mid \exists \text{ a symplectic embedding }(B^{2n}(r),\omega_{\mathrm{std}}) \hookrightarrow (M,\omega)\right\},
\]  
where $\omega_{\mathrm{std}}$ is defined as in~\eqref{equation; omega std on r2n}.
Following earlier work by \cite{KT,GLu,Z,P}, Caviedes Castro proved the following.
	
\begin{theorem}\cite{CC}\label{GW upper  bound} Let $K$ be a compact connected simple Lie group. For all $\lambda \in  \ttt_+^*$, 
\begin{equation}\label{GW upper bound equation}
	\mathrm{GWidth}(\mathcal{O}_\lambda, \omega_\lambda) \leq \min\{ 2\pi\langle \sqrt{-1}\lambda, \alpha^\vee\rangle \mid \alpha\in R_+ \text{ and } \langle \sqrt{-1}\lambda, \alpha^\vee \rangle >0\}.
\end{equation}
\end{theorem} 

\begin{example}
Let $K = SU(2)$, let $\lambda \in \mathring{\ttt}_+^*$, and let $\alpha$ denote the positive root. In this case,
\[
    \mathrm{GWidth}(\mathcal{O}_\lambda, \omega_\lambda) = \Vol(\mathcal{O}_\lambda,\omega_\lambda) = 2\pi \frac{(\sqrt{-1}\lambda,\alpha)}{(\frac{1}{2}\alpha,\alpha)}= 2\pi\langle \sqrt{-1}\lambda,\alpha^\vee\rangle.
\]
\end{example}

Lower bounds for Gromov width of symplectic manifolds with action-angle coordinates $U\times \T^n$ can be obtained by studying the ``integral affine geometry'' of the domain $U$ as follows. The \emph{open simplex of size} 1 is
\begin{equation}
    \triangle^n := \left\{ \mathbf{x} = (x_1,\dots, x_n) \in \R^n \mid  x_1 + \dots + x_n < \frac{1}{2\pi} \text{ and } x_i > 0\ \forall i = 1, \dots , n\right\}.
\end{equation}
The open simplex of size $\ell>0$ is the scaling $\ell\triangle^n$ (i.e.~the set of points $\ell \mathbf{x}$, $\mathbf{x} \in \triangle^n$).
The group of  \emph{integral affine transformations} of $(\R^n,\Z^n)$, denoted $\mathrm{Aff}(\R^n,\Z^n)$, consists of all transformations $\Psi\colon \R^n \to \R^n$ of the form $\Psi(\mathbf{x}) = A\mathbf{x}+\mathbf{b}$, where $A \in SL_n(\Z)$  and $\mathbf{b} \in \R^n$.  Given $U, U' \subset \R^n$, an \emph{integral affine embedding} of $U$ into  $U'$ is an integral affine transformation $\Psi \in \mathrm{Aff}(\R^n,\Z^n)$ such that $\Psi(U) \subset U'$.

\begin{definition}
    The \emph{integral affine width} of a set $U\subset \R^n$ is
    \[
	    c_\triangle(U) := \sup_{\ell \geq 0} \{ \ell \mid \exists \Psi \in \mathrm{Aff}(\R^n,\Z^n) \text{ such that  } \Psi(\ell\triangle^n) \subset U\}.
    \] 
\end{definition}

Integral affine width has two important properties that are immediate from the definition.
\begin{enumerate}[(i)]
    \item (conformality) For all $\beta \in \R$ and $U \subset \R^n$, $c_\triangle(\beta U) = \vert\beta\vert\cdot  c_\triangle(U)$, where $\beta U = \{ \beta u \mid u \in U\}$. 
    \item (monotonicity) If $\Psi \in \mathrm{Aff}(\R^n,\Z^n)$ and $U,U' \subset \R^n$ such that $\Psi(U) \subset U'$, then $c_\triangle(U) \leq c_\triangle(U')$.
\end{enumerate}
Gromov width has analogous properties. Conformality  is the property that
\[
    \GWidth(M,\beta\omega) = |\beta|\GWidth(M,\omega)
\]
for all $\beta \in \R$. Monotonicity is the property that if $(N,\tau)$ embeds symplectically into $(M,\omega)$, then $\GWidth(N,\tau) \leq \GWidth(M,\omega)$. 

We note two additional properties of integral affine width. First, for all $U\subset \R^n$, $c_\triangle(U) = c_\triangle(U^{\mathrm{int}})$, where $U^{\mathrm{int}}$ denotes the interior of $U$. 
The second property is continuity. Recall the Hausdorff distance between non-empty sets $A,B \subset \R^n$ is 
\[
    d_H(A,B) := \inf_{\varepsilon\geq 0}\{ \varepsilon \mid A \subset B_\varepsilon \text{ and } B \subset A_\varepsilon\},
\]
where $A_\varepsilon := \{ \mathbf{x} \in \R^n \mid d(\mathbf{x},A) \leq \varepsilon\}$ and $d(\mathbf{x},A) := \inf\{d(\mathbf{x},\mathbf{y})\mid \mathbf{y} \in A\}$. The Hausdorff distance defines a metric on the set $\mathcal{K}_n$ of non-empty compact subsets of $\R^n$. Integral affine width is continuous as a real valued function on $(\mathcal{K}_n,d_H)$.

\begin{lemma}\label{comparison of widths}
    Let $U$ be an open subset of $ \R^n$. Then, 
    \begin{equation}\label{simplicial integral affine width lower bound}
    \GWidth(U \times \T^n,\omega_{\mathrm{std}})\geq c_\triangle(U),
    \end{equation}
    where $\omega_{\mathrm{std}}$ denotes the standard symplectic structure on $\R^n \times \T^n$ defined in~\eqref{equation; omega std}.
\end{lemma}

\begin{proof} Let $\Psi\in \mathrm{Aff}(\R^n,\Z^n)$ and $\ell>0$ such that $\Psi(\ell\triangle^n) \subset U$. The map 
\[
	\widetilde \Psi\colon (\R^n \times \T^n,\omega_{\mathrm{std}}) \to (\R^n \times \T^n,\omega_{\mathrm{std}}), \qquad \widetilde \Psi(\mathbf{x},\theta) = (\Psi(\mathbf{x}),(\Psi\n)^T(\theta))
\] 
is symplectic. Restriction of $\widetilde \Psi$ defines a symplectic embedding of $(\ell\triangle^n \times \T^n,\omega_{\mathrm{std}})$ into $(U \times \T^n,\omega_{\mathrm{std}})$. It is known that $\mathrm{GWidth}(\ell\triangle^n \times \T^n,\omega_{\mathrm{std}}) = \ell$ (see e.g.\ \cite[Proposition 2.1]{FLP} and the references therein). By monotonicity of Gromov width, it follows that $ \GWidth(U \times \T^n,\omega_{\mathrm{std}}) \geq \ell$. Taking the supremum over all $\ell$ such that there exists $\Psi\in \mathrm{Aff}(\R^n,\Z^n)$ with $\Psi(\ell\triangle^n) \subset U$  completes the proof.
\end{proof}

Assume that $K$ is a simple compact Lie group. Fix a double reduced word $\ii$ for $(w_0,e)$, and let $\theta=\zeta\circ \theta_{\overline{\sigma}(\ii)}$, as in Remark~\ref{stringremark}. For $\lambda\in \mathring{\ttt}^*_+$, let $\triangle_\lambda \subset \R^m$ denote the projection to $\R^m$ of the fiber $(\hw^{PT})\n(\lambda) \subset PT(K^*,\theta)$; see also~\eqref{coneinj}. The parameterization $\mathring{\ttt}_+^* \to \mathcal{K}_m$, $\lambda \mapsto \triangle_{\lambda}$ is continuous with respect to $d_H$. Note that for $\beta>0$, $\triangle_{\beta\lambda} = \beta \triangle_{\lambda}$.

\begin{lemma}\label{string polytope width}
	With $K$ and $\theta$ as above, for all $\lambda \in \mathring{\ttt}_+^*$, 
	\begin{equation}\label{int aff inequality}
	    c_\triangle(\triangle_{ \lambda}) \geq \min\{ 2\pi \langle \sqrt{-1}\lambda,\alpha^\vee \rangle \mid \alpha \in R_+\}.
	\end{equation}
\end{lemma}

\begin{proof}
For all $\lambda \in \mathring{\ttt}_+^*$, let $\ell_\lambda := \min \{ 2\pi \langle \sqrt{-1}\lambda,\alpha^\vee \rangle \mid \alpha \in R_+\}$. By Remark~\ref{stringremark} together with \cite[Theorem 7.2]{FLP} and \cite{MG}, there exists an integral affine embedding of $\ell_\lambda\triangle^m$ into $\triangle_{\lambda}$ for all $\lambda \in \sqrt{-1}P\cap\ttt_+^*$.  For all $\beta>0$ and $\lambda \in \sqrt{-1}P\cap\ttt_+^*$, conformality of integral affine width implies
\[
    c_\triangle(\triangle_{ \beta\cdot \lambda}) = c_\triangle(\beta\triangle_{ \lambda}) = \beta c_\triangle(\triangle_{ \lambda}) \geq \beta\ell_\lambda = \ell_{\beta \cdot\lambda}.
\]
Thus \eqref{int aff inequality} holds for all $\lambda$ in the set  
\[
    \mathring{\ttt}_{+,\Q}^* := \left\{ \beta \cdot \lambda \mid \beta >0, \, \lambda \in \mathring{\ttt}_+^* \text{ such that }\lambda \in \sqrt{-1}P\cap\ttt_+^* \right\}.
\]

Now let $\lambda \in \mathring{\ttt}_{+}^*$ arbitrary. The set $\mathring{\ttt}_{+,\Q}^*$ is dense in $\mathring{\ttt}_+^*$, so we can find a sequence $(\lambda_n)_{n=1}^\infty \subset \mathring{\ttt}_{+,\Q}^*$ converging to $\lambda \in \mathring{\ttt}_{+}^*$. Then $\triangle_{\lambda_n}$ converge to  $\triangle_{\lambda}$ with respect to $d_H$. It follows by continuity of $c_\triangle$ and $\ell_\lambda$ that 
\[
    c_\triangle(\triangle_{ \lambda}) = \lim_{n\to \infty}c_\triangle(\triangle_{ \lambda_n}) \geq \lim_{n\to \infty}\ell_{\lambda_n} = \ell_\lambda. \tag*{\qedhere}
\]
\end{proof}

\begin{proof}[Proof of Theorem \ref{gromov width theorem}]
Let $K$ be an arbitrary compact connected simple Lie group and fix $\lambda \in \mathring{\ttt}_+^*$ arbitrary. By Theorem \ref{GW upper  bound}, it remains to prove the lower bound. 

Fix any choice of double reduced word $\ii$ for $(w_0,e)$, and let $\theta$ be as above. Let $\varepsilon >0$. For $k\in \N$, consider the $1/k$-interior $\triangle_{ \lambda}(1/k)$ of $\triangle_\lambda$. Since $\triangle_{ \lambda}(1/k)$ converges to $\triangle_{ \lambda}$ with respect to $d_H$, continuity of $c_\triangle$ implies that for $k$ sufficiently large,
\begin{equation}\label{last proof eq1}
    c_\triangle(\triangle_{ \lambda}(1/k)) \geq c_\triangle(\triangle_{ \lambda}) - \varepsilon.
\end{equation}

By Theorem~\ref{thm:intro_2}, there exists a symplectic embedding $(\triangle_{ \lambda}(1/k) \times \T,\omega_{\mathrm{std}}) \hookrightarrow (\mathcal{O}_\lambda,\omega_\lambda)$. Combining this with monotonicity of Gromov width, Lemma \ref{comparison of widths}, Equation \eqref{last proof eq1}, and Lemma \ref{string polytope width},
\begin{equation}\label{last proof eq2}
    \begin{split}
        \GWidth(\mathcal{O}_\lambda,\omega_\lambda) & \geq \GWidth(\triangle_{ \lambda}(1/k)\times \T,\omega_{\mathrm{std}}) \\
        & \geq c_\triangle(\triangle_{ \lambda}(1/k)) \\
        & \geq c_\triangle(\triangle_{ \lambda}) - \varepsilon \\
        & \geq \min\{ 2\pi\langle \lambda,\alpha^\vee \rangle \mid \alpha \in R_+\}- \varepsilon.
    \end{split}
\end{equation} 
Since $\varepsilon>0$ was arbitrary, this completes the proof.
\end{proof}

\begin{remark}
    As we see, our lower bounds come from an upper semicontinuity property enjoyed by integral affine width. One might wonder if there is a more direct approach using an upper semicontinuity property of the Gromov width of regular coadjoint orbits. To our knowledge, it is an open problem whether  Gromov width has such a property. 
    See \cite{FLP}[Remark 4.1] for more details.
\end{remark}

\begin{remark}
    Note that by Theorem \ref{GW upper  bound} and  \eqref{last proof eq2},
    \[
       \min\{ 2\pi\langle \sqrt{-1}\lambda,\alpha^\vee \rangle \mid \alpha \in R_+\} \geq  c_\triangle(\triangle_{ \lambda}).
    \]
    Thus the inequality of Lemma \ref{string polytope width} is in fact an equality.
\end{remark}

\Addresses

\end{document}